\pdfoutput=1
\documentclass[a4paper,10pt]{amsart}
\usepackage{
 amssymb,
 enumerate,
 fullpage,
 mathdots,
 mathrsfs,
 mathtools,
 stmaryrd,
 svninfo,
 thmtools
}
\usepackage[l2tabu, orthodox]{nag}
\usepackage{diagrams}
\usepackage[breaklinks,pagebackref]{hyperref}

\newcommand{\QQ}{\mathbf{Q}}
\newcommand{\Qp}{\QQ_p}
\newcommand{\Ql}{\QQ_\ell}
\newcommand{\ZZ}{\mathbf{Z}}
\newcommand{\Zp}{\ZZ_p}
\newcommand{\Zl}{\ZZ_\ell}
\newcommand{\DD}{\mathbf{D}}
\newcommand{\Zhat}{\widehat{\ZZ}}
\renewcommand{\AA}{\mathbf{A}}
\newcommand{\RR}{\mathbf{R}}
\newcommand{\CC}{\mathbf{C}}
\newcommand{\QQbar}{\overline{\QQ}}
\newcommand{\Qpbar}{\overline{\QQ}_p}
\DeclareMathOperator{\Frob}{Frob}

\newcommand{\cH}{\mathcal{H}}
\newcommand{\cO}{\mathcal{O}}
\newcommand{\cS}{\mathcal{S}}
\newcommand{\cZ}{\mathcal{Z}}

\newcommand{\et}{\text{\upshape \'et}} 
\newcommand{\mot}{\mathrm{mot}}
\newcommand{\dR}{\mathrm{dR}}
\newcommand{\Iw}{\mathrm{Iw}}
\newcommand{\dd}{\mathop{}\!\mathrm{d}}

\renewcommand{\ge}{\geqslant}
\renewcommand{\le}{\leqslant}

\newcommand{\into}{\hookrightarrow}
\newcommand{\onto}{\twoheadrightarrow}

\newcommand{\sA}{\mathscr{A}}
\newcommand{\sD}{\mathscr{D}}
\newcommand{\sE}{\mathscr{E}}
\newcommand{\sF}{\mathscr{F}}
\newcommand{\sH}{\mathscr{H}}
\newcommand{\sV}{\mathscr{V}}

\newcommand{\grz}{\mathfrak{z}}
\newcommand{\grZ}{\mathfrak{Z}}
\newcommand{\f}{\mathrm{f}}

\newcommand{\uchi}{\underline{\chi}}
\newcommand{\upsi}{\underline{\psi}}
\newcommand{\uphi}{\underline{\phi}}

\newcommand{\LE}{\mathcal{LE}}
\newcommand{\EI}{\mathcal{EI}}
\newcommand{\cEI}{{}_c\EI}

\DeclareMathOperator{\br}{br}
\DeclareMathOperator{\Eis}{Eis}
\DeclareMathOperator{\GSp}{GSp}
\DeclareMathOperator{\Sp}{Sp}
\DeclareMathOperator{\SO}{SO}
\DeclareMathOperator{\GL}{GL}
\DeclareMathOperator{\SL}{SL}
\DeclareMathOperator{\Anc}{Anc}
\DeclareMathOperator{\cEis}{{}_{\mathit{c}}Eis}
\DeclareMathOperator{\Hom}{Hom}
\DeclareMathOperator{\ord}{ord}
\DeclareMathOperator{\Spec}{Spec}
\DeclareMathOperator{\Gal}{Gal}
\DeclareMathOperator{\pr}{pr}
\DeclareMathOperator{\Sym}{Sym}
\DeclareMathOperator{\mom}{mom}
\DeclareMathOperator{\ch}{ch}
\DeclareMathOperator{\vol}{Vol}

\newcommand{\Gl}{G_\ell}
\newcommand{\Hl}{H_\ell}
\newcommand{\sG}{\mathscr{G}}
\DeclareMathOperator{\Ind}{Ind}
\DeclareMathOperator{\cInd}{c-Ind}

\newenvironment{smatrix}{\left(\begin{smallmatrix}}{\end{smallmatrix}\right)}
\newcommand{\stbt}[4]{\begin{smatrix}#1 & #2 \\ #3 & #4\end{smatrix}}
\newcommand{\dfour}[4]{\begin{smatrix}#1 \\ & #2 \\ & & #3 \\ & & & #4\end{smatrix}}

\newtheorem{definition}{Definition}[subsection]
\newtheorem{theorem}[definition]{Theorem}
\newtheorem{lemma}[definition]{Lemma}
\newtheorem{proposition}[definition]{Proposition}
\newtheorem{corollary}[definition]{Corollary}

\theoremstyle{remark}
\newtheorem{notation}[definition]{Notation}
\newtheorem{assumption}[definition]{Assumption}

\declaretheorem[name=Remark,sibling=definition,qed={\lower-0.3ex\hbox{$\diamond$}}]{remark}
\declaretheorem[name=Remark,parent=section,qed={\lower-0.3ex\hbox{$\diamond$}}]{remark2}

\begin{document}
 \title{Euler systems for GSp(4)}

\renewcommand{\urladdrname}{{\itshape Open Research Contributor ID (ORCID)}}

 \author{David Loeffler}
 \address[Loeffler]{Mathematics Institute\\
 Zeeman Building, University of Warwick\\
 Coventry CV4 7AL, UK.}
 \email{d.a.loeffler@warwick.ac.uk}
 \urladdr{\href{http://orcid.org/0000-0001-9069-1877}{0000-0001-9069-1877}}

 \author{Christopher Skinner}
 \address[Skinner]{
 Department of Mathematics\\
 Princeton University\\
 Fine Hall, Washington Road \\
 Princeton, NJ 08544-1000\\
 USA}
 \email{cmcls@princeton.edu}

\author{Sarah Livia Zerbes}
 \address[Zerbes]{Department of Mathematics\\
 University College London\\
 Gower Street\\
 London WC1E 6BT, UK}
 \email{s.zerbes@ucl.ac.uk}
 \urladdr{\href{http://orcid.org/0000-0001-8650-9622}{0000-0001-8650-9622}}

\thanks{Supported by: Royal Society University Research Fellowship ``$L$-functions and Iwasawa theory'' (Loeffler); Simons Investigator Grant \#376203 from the  Simons Foundation and and NSF grant DMS-1501064 (Skinner); ERC Consolidator Grant ``Euler systems and the Birch--Swinnerton-Dyer conjecture'' (Zerbes).}

\date{\today}

\begin{abstract}
 We construct an Euler system for Galois representations associated to cohomological cuspidal automorphic representations of $\GSp_4$, using the pushforwards of Eisenstein classes for $\GL_2 \times \GL_2$.
\end{abstract}

\maketitle

\setcounter{tocdepth}{1}
\tableofcontents

\section{Introduction}

 The theory of Euler systems is one of the most powerful tools available for studying the arithmetic of global Galois representations. However, constructing Euler systems is a difficult problem, and the list of known constructions is accordingly rather short. In this paper, we construct a new example of an Euler system, for the four-dimensional Galois representations associated to cohomological cuspidal automorphic representations of $\GSp_4 / \QQ$, and apply this to studying the Selmer groups of these Galois representations. Our construction relies crucially on an unexpected relation with branching problems in smooth representation theory, which is the key input in proving the norm-compatibility relations for our Euler system classes.
 
 We construct this Euler system in the \'etale cohomology of the Shimura variety of $\GSp_4$.  The strategy that we use for this construction is also applicable to many other examples of Shimura varieties, including those associated to the groups $\operatorname{GU}(2, 1)$, $\GSp_6$, and $\GSp_4 \times \GL_2$, which will be explored in forthcoming work. 
 
 The starting-point for our construction is a family of motivic cohomology classes for Siegel threefolds, which were introduced and studied by Francesco Lemma in the papers \cite{lemma10,lemma15,lemma17}. Lemma's classes are constructed by using the subgroup $H = \GL_2 \times_{\GL_1} \GL_2$ inside $\GSp_4$. Beilinson's Eisenstein symbol gives a supply of motivic cohomology classes for the Shimura varieties attached to $H$, and pushing these forward to $\GSp_4$ gives motivic cohomology classes for the Siegel threefold. By applying the \'etale realisation map and projecting to an appropriate Hecke eigenspace, Lemma's motivic classes give rise to elements of the groups $H^1\left(\QQ, W_{\Pi}^*(-q)\right)$, where $\Pi$ is a suitable automorphic representation of $\GSp_4$, $W_{\Pi}^*$ the dual of the associated 
 $p$-adic 
 Galois representation, and $q$ is an integer in a certain range depending on the weight of $\Pi$.
 
To build an Euler system for these representations $W_{\Pi}^*$, we need to modify this construction in order to obtain classes defined over cyclotomic fields $\QQ(\zeta_m)$. These classes are required to satisfy an appropriate norm-compatibility relation as $m$ changes, and to take values in a $\Zp$-lattice in $W_{\Pi}^*$ independent of $m$. We define these classes by translating the natural embedding of $H$ in $G$ via appropriately chosen elements of $G(\AA_f)$, following a strategy that has been successfully used in several earlier Euler-system constructions \cite{LLZ,leiloefflerzerbes18}.
 
 Using the theory of $\Lambda$-adic Eisenstein classes initiated by Kings, we show that these Euler system classes can be interpolated $p$-adically as the parameters (including the Tate twist $q$) vary. This leads to a definition of a ``motivic $p$-adic $L$-function'' for $\Pi$, which is a $p$-adic measure on $\Zp^\times$ interpolating the images of the Euler system classes under the Bloch--Kato logarithm and dual-exponential maps at $p$. Assuming various technical hypotheses, we prove in \S \ref{sect:selmer} that if this motivic $L$-function is non-vanishing for a value of $q$ such that $W_{\Pi}^*(-q)$ is critical (in the sense of Deligne), then the corresponding Bloch--Kato Selmer groups $H^1_\f(\QQ, W_{\Pi}^*(-q))$ and $H^1_\f(\QQ, W_{\Pi}(1+q))$ are zero. Our motivic $p$-adic $L$-function should interpolate the critical values of the spin $L$-function of $\Pi$ (that is, we expect an ``explicit reciprocity law'', analogous those that have been proved for the Beilinson--Kato and Beilinson--Flach Euler systems). If such an explicit reciprocity law holds, then our bounds for $H^1_\f$ would give new cases of the Bloch--Kato conjecture. The construction of a spin $p$-adic $L$-function and the proof of an explicit reciprocity law are the topics of forthcoming joint work with Vincent Pilloni. 
 
 One of the chief novelties of our construction is in the proofs of the norm-compatibility relations for the Euler system classes. In place of the (exceedingly laborious) double-coset computations used in \cite{LLZ} for example, we use methods of smooth representation theory to reduce the norm-compatibility statement to a far easier, purely local statement involving Bessel models of unramified representations of $\GSp_4(\Ql)$. This reduction is possible thanks to a case of the local Gan--Gross--Prasad conjecture due to Kato, Murase and Sugano, showing that the space of $\SO_4(\Ql)$-invariant linear functionals on an irreducible spherical representation of $\SO_4(\Ql) \times \SO_5(\Ql)$ has dimension $\le 1$. This technique promises to be applicable in many other settings where local multiplicity one results of this type are known; for instance, in a forthcoming paper we shall use a similar approach to prove norm-compatibility relations in an Euler system for the Shimura variety of the unitary group $\operatorname{GU}(2, 1)$, using the local Gan--Gross--Prasad conjecture for the pair $( \operatorname{U}(2), \operatorname{U}(3) )$. 

 \subsection*{Outline of the paper}
 
  For the benefit of the reader, we give a brief outline of how our Euler system classes are constructed, and how the norm-compatibility relations for these are proved. 
  
  \subsubsection*{Construction of the elements} Let $G = \GSp_4$, and for each open compact $U \subset G(\AA_f)$, let $Y_G(U)$ be the Siegel three-fold of level $U$ (a Shimura variety for $G$). We construct a map of $G(\AA_f)$-representations, the \emph{Lemma--Eisenstein map},
  \[ \LE: I \mathop{\otimes}_{\cH(H(\AA_f))} \cH(G(\AA_f))\rTo \varinjlim_U H^4_{\mot}\left(Y_G(U), \sD \right) \]
  where $\cH(-)$ denotes the Hecke algebra, $\sD$ is a relative Chow motive (a ``motivic sheaf'') over $Y_G$ arising from some algebraic representation of $G$, and $I$ is a certain explicit representation of $H(\AA_f)$. This construction depends on parameters $a,b,q,r$, specifying weights for $G$ and for $H$, but we shall suppress this for now. The construction of $\LE$, given in \S\ref{sect:LEconstruct}, is essentially formal: the representation $I$ records the data needed to define an Eisenstein class in the motivic cohomology of $Y_H$, and $\LE$ maps this Eisenstein class to a linear combination of $G(\AA_f)$-translates of its pushforward to $Y_G$, with the $\cH(G(\AA_f))$ term recording which translations to apply. 
  
  Let $K$ be a level, unramified outside $S \cup \{p\}$ (where $S \not\ni p$ is a finite set of primes) and having a certain specific form at $p$. Then for any $n$ coprime to $S$, the base-extension $Y_G(K) \times_{\Spec \QQ} \Spec \QQ(\mu_n)$ is itself a Shimura variety for some subgroup $K' \subseteq K$. In \S\ref{sect:localdata} we use these isomorphisms, and certain explicit choices of test data as input to $\LE$, to define a collection of classes
  \[ z_{M, m} \in H^4_{\mot}\left(Y_G(K) \times \Spec \QQ(\mu_{Mp^m}), \sD \right),\]
  for $m \ge 0$ and $M$ square-free and coprime to $pS$. These are our Euler system classes.
  
  \subsubsection*{Norm-compatibility} The ``ideal'' norm-compatibility result for these classes would be an identity of the form
  \[
   \operatorname{norm}\left( z_{\ell M, m} \right) = \mathcal{P}_\ell(\sigma_\ell^{-1}) \cdot z_{M, m}
  \]
  for primes $\ell \nmid M p S$. Here ``$\operatorname{norm}$'' denotes the Galois norm map from $\QQ(\mu_{M\ell p^m})$ to $\QQ(\mu_{Mp^m})$, $\sigma_\ell$ is the Frobenius at $\ell$, and $\mathcal{P}_\ell(X)$ is a degree 4 polynomial with coefficients in the spherical Hecke algebra at $\ell$, which acts on each irreducible representation as the corresponding spin $L$-factor. However, we cannot prove the full strength of this statement here (we hope to return to this issue in a later paper). Instead, we prove a version of this result after mapping to Galois cohomology. We choose $\Pi$ a suitably nice cohomological automorphic representation of $\GSp_4$ such that $\Pi_f^K \ne 0$. (We need $\Pi_\ell$ to be generic for almost all $\ell$, which excludes certain ``endoscopic'' representations such as Saito--Kurokawa lifts.) Then $\Pi_f^* \otimes W_\Pi^*$ appears with multiplicity 1 as a direct summand of $\varinjlim_U H^3_{\et}(Y_G(U)_{\QQbar}, \sD)$, and does not contribute to cohomology outside degree 3.  Choosing a vector $\varphi \in \Pi_f$ thus gives a homomorphism of Galois representations
  \[ \Pi_f^* \otimes W_\Pi^* \rTo W_{\Pi}^*, \]
  which factors through $(\Pi_f^*)^K$ if $\varphi$ is $K$-invariant. Combining this with the Hochschild--Serre spectral sequence gives a map of vector spaces
  \[ H^4_{\et}\left(Y_G(K) \times \Spec \QQ(\mu_{Mp^m}), \sD \right) \rTo H^1(\QQ(\mu_{Mp^m}), W_{\Pi}^*).\]
  We thus obtain a collection of cohomology classes $z^{\Pi}_{M, m} \in H^1(\QQ(\mu_{Mp^m}), W_{\Pi}^*)$, depending on the choice of $\varphi$, and we shall prove the norm-compatibility relations for these instead.
  
  For simplicity, assume that $M = 1$ and $m = 0$, so we are trying to compare $z^{\Pi}_{1, 0}$ with $\operatorname{norm}(z^{\Pi}_{\ell, 0})$ (the general case can be reduced to this by twisting). We have constructed a $G(\AA_f)$-equivariant bilinear pairing
  \[ 
   \left( I \mathop{\otimes}_{\cH(H(\AA_f))} \cH(G(\AA_f)) \right) \otimes \Pi_f 
   \to 
   H^1(\QQ, W_{\Pi}^*)
  \]
  or equivalently (via Frobenius reciprocity) an $H(\AA_f)$-equivariant pairing
  \[ 
   \grz: I \otimes \Pi_f \to H^1(\QQ, W_{\Pi}^*).
  \]
  By construction the classes $z^{\Pi}_{1, 0}$ and $\operatorname{norm}(z^{\Pi}_{\ell, 0})$ are values of this pairing, at different choices of test data $v, v' \in I \otimes \Pi_f$. In most cases (away from a few small weights) the representation $I$ is a direct sum of principal series representations $\tau$, each of which factors as $\displaystyle \mathop{\bigotimes}_{\text{$w$ prime}} \tau_w$; and by construction the projections of $v$ and $v'$ to $\tau_w \otimes \Pi_w$ coincide for $w \ne \ell$. 
  
  It is at this point that the decisive input from local representation theory appears: known cases of the Gan--Gross--Prasad conjecture imply that $\Hom_{H(\Ql)}(\tau_\ell \otimes \Pi_\ell, 1_H)$ is one-dimensional, and we can construct a canonical basis $\grz_\ell$ of this space using zeta-integrals. So it suffices to show that $\grz_\ell(v_\ell) = P_\ell(1) \grz_\ell(v_\ell')$, which is a simple, purely local computation which we carry out in \S\ref{sect:repth}. It then follows that $\grz'_\ell(v_\ell) = P_\ell(1) \grz'_\ell(v_\ell')$ for \emph{every} $H$-equivariant homomorphism $\grz_\ell'$ from $\tau_\ell \otimes \Pi_\ell$ to a space with trivial $H$-action, and the desired norm relation follows (see Proposition \ref{prop:ESnorm-nonint}).
  
 \subsection*{Acknowledgements} 
 
  This paper owes its existence to a question raised by Francesco Lemma, who asked us whether the techniques used to build the Euler system of Beilinson--Flach elements for $\GL_2 \times \GL_2$ could be adapted to the setting of $\GSp_4$. We are very grateful to Francesco for this inspiration, and for several interesting conversations during the writing of the paper. We would also like to thank several others for helpful advice and comments, notably Giuseppe Ancona, Martin Dickson, Dimitar Jetchev, Jacques Tilouine, and Bin Xu; and the referee for numerous helpful corrections and comments.
  
  Substantial parts of this paper were written during a visit by the first and third authors to the Institute for Advanced Study in Princeton in the spring of 2016, and we are very grateful to the Institute for its support and hospitality.

\section{General notation}
\label{sect:groups}

 \begin{itemize}
  \item Let $J$ be the skew-symmetric $4 \times 4$ matrix over $\ZZ$ given by $\begin{smatrix} &&&1\\&&1\\&-1\\-1\end{smatrix}$. We let $G = \GSp_4$ be the group scheme over $\ZZ$ defined by
  \[ G(R) \coloneqq \GSp_4(R) = \left\{ (g, \mu) \in \GL_4(R) \times \GL_1(R) : g^t \cdot  J \cdot g = \mu J\right\} \]
  for any commutative unital ring $R$. We write $\mu: G \to \GL_1$ for the symplectic multiplier map. 
  
  \item We define the \emph{standard Borel subgroup} $B \subseteq G$ to be the subgroup $\{ (g, \mu): g$ is upper-triangular$\}$.
  
  \item We define a \emph{standard parabolic subgroup} to be a subgroup of $G$ containing $B$; there are exactly four of these, namely $B$, $G$, the \emph{Siegel parabolic} $P_\mathrm{S}$ and the \emph{Klingen parabolic} $P_{\mathrm{Kl}}$, where
  \begin{align*}
   P_\mathrm{S} &= 
   \begin{smatrix}
    * & * & * & * \\
    * & * & * & * \\
    &   & * & * \\
    &   & * & * 
   \end{smatrix}, &
   P_{\mathrm{Kl}} = 
   \begin{smatrix}
    * & * & * & * \\
    & * & * & * \\
    & * & * & * \\
    &   &   & * 
   \end{smatrix}.
  \end{align*}
  
  \item We write $T$ for the diagonal torus of $G$, which is equal to the product $A \times T'$, where $A, T'$ are the tori defined by
  \begin{align*} 
   A &= \dfour{x}{y}{x}{y}, &T' &= \dfour{x}{x}{1}{1}.
  \end{align*}
  
  \item  Let $H = \GL_2 \times_{\GL_1} \GL_2$ (fibre product over the determinant map), and let $\iota$ denote the embedding $H \into G$ given by
  \[
  \left( 
  \begin{smatrix} a  & b  \\ c  & d  \end{smatrix},
  \begin{smatrix} a' & b' \\ c' & d' \end{smatrix}
  \right) \mapsto
  \begin{smatrix}
  a &   &   & b \\
  & a'& b'&   \\
  & c'& d' &   \\
  c &   &   & d
  \end{smatrix}.
  \]
  We write $B_H = \iota^{-1}(B) = \iota^{-1}(P_\mathrm{S})$ for the standard Borel subgroup of $H$.
 \end{itemize}

 \begin{remark2}
  The quotient of $G$ by its centre $Z_G$ is the split form of the orthogonal group $\SO_5$. We have $Z_G \subset \iota(H)$, and the image of $H / Z_G$ in $G/Z_G$ via $\iota$ is the split form of $\SO_4$, embedded as the stabiliser of an anisotropic vector in the defining 5-dimensional representation. This will be used in \S \ref{sect:GGP} below, in order to make use of the results of the Gan--Gross--Prasad theory of restriction of representations of $\SO_5$ to $\SO_4$. 
 \end{remark2}
  
\section{Preliminaries I: Local representation theory}
\label{sect:repth}
 
 In this section, we fix a prime $\ell$ and collect some definitions and results regarding smooth representations of the groups $\GL_2(\Ql)$, $G(\Ql)$, and $H(\Ql)$ on complex vector spaces. For brevity we shall write $\Gl = G(\Ql)$ and similarly $\Hl$.
 
 \subsection{Principal series of $\GL_2(\Ql)$}

  \begin{notation}
   We write $\dd x$ and $\dd^\times x$ for the Haar measures on $\Ql$ and $\Ql^\times$ normalised such that $\Zl$ (resp.~$\Zl^\times$) has volume 1. The norm $|\cdot|$ is normalised such that $|\ell| = 1/\ell$. If $\chi$ is a smooth character of $\Ql^\times$, we write $L(\chi, s)$ for its local $L$-factor, which is
   \[ L(\chi, s) = L(\chi |\cdot|^s, 0) = \begin{cases} (1 - \chi(\ell) \ell^{-s})^{-1} & \text{if $\chi |_{\Zl^\times} = 1$,}\\
   1 & \text{otherwise.}\end{cases}
   \]
  \end{notation}
  
  \begin{definition}
   Given two smooth characters $\chi$ and $\psi$ of $\Ql^\times$, we let $I(\chi,\psi)$ be the space of smooth functions $f:\GL_2(\Ql)\rightarrow \CC$ such that
   \[
    f(\begin{smatrix} a & b \\ 0 & d\end{smatrix} g )  = \chi(a)\psi(d)|a/d|^{1/2}f(g),
   \]
   equipped with a $\GL_2(\Ql)$-action via right translation of the argument. 
  \end{definition}
  
  As is well known, the pairing $I(\chi, \psi) \times I(\chi^{-1}, \psi^{-1}) \to \CC$ defined by
  \[ \langle f_1, f_2 \rangle = \int_{\GL_2(\Zl)} f_1(g) f_2(g) \dd g, \]
  where we normalise the measure such that $\GL_2(\Zl)$ has volume 1, identifies $I(\chi^{-1}, \psi^{-1})$ with the dual of $I(\chi, \psi)$. Moreover, if $\chi / \psi \neq |\cdot|^{\pm 1}$, then $I(\chi,\psi)$ is an irreducible representation.
  
  We will frequently need to use analytic continuation in an auxiliary parameter $s$. The following construction will be helpful:  
 
  \begin{definition}
   A \emph{polynomial section} of the family of representations $I(\chi |\cdot|^s, \psi |\cdot|^{-s})$ is a function on $\GL_2(\Ql) \times \CC$, $(g, s) \mapsto f_s(g)$, such that $g \mapsto f_s(g)$ is in $I(\chi |\cdot|^s, \psi |\cdot|^{-s})$ for each $s \in \CC$, and $s \mapsto f_s(g)$ lies in $\CC[\ell^s, \ell^{-s}]$ for every $g \in \GL_2(\Ql)$. A section is \emph{flat} if its restriction to $\GL_2(\Zl)$ is independent of $s$.
  \end{definition}
  
  From the Iwasawa decomposition, one sees that every $f \in I(\chi, \psi)$ extends to a unique flat section. The space of polynomial sections is stable under the action of $\GL_2(\Ql)$ (while the space of flat sections clearly is not).
  
  \begin{definition}
   Let $M:I(\chi,\psi) \rightarrow I(\psi,\chi)$ be the 
   normalised  
   standard intertwining operator, defined by analytic continuation to $s = 0$ of the integral
   \[ 
    M(f_s; g) \coloneqq L(\chi/\psi, 2s)^{-1} \int_{\Ql} f_s\left(w \stbt{1}{n}{0}{1} g\right) \dd n 
   \]
   where $w=\stbt{0}{1}{-1}{0}$ is the long Weyl element.
  \end{definition}
  
  More precisely, if $|\chi/\psi| = |\cdot|^r$, and $(f_s)_{s \in \CC}$ is any polynomial section, then the intertwining integral for $M(f_s; g)$ is absolutely convergent for $r + 2 \Re(s) > 0$ and defines a polynomial section of $I(\psi |\cdot|^{-s}, \chi |\cdot|^s)$ (see e.g.~\cite[Proposition 4.5.7]{bump97} for further details). The specialisation of this section at $s = 0$ depends only on $f_0 \in I(\chi, \psi)$, not on the choice of section passing through $f_0$, and this defines a non-zero intertwiner between $I(\chi, \psi)$ and $I(\psi, \chi)$ (even in the exceptional case $\chi = \psi$).
  
  \begin{proposition}
   \label{prop:Madjoint}
   Suppose $\chi$ and $\psi$ are unramified. Then for all $f_1 \in I(\chi, \psi)$ and $f_2 \in I(\psi^{-1}, \chi^{-1})$ we have
   \[ \langle M(f_1), f_2 \rangle = \langle f_1, M(f_2) \rangle. \]
  \end{proposition}
 
  \begin{proof}
   By choosing polynomial sections passing through the $f_i$, we may assume without loss of generality that $\chi/\psi \ne |\cdot|^{\pm 1}$, so that both $I(\chi, \psi)$ and $I(\psi^{-1}, \chi^{-1})$ are irreducible. Hence it suffices to check the equality when $f_1$ and $f_2$ are the respective spherical vectors (normalised such that $f_i(1) = 1$). With our conventions, $M$ sends the normalised spherical vector of $I(\chi, \psi)$ to the normalised spherical vector of $I(\psi, \chi)$, and these normalised vectors pair to 1 under the duality pairing.
  \end{proof}

 \subsection{Siegel sections}
 
  \begin{notation}
   Let $\cS(\Ql^2; \CC)$ denote the space of Schwartz functions (locally-constant, compactly-supported functions) on $\Ql^2$. We let $\GL_2(\Ql)$ act on this space from the left by $(g \cdot \phi)(x, y) = \phi( (x, y) \cdot g)$ for $g \in \GL_2(\Ql)$ and $\phi \in \cS(\Ql^2; \CC)$. For $\phi \in \cS(\Ql^2, \CC)$, we define its Fourier transform $\hat \phi$ by
   \[ 
    \hat \phi(x, y) = \iint e_\ell(xv-yu) \phi(u, v)\dd u\dd v,
   \]
   where $e_\ell(x)$ is the standard additive character of $\Ql$, mapping $1/\ell^n$ to $\exp(2\pi i / \ell^n)$.
  \end{notation}
  
  \begin{proposition}
   Let $\phi \in \cS(\Ql^2, \CC)$, and let $\chi$, $\psi$ be characters of $\Ql^\times$ with $|\chi/\psi| = |\cdot|^r$. Then the integral
   \[ 
    f_{\phi, \chi, \psi}(g, s) \coloneqq \frac{\chi(\det g) |\det g|^{s + 1/2}}{L(\chi/\psi, 2s+1)} \int_{\Ql^\times} \phi\big( (0, x) g\big) (\chi/\psi)(x) |x|^{2s+1}\dd^\times x
   \]
   converges for $r + 2 \Re(s) > -1$, and defines a polynomial section of $I(\chi|\cdot|^s, \psi|\cdot|^{-s})$, so $f_{\phi, \chi, \psi}(g) \coloneqq f_{\phi, \chi, \psi}(g, 0) \in I(\chi, \psi)$ is well-defined. These elements satisfy
   \begin{equation}
    \label{zeta-section-eq}
    \begin{aligned}
     f_{g \cdot \phi, \chi, \psi}(h) &= \chi(\det g)^{-1} |\det g|^{-1/2} f_{\phi, \chi, \psi}(hg),\\
     f_{\widehat{g \cdot \phi}, \chi, \psi}(h) &= \psi(\det g)^{-1} |\det g|^{-1/2} f_{\hat\phi, \chi, \psi}(hg)
    \end{aligned}
   \end{equation}
   for all $g, h \in \GL_2(\Ql)$.
  \end{proposition}
  
  \begin{proof}
   The convergence of the integral, and its analytic continuation as a function of $s$, form part of Tate's theory of local zeta integrals for $\GL_1$. The fact that $f_{\phi, \chi, \psi}(-, s)$ lies in $I(\chi|\cdot|^s, \psi|\cdot|^{-s})$ is immediate from the definition in the region of convergence of the integral, and follows for all $s$ by analytic continuation. The first of the transformation formulae \eqref{zeta-section-eq} is obvious from the definition, and the second follows from the identity  $\widehat{(g \cdot \phi)} = \frac{1}{|\det g|} ({}^\iota g \cdot \hat \phi)$, where ${}^\iota g = (\det g)^{-1} g$.
  \end{proof}
  
  \begin{proposition}
   \label{prop:SSFE}
   We have
   \[ 
    M(f_{\phi, \chi, \psi}) = \frac{\varepsilon(\psi/\chi)}{L(\chi/\psi, 1)} f_{\hat\phi, \psi, \chi},
   \]
   where $\varepsilon(\psi/\chi)$ is the local $\varepsilon$-factor (a non-zero scalar, equal to 1 if $\chi/\psi$ is unramified). 
  \end{proposition}

  \begin{proof}
   This is a straightforward consequence of the functional equation for Tate's $\GL_1$ zeta integral.
  \end{proof}
%
%
  If $\chi/\psi = |\cdot|^{-1}$ we interpret the right-hand side as 0, so the elements $f_{\phi, \chi, \psi}$ all land in the 1-dimensional subrepresentation. Let us evaluate these integrals explicitly for some specific choices of $\phi$, assuming now that $\chi$ and $\psi$ are \emph{unramified} characters.
  
  \begin{definition}
   We define functions $\phi_t \in \cS(\Ql^2, \CC)$, for integers $t \ge 0$, as follows.
   \begin{itemize}
    \item For $t = 0$, we let $\phi_0 \coloneqq \ch(\Zl \times \Zl)$.
    \item For $t > 0$, we let $\phi_t \coloneqq \ch(\ell^t\Zl \times \Zl^\times)$.
   \end{itemize}
  \end{definition}
  
  Note that $\phi_t$ is preserved by the action of the group $K_0(\ell^t) = \left\{\stbt{a}{b}{c}{d} \in \GL_2(\Zl): c = 0 \bmod \ell^t\right\}$.
  
  \begin{lemma}\label{GL2-lem}
   We have
   \[ f_{\phi_t, \chi, \psi}(1) = 
    \begin{cases}
     1 & \text{if $t = 0$,}\\
     L(\chi/\psi, 1)^{-1} & \text{if $t > 0$.}
    \end{cases}
   \]
   Moreover, the function $f_{\phi_t, \chi, \psi}$ is supported on $B(\Ql) K_0(\ell^t)$.
  \end{lemma}
  
  \begin{proof}
   The computation of the value at the identity is immediate. The assertion regarding the support of the function is vacuous for $t = 0$, and for $t > 1$ we have
   \[ \phi_t = \stbt{\ell^{1-t}}{0}{0}{1} \phi_1, \]
   so in fact it suffices to prove the assertion for $t = 1$; in this case, we simply observe that the function $f_{\phi_1, \chi, \psi}$ vanishes on the long Weyl element $w = \stbt{0}{1}{-1}{0}$, since $\phi_t\big( (0, x) w\big) = \phi_t(-x, 0) = 0$ for all $x$.
  \end{proof}
 
 \subsection{Notation: subgroups of $\Gl$ and $\Hl$}
  \label{sect:compacta}
  We now define an assortment of open compact subgroups of $\Gl$ and $\Hl$. We represent elements of $\Gl$ in block form $\stbt{A}{B}{C}{D}$, where $A, B, C, D$ are $2\times 2$ matrices.
  \begin{itemize}
   \item $K_{\Gl} = G(\Zl)$.\smallskip
   \item $K_{\Gl, 0}(\ell^n) = \left\{g \in G(\Zl):g = \stbt{*}{*}{0}{*} \bmod \ell^n\right\}$, for $n \ge 0$.\smallskip
   \item $K_{\Gl, 1}(\ell^n) = \left\{ g \in G(\Zl):g = \stbt{*}{*}{0}{1} \bmod \ell^n\right\}$, for $n \ge 0$.\smallskip
   \item $K_{\Gl}(\ell^m, \ell^n) = \left\{ g \in K_{\Gl, 1}(\ell^n): \mu(g) = 1 \bmod \ell^m\right\}$, for $m, n \ge 0$.\smallskip
   \item $K_{\Gl}'(\ell^m, \ell^n) =  \left\{ g \in K_{\Gl, 1}(\ell^n): g = 1 \bmod \ell^m\right\}$, for $n \ge m \ge 0$.
  \end{itemize}
  (The subgroups $K_{\Gl}'(\ell^m, \ell^n)$ will not be used until \S\ref{sect:localdata} below.) We define subgroups $K_{\Hl}$, $K_{\Hl, 0}(\ell^n)$, etc of $\Hl$ as the preimages (via $\iota$) of the corresponding subgroups of $\Gl$. We write $\dd g$ and $\dd h$ for the Haar measures on $\Gl$ and $\Hl$ normalised such that $K_{\Gl}$ (resp.~$K_{\Hl}$) has volume 1.
  
 \subsection{Induced representations of $\Hl$}
  \label{sect:Hindrep}
  
  Given two pairs of characters $\uchi=(\chi_1,\chi_2)$ and $\upsi = (\psi_1,\psi_2)$, we define $I_H(\uchi, \upsi)$ as the representation of $\Hl$ given by the normalised induction from $B_H(\Ql)$ of the character
  \[
   \dfour{a}{a'}{b'}{b}
   \mapsto \chi_1(a)\psi_1(b)\chi_2(a')\psi_2(b') \quad (\text{where $ab = a'b'$}).
  \]
  Since $\Hl$ acts transitively on $\mathbf{P}^1(\Ql) \times \mathbf{P}^1(\Ql)$, restriction of functions from $\GL_2(\Ql) \times \GL_2(\Ql)$ to $\Hl$ defines an isomorphism of $\Hl$-representations $I(\chi_1, \psi_1) \otimes I(\chi_2, \psi_2) \to I_H(\uchi, \upsi)$, where $\Hl$ acts on the source via its inclusion in $\GL_2(\Ql) \times \GL_2(\Ql)$. In particular, there is an intertwining operator $M: I_H(\uchi, \upsi) \to I_H(\upsi, \uchi)$ given by the tensor product of the two $\GL_2$ interwtining operators.
  
  \begin{proposition}
   If there is no quadratic character $\eta$ such that $\chi_1 / \psi_1 = \chi_2/\psi_2 = \eta$, then every irreducible subquotient of $I(\chi_1,\psi_1) \otimes I(\chi_2, \psi_2)$ as a representation of $\GL_2(\Ql) \times \GL_2(\Ql)$ remains irreducible as a representation of $\Hl$.
  \end{proposition}
  
  \begin{proof}
   This is an instance of \cite[Lemma 2.1]{gelbartknapp82}.
  \end{proof}

  Given $\uphi = \sum \phi_{i,1}\otimes\phi_{i,2} \in \cS(\Ql^2, \CC)^{\otimes 2}$, we let $f_{\uphi,\uchi,\upsi}$ be the image in $I_H(\uchi,\upsi)$ of the element
  \[
   \sum_i f_{\phi_{i,1},\chi_1,\psi_1} \otimes f_{\phi_{i,2},\chi_2,\psi_2} \in I(\chi_1,\psi_1)\otimes I(\chi_2,\psi_2).
  \]
  For a non-negative integer $t$, let $\uphi_t = \phi_t\otimes\phi_t$.

  \begin{proposition}
   \label{Mf-H-lem} Let $t \ge 0$. Then:
   \begin{itemize}
    \item[(a)] We have 
    \[ f_{\uphi_t,\uchi,\upsi}(1) = 
     \begin{cases}
      1 & \text{if $t = 0$,}\\
      L(\chi_1/\psi_1,1)^{-1}L(\chi_2/\psi_2,1)^{-1} & \text{if $t \ge 1$.}
     \end{cases}
    \]
    \item[(b)] $f_{\uphi_t,\uchi,\upsi}$ is supported on $B_H(\Ql)K_{\Hl,0}(\ell^t)$.
   \end{itemize}
  \end{proposition}
 
  \begin{proof}
   Follows from Lemma \ref{GL2-lem}.
  \end{proof}

 \subsection{Representations of $\Gl$}
  
  \subsubsection{Principal series representations of $\Gl$}
  
   We follow the notations of \cite{robertsschmidt07} for representations of $\Gl$. See \emph{op.cit}.~for further details (in particular \S 2.2 and the tables in Appendix A). 
   
   \begin{definition}
    Let $\chi_1, \chi_2, \rho$ be smooth characters of $\Ql^\times$ such that 
    \begin{equation}
     \label{eq:irredcond}
     |\cdot|^{\pm 1} \notin \{ \chi_1, \chi_2, \chi_1 \chi_2, \chi_1 /\chi_2\}.
    \end{equation}
    We let $\chi_1 \times \chi_2 \rtimes \rho$ denote the representation of $\Gl$ afforded by the space of smooth functions $f: \Gl \to \CC$ satisfying
    \[ 
     f\left(\begin{smatrix}
      a & * &    *    & *      \\
        & b &    *    & *      \\
        &   & cb^{-1} & *      \\
        &   &         & ca^{-1}
     \end{smatrix} g \right)
     = \frac{|a^2 b|}{|c|^{3/2}} \chi_1(a) \chi_2(b) \rho(c) f(g), 
    \]
    with $\Gl$ acting by right translation. We refer to such representations as \emph{irreducible principal series}.
   \end{definition}
   
   This representation has central character $\chi_1 \chi_2 \rho^2$; the condition \eqref{eq:irredcond} implies that it is irreducible and generic. If $\eta$ is a smooth character of $\Ql^\times$, then twisting $\chi_1 \times \chi_2 \rtimes \rho$ by $\eta$ (regarded as a character of $\Gl$ via the multiplier map) gives the representation $\chi_1 \times \chi_2 \rtimes \rho\eta$.
   
   \begin{definition}
    \label{def:LfactorG}
    Let $\sigma = \chi_1 \times \chi_2 \rtimes \rho$ be an irreducible principal series representation. The local (spin) $L$-factor of $\sigma$ is the function
    \[ L(\sigma, s) = L(\sigma \otimes |\cdot|^s, 0) = L(\rho, s) L(\rho \chi_1, s) L(\rho \chi_2, s) L(\rho \chi_1 \chi_2, s).\]
   \end{definition}
     
   \begin{proposition}
    If $\sigma = \chi_1 \times \chi_2 \rtimes \rho$ is an irreducible principal series representation, then $\sigma$ is unramified if and only if all three characters $\chi_1, \chi_2, \rho$ are unramified. Moreover, every irreducible, generic, unramified representation of $\Gl$ is isomorphic to $\chi_1 \times \chi_2 \rtimes \rho$, for a unique Weyl-group orbit of unramified characters $(\chi_1, \chi_2, \rho)$ satisfying \eqref{eq:irredcond}.
   \end{proposition}
  
   \begin{proof}
    See \cite[\S 2.2]{robertsschmidt07}.
   \end{proof}

  \subsubsection{Hecke operators}
   \label{sect:heckeops}

   Firstly, we consider the action of the spherical Hecke algebra $\cH(K_{\Gl} \backslash \Gl / K_{\Gl})$ on $\sigma^{K_{\Gl}}$ when $\sigma$ is an unramified principal series representation.
   
   \begin{lemma}
    \label{lemma:Tleigenvalues}
    Consider the following elements in the Hecke algebra $\cH(K \backslash \Gl / K)$:
    \begin{align*}
     T(\ell) &= K_{\Gl} \dfour{1}{1}{\ell}{\ell} K_{\Gl}, &     T_1(\ell^2) &= K_{\Gl} \dfour{1}{\ell}{\ell}{\ell^2} K_{\Gl},
     &
     R(\ell) &= K_{\Gl} \dfour{\ell}{\ell}{\ell}{\ell} K_{\Gl}.
    \end{align*}
    If $\mathcal{P}_\ell(X)$ is the polynomial over $\cH(K_{\Gl} \backslash \Gl / K_{\Gl})$ defined by
    \[ 1 - T(\ell) X + \ell(T_1(\ell^2) + (\ell^2 + 1)R(\ell))X^2 - \ell^3 T(\ell) R(\ell) X^3 + \ell^6 R(\ell)^2 X^4, \]
    then for any unramified principal series $\sigma = \chi_1 \times \chi_2 \rtimes \rho$, $\mathcal{P}_\ell(\ell^{-s})$ acts on $\sigma^{K_{\Gl}}$ as $L(\sigma, s - 3/2)^{-1}$.
   \end{lemma}
   
   \begin{proof}
    See \cite[\S 2.4]{taylor-thesis}; our $\mathcal{P}_\ell(X)$ is $X^4 Q_\ell(1/X)$ in Taylor's notation.
   \end{proof}
   
   Secondly, we consider the larger space of invariants under the Siegel parahoric subgroup $K_{\Gl, 0}(\ell)$. We let $U(\ell)$ denote the Hecke operator $\tfrac{1}{\vol K_{\Gl,0}(\ell)} \ch\left(K_{\Gl,0}(\ell) \dfour{\ell}{\ell}{1}{1}  K_{\Gl,0}(\ell)\right)$, which acts on $\sigma^{K_{\Gl,0}(\ell)}$ via
   \[ x \mapsto \sum_{u,v,w \in \ZZ / \ell} \begin{smatrix} \ell & & u & v \\ & \ell & w& u \\ & & 1 \\&&&1\end{smatrix} x.\]
   
   \begin{proposition}
    \label{prop:Uleigenvalues}
    If $\sigma$ is an unramified principal series, then $\sigma^{K_{\Gl,0}(\ell)}$ is 4-dimensional, and we have
    \[ \det\left(1 - U(\ell) \ell^{-s}\ \middle|\ \sigma^{K_{\Gl,0}(\ell)}\right) = L(\sigma, s-3/2)^{-1}. \]
   \end{proposition}
  
   \begin{proof}
    See \cite[Lemma 2.4]{taylor-thesis}.
   \end{proof}
 
 \subsection{Zeta integrals for $G$}
 
  In this section we isolate the key local zeta integral calculations used in our proofs of the tame norm relations. 
  
  \subsubsection{The Bessel model}
  
   \begin{definition}
    Let $A$ be the torus $\left\{ \dfour{x}{y}{x}{y} : x,y \in \mathbf{G}_m\right\}$. The \emph{Bessel subgroup} $R$ of $G$ is the semidirect product $A \ltimes N_\mathrm{S}$, where $N_\mathrm{S}$ is the unipotent radical of the Siegel parabolic $P_\mathrm{S}$.
   \end{definition}
   
   \begin{definition}
    Let $\lambda$ be a character of $A$. A (split) \emph{$\lambda$-Bessel functional} on a representation $\sigma$ of $\Gl$ is a linear functional $\mu: \sigma \to \CC$ transforming under left-translation by $R(\Ql)$ via the formula
    \begin{equation}
     \label{eq:besselmodel}
     \mu\left( \begin{smatrix} 1 & & u & v \\ & 1 & w & u \\ & & 1 & \\ & & & 1\end{smatrix}a \cdot \varphi\right) = e_\ell(u) \lambda(a) \mu(\varphi)
    \end{equation}
    for all $\varphi \in \sigma$, $a \in A$, and $u, v, w \in \Ql$.
   \end{definition}
   
   If $\sigma$ is irreducible, then the space of $\lambda$-Bessel functionals on $\sigma$ has dimension $\le 1$, by \cite[Theorem 6.3.2]{robertsschmidt16}. It is clearly zero unless $\lambda |_{Z(\Gl)}$ coincides with the central character of $\sigma$.
   
   \begin{theorem}[Roberts--Schmidt]
    If $\sigma$ is an irreducible generic representation of $\Gl$ (such as an irreducible principal series representation), then $\sigma$ admits a non-zero $\lambda$-Bessel functional $\mu_\lambda$ for \emph{every} $\lambda$ whose restriction to $Z(\Gl)$ agrees with the central character of $\sigma$. If both $\sigma$ and $\lambda$ are unramified, then we may normalise $\mu_{\lambda}$ such that $\mu_\lambda(\varphi_0) = 1$, where $\varphi_0$ is the spherical vector of $\sigma$.
   \end{theorem}
   
   \begin{proof}
    For the existence of the Bessel functional see \cite[Proposition 3.4.2]{robertsschmidt16}. It is shown in \emph{op.cit.} that the Bessel functional can be explicitly constructed by integrating functions in the Whittaker model of $\sigma$; and the assertion that in the unramified case the spherical vector maps to 1 under this functional follows, for example, from the computations of \cite[\S 7.1]{robertsschmidt07}.
   \end{proof}
   
   If $\sigma$ is any irreducible representation admitting some $\lambda$-Bessel functional $\mu_\lambda$, then for any $\varphi \in \sigma$ we may define a function $B_{\varphi, \lambda}$ on $\Gl$ by $B_{\varphi, \lambda}(g) = \mu_\lambda(g \cdot \varphi)$. The space of functions $\{B_{\varphi, \lambda}: \varphi \in \sigma\}$ is the \emph{$\lambda$-Bessel model} of $\sigma$.
   
   \begin{proposition}
    \label{prop:UlBessel}
    Let $\sigma$ be an irreducible representation of $\Gl$ admitting a $\lambda$-Bessel model, and let $\varphi \in \sigma$ be invariant under $N_{\mathrm{S}}(\Zl)$. Let us define
    \[ U(\ell^k) \varphi \coloneqq \sum_{u,v,w \in \ZZ / \ell^k} \begin{smatrix} 1 & & u & v \\ & 1 & w & u \\ & & 1 & \\ & & & 1\end{smatrix}\dfour{\ell^k}{\ell^k}{1}{1} \varphi.
    \]
    Then we have
    \[ 
     B_{U(\ell^k) \varphi, \lambda}(\dfour{x}{x}{1}{1}) = 
     \begin{cases}
      0 & \text{if $|x| > 1$, }\\
      \ell^{3k} B_{\varphi, \lambda}(\dfour{\ell^k x}{\ell^k x}{1}{1}) & \text{if $|x| \le 1$.}
     \end{cases}
    \]
   \end{proposition}
   
   \begin{proof}
    We first note that the assumption that the Bessel function $B_{\varphi, \lambda}$ is fixed by right-translation by $N_{\mathrm{S}}(\Zl)$, and transforms on the left via \eqref{eq:besselmodel}, implies that $B_{\varphi, \lambda}(\dfour{x}{x}{1}{1})$ is zero if $|x| > 1$. We now compute that
    \begin{align*}
     B_{U(\ell^k) \varphi, \lambda}(\dfour{x}{x}{1}{1}) 
     &= \sum_{u, v, w} B_{\varphi, \lambda}\left( \dfour{x}{x}{1}{1} 
     \begin{smatrix} 1 & & u & v \\ & 1 & w & u \\ & & 1 & \\ & & & 1\end{smatrix}\dfour{\ell^k}{\ell^k}{1}{1}\right) \\
     &= \ell^{2k} \sum_{u \bmod \ell^k} e_\ell(xu) B_{\varphi, \lambda}(\dfour{\ell^k x}{\ell^k x}{1}{1}).
    \end{align*}
    If $|x| > \ell^k$ then the Bessel function is zero; and if $\ell^k \ge |x| > 1$, then the sum of the $e_\ell$ terms vanishes. This leaves the cases $|x| \le 1$, in which case the terms $e_\ell(xv)$ are all equal to 1 and we obtain the result.
   \end{proof}

  \subsubsection{Novodvorsky's integral} In order to construct an intertwining operator between $\sigma$ and a principal-series $H$-representation, we shall use an integral involving a choice of vector in the Bessel model of $\sigma$, for some choice of character $\lambda$ as above. For $\varphi \in \sigma$, $\eta$ an unramified character of $\Ql^\times$, and $s \in \CC$, we define
  \[ 
   Z(\varphi, \eta, \lambda, s) \coloneqq L(\sigma \otimes \eta, s)^{-1}
  \int_{\Ql^\times} B_{\varphi, \lambda}(\begin{smatrix} x & & & \\ & x & & \\ & & 1 & \\ & & & 1\end{smatrix}) \eta(x) |x|^{s-3/2} d^\times x.
  \]
  Here $L(\sigma \otimes \eta, s)$ is the spin $L$-factor of $\sigma \otimes \eta$, as in Definition \ref{def:LfactorG}.
  
  \begin{remark}
   This integral apparently first appears in \cite[Equation 2.7]{novodvorsky79}. In an earlier draft of the present paper, we mistakenly ascribed this construction to Sugano; in fact Sugano's paper \cite{sugano85} considers a related but slightly different integral -- see Remark \ref{rmk:PS} below.
  \end{remark}
   
   \begin{proposition}
    \label{prop:novintegral}
    Suppose $\sigma$ is an irreducible unramified principal series representation, with central character $\chi_\sigma$, and let $\eta$ be an unramified character. Let $\lambda$ be given by $\dfour{x}{y}{x}{y} \mapsto \lambda_1(x) \lambda_2(y)$ for unramified characters $\lambda_1, \lambda_2$ of $\Ql^\times$.
    \begin{enumerate}[(a)]
     
     \item The integral defining $Z(\varphi, \eta, \lambda, s)$ is absolutely convergent for $\Re(s) \gg 0$, and it has analytic continuation to all $s \in \CC$ as an element of $\CC[\ell^s, \ell^{-s}]$.
     
     \item If $\varphi_0$ is the spherical vector (normalised so that $B_{\varphi_0, \lambda}(1) = 1$) then we have
     \[ 
      Z(\varphi_0, \eta, \lambda, s) = 
      \left[ L(\lambda_1 \eta,s + \tfrac{1}{2})L(\lambda_2 \eta, s + \tfrac12)\right]^{-1}.
     \]
     
     \item We have
     \[ Z\left( \begin{smatrix} ta & & & v \\ &tb & w \\ && a \\ &&& b \end{smatrix} \varphi, \eta, \lambda, s\right) = \tfrac{\lambda_1(a) \lambda_2(b)}{\eta(t) |t|^{s-3/2}} Z(\varphi, \eta, \lambda, s)
     \]
     for any $v, w \in \Ql$ and $a,b,t \in \Ql^\times$.
     
    \end{enumerate}
   \end{proposition}
   
   \begin{proof}
    Replacing $\sigma$ with $\sigma \otimes \eta$, and $(\lambda_1, \lambda_2)$ with $(\eta\lambda_1, \eta\lambda_2)$, we may assume $\eta$ is trivial. It suffices to prove (a) under the assumption that $\varphi = g \cdot \varphi_0$ for some $g \in \Gl$ (since these vectors span $\sigma$). The validity of (a) for this vector will only depend on the class of $g$ in the double coset space $R(\Ql) \backslash \Gl / K_{\Gl}$. A set of coset representatives for this double quotient, and a formula for the values of $B_{\varphi_0, \lambda}$ on these representatives, is given in \cite[Proposition 2-5]{sugano85}; see also \cite[Corollary 1.9]{bumpfriedbergfurusawa97} for an alternative, slightly more concrete formulation. The result now follows by an explicit calculation, which also gives (b) as a special case (compare also \cite[\S 3.2]{pitaleschmidt09}).
    
    Finally, part (c) is obvious from the integral formula if $\Re(s) \gg 0$, and follows for all $s$ by analytic continuation.
   \end{proof}
   
   From Proposition \ref{prop:UlBessel} above, we see that
   \begin{equation}
    \label{eq:ulintegral}
    Z(U(\ell)\varphi_0, \eta, \lambda, s) = \frac{\ell^{s + 3/2}}{\eta(\ell)}\left[ L(\lambda_1 \eta,s + \tfrac{1}{2})^{-1}L(\lambda_2 \eta, s + \tfrac12)^{-1} - L(\sigma \otimes \eta, s)^{-1} \right].
   \end{equation}
   This formula will be fundamental to the proof of our Euler system norm relations later in the paper.
   
   \begin{remark}
    It is not always true that the ideal of $\CC[q^s, q^{-s}]$ given by $\{Z(\varphi, \eta, \lambda, s): \varphi \in \sigma\}$ is the unit ideal. A sufficient condition is that $L(\lambda_1 \eta, s+\tfrac{1}{2}) L(\lambda_2\eta, s+\tfrac{1}{2})$ and $L(\sigma \otimes \eta, s)$  should have no poles in common, since then at least one of $Z(\varphi_0, \eta, \lambda, s)$ and $Z(U(\ell)\varphi_0, \eta, \lambda, s)$ is non-vanishing for every $s$. In fact this condition is also necessary, as shown by the computations of \cite{roesnerweissauer17}, although we do not need this here.
   \end{remark}

 \subsection{A local bilinear form}
  \label{sect:GGP}
   
  As in the preceding section, let $\sigma$ be an irreducible unramified principal series representation of $\Gl$, with central character $\chi_\sigma$. Let $\uchi = (\chi_1,\chi_2)$ and $\upsi = (\psi_1, \psi_2)$ be pairs of unramified characters of $\Ql^\times$, satisfying
  \[ 
   \chi_1 \chi_2 \cdot \psi_1 \psi_2 \cdot \chi_\sigma = 1,
  \]
  and suppose that neither $\chi_1/\psi_1$ nor $\chi_2/\psi_2$ is quadratic or equal to $|\cdot|^{-1}$ (but we do allow either or both to equal $|\cdot|$). In the notation of the above section, we define a character $\lambda$ of $A$ by $\lambda_1 = (\psi_1 \chi_2)^{-1}$, $\lambda_2 = (\chi_1 \psi_2)^{-1}$; and we take for $\eta$ the character $\psi_1 \psi_2$. For brevity, we write $\uchi_s$ for the pair $(\chi_1 |\cdot|^{-s},\chi_2 |\cdot|^{-s})$, and similarly $\upsi_s = (\psi_1 |\cdot|^s,\psi_2 |\cdot|^s)$
 
  \begin{proposition}
   \label{prop:ulintegral}
   Mapping $\varphi \in \sigma$ to the function $z_s(\varphi)$ on $\Hl$ defined by
   \[ 
    z_s(\varphi)(h) = Z(h \cdot \varphi, \eta, \lambda, 2s+\tfrac{1}{2})
   \]
   gives an $\Hl$-equivariant map from $\sigma$ to the space of polynomial sections of $I_H\left(\upsi_s^{-1}, \uchi_s^{-1}\right)$. For the spherical vector $\varphi_0$, normalised as in Proposition \ref{prop:novintegral}(b), we have
   \begin{align*}
    z_s(\varphi_0)(1) &= L(\psi_1/\chi_1, 2s+1)^{-1} L(\psi_2/\chi_2, 2s+1)^{-1},\\
    z_s(U(\ell) \varphi_0)(1) &=
    \frac{\ell^{2+2s}}{\psi_1 \psi_2(\ell)} \left[ L(\psi_1/\chi_1, 2s+1)^{-1} L(\psi_2/\chi_2, 2s+1)^{-1} - L(\sigma \otimes \psi_1 \psi_2, 2s+\tfrac12)^{-1}\right].
   \end{align*}
  \end{proposition}
 
  \begin{proof}Follows from Proposition \ref{prop:novintegral} and equation \eqref{eq:ulintegral}.
 \end{proof}
  
  We impose the following assumption:
  \begin{itemize}
   \item The functions $L(\sigma \otimes \psi_1\psi_2, s)$ and $L(\psi_1 / \chi_1, s+\tfrac12) L(\psi_2/\chi_2, s+\tfrac12)$ do not both have a pole at $s = \tfrac{1}{2}$.
  \end{itemize}
  It follows that the homomorphism $z \in \Hom_{\Hl}\left(\sigma, I_H(\upsi^{-1}, \uchi^{-1})\right)$ given by specialising $z_s$ at $s = 0$ is not zero, since at least one of $z(\varphi_0)$ and $z(U(\ell) \varphi_0)$ is non-vanishing at $1 \in \Hl$. Our assumptions on $\uchi, \upsi$ imply that, although $I_H(\upsi^{-1}, \uchi^{-1})$ may be reducible, it has a unique irreducible subrepresentation, and this subrepresentation is generic.
  
  \begin{lemma}\label{lem:properquot}
   The image of the homomorphism $z$ is contained in the unique irreducible subrepresentation of $I_H(\upsi^{-1}, \uchi^{-1})$.
  \end{lemma}
 
  \begin{proof}
   If $L(\psi_1 / \chi_1, s+\tfrac12) L(\psi_2/\chi_2, s+\tfrac12)$ is finite at $s = \tfrac{1}{2}$, then $I_H(\upsi^{-1}, \uchi^{-1})$ is irreducible and there is nothing to prove. So it suffices to treat the case when one or both of $\chi_i/\psi_i$ is $|\cdot|$, assuming that $L(\sigma \otimes \psi_1\psi_2, s)$ has no pole at $s = \tfrac{1}{2}$. We shall not give the details of this computation, as it is somewhat technical, and it will only be relevant in a few boundary cases. We write $\sigma$ as an induced representation from the Siegel parabolic $P_{\mathrm{S}}(\Ql)$. There are exactly two orbits of $\Hl$ on the flag variety $\Gl / P_{\mathrm{S}}(\Ql)$, and an application of Mackey theory allows us to compute $\Hom_{\Hl}(\sigma, \tau)$ for each non-generic quotient $\tau$ of $I_H(\upsi^{-1}, \uchi^{-1})$ in terms of the inducing data for $\sigma$. These $\Hom$-spaces all turn out to be zero unless $L(\sigma \otimes \psi_1\psi_2, s)$ has a pole at $s = \tfrac{1}{2}$.
  \end{proof}
 
  \begin{corollary}
   \label{prop:defofbiform}
   Let $\langle \cdot, \cdot \rangle$ denote the canonical duality pairing
   \[ I_H\left(\upsi_s, \uchi_s\right) \times I_H\left(\upsi^{-1}_s, \uchi_s^{-1}\right) \to \CC.\]
   Then the bilinear form $\grz_{\uchi, \upsi} \in \Hom_{\Hl}\big(I(\uchi, \upsi) \otimes \sigma, \CC\big)$ defined by
   \[ 
    \grz_{\uchi, \upsi}(f \otimes \varphi) \coloneqq 
      \lim_{s \to 0} L(\psi_1/\chi_1, 2s+1) L(\psi_2/\chi_2, 2s+1) \big\langle M(f_s), z_s(\varphi)\big\rangle,
   \]
   for $f_s$ any polynomial section of $I(\uchi_s, \upsi_s)$ passing through $f$, is well-defined and non-zero.
  \end{corollary}
 
  \begin{proof}
   We have $\big\langle M(f_s), z_s(\varphi)\big\rangle = \big\langle f_s, M(z_s(\varphi))\big\rangle$ by Proposition \ref{prop:Madjoint}. From the previous lemma, $M(z_s(\varphi))$ vanishes at $s=0$ to order equal to the order of the pole of the Euler factor, so the limit is well-defined and depends only on $f$.
  \end{proof}
  
  \begin{remark}
   \label{rmk:PS}
   In the paper \cite{piatetskishapiro97} (published in 1997, but circulated as a preprint many years before), Piatetski-Shapiro defines a zeta-integral $Z(\varphi, \uphi, \lambda, \eta, s)$, for $\varphi \in \sigma$ and $\uphi \in \cS(\Ql^2 \times \Ql^2)$, by
   \[ Z(\varphi, \uphi, \lambda, \eta, s) \coloneqq \int_{N_{\Hl} \backslash \Hl} B_{\varphi, \lambda}(h) \uphi\big( (0, 1) \cdot h_1, (0, 1) \cdot h_2\big) \eta(\det h) |\det h|^{s+\tfrac{1}{2}}\dd h, \]
   where $N_{\Hl}$ is the unipotent radical of $B_H(\Ql)$. This integral also appears in Sugano's work \cite{sugano85}. If $\eta$ and $\lambda$ are chosen as above, then one checks that $\grz_{\uchi, \upsi}\left(f_{\hat\uphi, \uchi, \upsi} \otimes \varphi\right)$ is equal to the leading term of $Z(\varphi, \uphi, \Lambda, \eta, s)$ at $s = \tfrac{1}{2}$, up to a non-zero scalar factor. However, we cannot simply take this as the definition of $\grz_{\uchi, \upsi}$, since it is not \emph{a priori} clear that this leading term depends only on the vector $f_{\hat\uphi, \uchi, \upsi} \in I_H(\uchi, \upsi)$ rather than on $\uphi$ itself.
  \end{remark}

  Vitally, the linear functional $\grz_{\uchi, \upsi}$ of Proposition \ref{prop:defofbiform} is \emph{unique}:
  
  \begin{theorem}[Kato--Murase--Sugano]
   \label{thm:KMS}
   For $\uchi$ and $\upsi$ satisfying the above assumptions, we have 
   \[ \dim \Hom_{\Hl}\big(I_H(\uchi, \upsi) \otimes \sigma, \CC\big) \le 1.\]
  \end{theorem}
  
  \begin{proof}
   Let $\tau$ be the representation $I_H(\uchi, \upsi)$. Since the group of unramified characters of $\Ql^\times$ is 2-divisible, we may replace $\sigma$ and $\tau$ with $\sigma \otimes \omega^{-1}$ and $\tau \otimes \omega$, where $\omega$ is any square root of $\chi_\sigma$, and therefore assume that both $\sigma$ and $\tau$ are trivial on $Z(\Gl)$. Thus $\sigma$ factors through $\overline{G} = \operatorname{PGSp}_4(\Ql) = \SO_5(\Ql)$; and $\tau$ factors through the image $\overline{H}$ of $H$ in $\SO_5(\Ql)$, which is a copy of $\SO_4(\Ql)$, embedded as the stabiliser of an anisotropic vector in the defining 5-dimensional representation.
   
   We now apply the main theorem of \cite{katomurasesugano03}, which shows that for any representations $\sigma$ of $\SO_5$ and $\tau$ of $\SO_4$ which are generated by a spherical vector, the Hom-space $\Hom(\tau \otimes \sigma, \CC)$ has dimension $\le 1$.
  \end{proof}
  
  \begin{remark}
   Alternatively, it follows from the proof of Lemma \ref{lem:properquot} that this $\Hom$-space injects into $\Hom_{\Hl}(\sigma \otimes \tau_0, \CC)$ where $\tau_0$ is the unique irreducible subrepresentation of $I_H(\uchi, \upsi)$. We can now invoke a very general result, which forms part of the Gan--Gross--Prasad conjecture for special orthogonal groups: for any $n \ge 0$ and any irreducible smooth representations $\sigma$ of $\SO_{n+1}(\Ql)$ and $\rho$ of $\SO_n(\Ql)$, one has $\dim \Hom_{\SO_n}(\sigma \otimes \rho, \CC) \le 1$, by \cite[Th\'eor\`eme 1]{waldspurger12}.
  \end{remark}
     
  \begin{corollary}
   In the situation of Theorem \ref{thm:KMS}, the bilinear form $\grz_{\uchi, \upsi}$ is a basis of $\Hom_{\Hl}\big(I_H(\uchi, \upsi) \otimes \sigma, \CC\big)$.
  \end{corollary}
 
  \begin{proof} Clear. \end{proof}

 \subsection{Explicit formulae for the unramified local pairing}

  We record the following formulae for the values of $\grz_{\uchi, \upsi}$. We assume, as before, that $\sigma$ is an irreducible unramified principal series representation of $\Gl$. We choose our characters $\uchi$ and $\upsi$ as follows:
  \begin{itemize}
   \item $\psi_1 = \psi_2 = |\cdot|^{-1/2}$,
   \item $\chi_i = |\cdot|^{(1/2 + k_i)} \tau_i$, where $\underline{\tau} = (\tau_1, \tau_2)$ is a pair of finite-order unramified characters, and $k_i \ge 0$ are integers.
  \end{itemize}
  If one or both of the $k_i$ is zero, we also assume that $\sigma$ is essentially tempered (a twist of a tempered representation); since $\chi_\sigma$ is $|\cdot|^{-(k_1+k_2)}$ up to a finite-order character, all poles of $L(\sigma, s)$ therefore have real part $\tfrac{k_1+k_2}{2} \ge 0$, so that $L(\sigma \otimes \psi_1\psi_2, \tfrac{1}{2}) = L(\sigma, -\tfrac{1}{2})$ is finite and the assumptions of the previous section are satisfied.
  
  For $\uphi \in \cS(\Ql^2)^{\otimes 2}$, we write $F_{\uphi}$ for the Siegel section $f_{\hat\uphi, \uchi, \upsi} \in I_H(\uchi, \upsi)$; from \eqref{zeta-section-eq}, this depends $\Hl$-equivariantly on $\uphi$. We shall apply this to the particular Schwartz functions $\uphi_t$ introduced in \S\ref{sect:Hindrep} above.
 
  \begin{theorem}
   \label{thm:tamenormrel}
   Let $\grz \in \Hom_{\Hl}(I(\uchi, \upsi) \otimes \sigma, \CC)$. Then, for any $t \ge 1$, we have
   \[ 
    \grz\left(F_{\uphi_t}, \varphi_0\right) = \frac{1}{\ell^{2t-2}(\ell + 1)^2} \left( 1 - \tfrac{\ell^{k_1}}{\tau_1(\ell)}\right) \left( 1 - \tfrac{\ell^{k_2}}{\tau_2(\ell)}\right) \cdot \grz\left(F_{\uphi_0}, \varphi_0\right)
   \]
   and
   \[
    \grz\left(F_{\uphi_1}, U(\ell) \varphi_0\right) =  \frac{\ell}{(\ell + 1)^2} \left[\left( 1 - \tfrac{\ell^{k_1}}{\tau_1(\ell)}\right) \left( 1 - \tfrac{\ell^{k_2}}{\tau_2(\ell)}\right) - L(\sigma, -\tfrac12)^{-1} \right]\grz\left(F_{\uphi_0}, \varphi_0\right).
   \]
  \end{theorem}
  
  \begin{proof}
   We know that $\Hom_{\Hl}\left(I(\uchi, \upsi) \otimes \sigma, \CC\right)$ is 1-dimensional and spanned by the specific bilinear form $\grz_{\uchi, \upsi}$ constructed above, so it suffices to assume that $\grz = \grz_{\uchi, \upsi}$. By construction $F_{\uphi_t}$ is the value at $s=0$ of the Siegel section $f_{\hat\uphi_t, \uchi_s, \upsi_s}$, and we have
   \[ M\left(f_{\hat\uphi_t, \uchi_s, \upsi_s}\right) = L(\chi_1/\psi_1, 1-2s)^{-1} L(\chi_2/\psi_2, 1-2s)^{-1} f_{\uphi_t, \upsi_s, \uchi_s},\]
   by the functional equation for Siegel sections (Proposition \ref{prop:SSFE}). As we have seen above, the restriction of $f_{\uphi_t, \upsi_s, \uchi_s}$ to $H(\hat\ZZ)$ is a scalar multiple of the characteristic function of $K_{\Hl,0}(\ell^t)$, so we have
   \[\grz_{\uchi, \upsi}\left(F_{\uphi_t}, \varphi\right) = \frac{\vol{K_{\Hl,0}(\ell^t)} f_{\uphi_t, \upsi, \uchi}(1)}{L(\chi_1/\psi_1, 1) L(\chi_2/\psi_2, 1)}
    \times \lim_{s \to 0} \Big[ L(\psi_1/\chi_1, 1+2s) L(\psi_2/\chi_2, 1 + 2s) z_s(\varphi)(1)\Big],
   \]
   for any $\varphi \in \sigma$ invariant under $K_{\Hl, 0}(\ell^t)$. In particular, if $\varphi = \varphi_0$ then the bracketed term is identically 1, and from the formula for $f_{\uphi_t, \uchi, \upsi}(1)$ given in Lemma \ref{GL2-lem} we see that for $t \ge 1$ we have
   \[ \grz_{\uchi, \upsi}(F_{\uphi_t}, \varphi_0) = \vol{K_{\Hl,0}(\ell^t)} \cdot  L(\psi_1/\chi_1, 1)^{-1} L(\psi_2/\chi_2, 1)^{-1} \grz_{\uchi, \upsi, s}(F_{\uphi_0}, \varphi_0),\]
   which is the first formula claimed. The second formula is similar, using the formula for $z_s(U(\ell)\varphi_0)(1)$ given in Proposition \ref{prop:ulintegral}.   
  \end{proof}

 \subsection{An application of Frobenius reciprocity} 
  
  \begin{proposition}
   \label{prop:frobrecip}
   Let $\tau$ (resp.~$\sigma$) be smooth representations of $\Hl$ and $\Gl$ respectively. Then there are canonical bijections of $\CC$-vector spaces
   \[ 
    \Hom_{\Gl}\left(\cInd_{\Hl}^{\Gl}(\tau), \sigma^\vee\middle) \cong
    \Hom_{\Gl}\middle(\cInd_{\Hl}^{\Gl}(\tau) \otimes \sigma, \CC\middle)
    \cong \Hom_{\Hl}\middle(\tau \otimes (\sigma|_{\Hl}), \CC\right).
   \]
  \end{proposition}
  
  \begin{proof}
   The first isomorphism is standard, and interchanging the roles of $\sigma$ and the compactly-induced representation also shows that
   \[ 
    \Hom_{\Gl}\left(\cInd_{\Hl}^{\Gl}(\tau) \otimes \sigma, \CC\middle) \cong \Hom_{\Gl}\middle(\sigma, (\cInd_{\Hl}^{\Gl} \tau)^\vee\right).
   \]
   One has a canonical isomorphism $(\cInd_{\Hl}^{\Gl} \tau)^\vee = \Ind_{\Hl}^{\Gl} (\tau^\vee)$ \cite[\S III.2.7]{renard10}. (Care must be taken here since the contragredient on the left-hand side denotes $\Gl$-smooth vectors in the abstract vector-space dual, while on the right-hand side it denotes $\Hl$-smooth vectors.) We then apply Frobenius reciprocity for the non-compact induction [\emph{op.cit}, \S III.2.5] to obtain
   \begin{align*}
    \Hom_{\Gl}\left(\sigma, \Ind_{\Hl}^{\Gl}(\tau^\vee)\right) & \cong \Hom_{\Hl}\Big((\sigma |_{\Hl}), \tau^\vee\Big) \\
    & \cong \Hom_{\Hl}\Big( \tau \otimes (\sigma|_{\Hl}), \CC\Big),
   \end{align*}
   as required.
  \end{proof}
 
  \begin{remark}
   The $\Hom$-spaces in Proposition \ref{prop:frobrecip} will \emph{not} in general be isomorphic to $\Hom_{\Hl}\big(\tau, (\sigma^\vee)|_{\Hl}\big)$. The problem is that $(\sigma^\vee) |_{\Hl}$ is in general much smaller than $(\sigma |_{\Hl})^\vee$, since the two notions of contragredient do not match -- an $\Hl$-smooth linear functional on $\sigma$ may not be $\Gl$-smooth.
  \end{remark}
  
  For later use it will be important to have an explicit form for this bijection. Let $\cH(\Gl)$ denote the Hecke algebra of locally-constant, compactly-supported $\CC$-valued functions on $\Gl$, with the algebra structure defined by convolution (normalising Haar measure as in \S\ref{sect:compacta}). We regard $\sigma$ as a left $\cH(\Gl)$-module via the usual formula
  \[ \xi \cdot \varphi = \int_{\Gl} \xi(g) (g \cdot \varphi)\dd g,\]
  so that $g_1 \cdot \Big(\xi \cdot (g_2 \cdot \varphi)\Big) = \xi(g_1^{-1} (-)g_2^{-1}) \cdot \varphi$.

  \begin{definition}
   \label{def:indrep}
   For smooth representations $\tau$, $\sigma$ as above, let $\mathfrak{X}(\tau, \sigma^\vee)$ denote the space of linear maps
   \[ \grZ: \tau \otimes_{\CC} \cH(\Gl) \to \sigma^\vee \]
   which are $\Hl \times \Gl$-equivariant, where the actions are defined as follows:
   \begin{itemize}
    \item The $\Hl$ factor acts trivially on $\sigma^\vee$, and on $\tau \otimes_{\CC} \cH(\Gl)$ it acts via the formula
    \[ h \cdot (v \otimes \xi) = (h \cdot v) \otimes \xi(h^{-1}(-)). \]
    \item The $\Gl$ factor acts trivially on $\tau$, and on $\cH(\Gl)$ it acts via $g \cdot \xi = \xi( (-) g)$.
   \end{itemize}
  \end{definition}
  
  Unwinding the definitions, we reach the following formula:
  
  \begin{proposition}
   There is a canonical bijection between $\mathfrak{X}(\tau, \sigma^\vee)$ and $\Hom_{\Hl}\left(\tau \otimes (\sigma |_{\Hl}), \CC\right)$, characterised as follows: if $\grZ \in \mathfrak{X}(\tau, \sigma^\vee)$ corresponds to $\grz \in \Hom_{\Hl}\left(\tau \otimes (\sigma |_{\Hl}), \CC\right)$, then we have
   \[ \grZ(f \otimes \xi)(\varphi) = \grz(f \otimes (\xi \cdot \varphi) ) \]
   for all $f \in \tau$, $\xi \in \cH(\Gl)$, and $\varphi \in \sigma$.
  \end{proposition}
 
  \begin{proof} Immediate.
  \end{proof}
 
  \begin{corollary}
   Suppose $\grz \leftrightarrow \grZ$ as in the above proposition; and let $U_0 \ge U_1$ be two open compact subgroups of $\Gl$, $f_0, f_1 \in \tau$, and $g_0, g_1 \in \Gl$. Suppose that
   \[ \grz(f_1, g_1 \cdot \varphi) = \grz(f_0, g_0 \cdot R \cdot \varphi) \]
   for some $R \in \cH(U_0 \backslash \Gl / U_0)$ and all $\varphi \in \sigma^{U_0}$. Then the elements $\grZ_i = \grZ(f_i \otimes \ch(g_i U_i)) \in (\sigma^\vee)^{U_i}$, $i = 0, 1$, are related by
   \[ 
    \sum_{u \in U_0 / U_1} u \cdot \grZ_1  = R' \cdot \grZ_0,
   \]
   as elements of $(\sigma^\vee)^{U_0}$, where $R'(g) = R(g^{-1})$.
  \end{corollary}
  
  \begin{proof}
   Since both sides of the desired equality are in $(\sigma^\vee)^{U_0} = (\sigma^{U_0})^\vee$, it suffices to check that they pair to the same value with $\varphi$ for every $\varphi \in \sigma^{U_0}$. This now follows from the above description of $\grZ(-)(\varphi)$.
  \end{proof}

 \subsection{Results for deeper levels}
 
  In order to prove norm-compatibility relations in the ``$p$-direction'' for our Euler system, we shall also need a few supplementary results which are proved directly (rather than using the uniqueness result of Theorem \ref{thm:KMS}). In this section, $W$ denotes an arbitrary smooth complex representation of $\Gl$ (not necessarily irreducible or even admissible), and we let $\mathfrak{X}(W)$ denote the space of homomorphisms
  \[ \grZ: \cS(\Ql^2, \CC)^{\otimes 2} \otimes_{\CC} \cH(\Gl) \to W \]
  satisfying the same equivariance property under $\Hl \times \Gl$ as in Definition \ref{def:indrep}.
  
  \begin{notation}
   For $t \ge 1$, let $\phi_{1, t} \in \cS(\Ql^2, \CC)$ denote the characteristic function of the set $\ell^t \Zl \times( 1 + \ell^t \Zl)$, and $\uphi_{1, t} = \phi_{1, t} \otimes \phi_{1, t} \in \cS(\Ql^2, \CC)^{\otimes 2}$. This is stable under the group $K_{\Hl, 1}(\ell^t)$ (see \S \ref{sect:compacta}).
  \end{notation}
  
  \begin{lemma}
   \label{lem:indept}
   Let $\xi \in \cH(\Gl)$ be invariant under left-translation by the principal congruence subgroup of level $\ell^T$ in $H(\Zl)$, for some $T \ge 1$. Then, for any $\grZ \in \mathfrak{X}(W)$, the expression
   \[ \frac{1}{\vol K_{\Hl, 1}(\ell^t)}\, \grZ\left( \uphi_{1, t} \otimes \xi \right) \]
   is independent of $t \ge T$, where $\vol(-)$ denotes volume with respect to our fixed Haar measure on $\Hl$.
  \end{lemma}
  
  \begin{proof}
   For any integers $t \ge T \ge 1$, let $J$ be a set of coset representatives for the quotient $K_{\Hl, 1}(\ell^T) / K_{\Hl, 1}(\ell^t)$. Then $\uphi_{1, T} = \sum_{\gamma \in J} \gamma \cdot \uphi_{1, t}$, so we have
   \[ 
    \grZ\left( \uphi_{1, T}\otimes\xi \right) = \sum_{\gamma \in J} \grZ\left( (\gamma \cdot {\uphi_{1, t}}) \otimes \xi \right) = \sum_{\gamma \in J} \grZ\left(\uphi_{1, t} \otimes \xi(\gamma(-)) \right).
   \]
   We can (and do) assume that $J$ is a subset of the principal congruence subgroup of level $\ell^T$ in $H$. By assumption, all such elements will act trivially on $\xi$ from the left, so the above equality becomes
   \[ 
    \grZ\left( \uphi_{1, T} \otimes \xi \right) = (\# J) \cdot \grZ\left(\uphi_{1, t} \otimes \xi \right) = \frac {\vol K_{\Hl, 1}(\ell^T)} {\vol K_{\Hl, 1}(\ell^t)}\grZ\left(\uphi_{1, t} \otimes \xi \right)
   \]
   as required. 
  \end{proof}
  
  \begin{notation}
   We write $\grZ\left( \uphi_{1, \infty} \otimes \xi \right)$ for this limiting value.
  \end{notation}
    
  The case that interests us is the following. Let $m, n$ be integers with $m \ge 0$ and $n \ge \max(m, 1)$, and let $\eta_m = \begin{smatrix} 1 &&{\ell^{-m}}\\&1&&{\ell^{-m}}\\&&1\\&&&1\end{smatrix}$. Recall the subgroup $K_{\Gl}(\ell^m, \ell^n)$ of \S\ref{sect:compacta}; we consider the Hecke operator
  \[ U'(\ell) = \frac{1}{\vol K_{\Gl}(\ell^m, \ell^n)} \ch\left(K_{\Gl}(\ell^m, \ell^n) \dfour{\ell^{-1}}{\ell^{-1}}{1}{1} K_{\Gl}(\ell^m, \ell^n)\right).\]
  
  \begin{proposition}
   \label{prop:wildnormrel}
   For any $\grZ \in \mathfrak{X}(W)$ we have
   \[ 
    \grZ\left( \uphi_{1, \infty} \otimes \ch(\eta_{m+1} K_{\Gl}(\ell^m, \ell^n))\right) = 
    \left.\begin{cases}
     \tfrac{1}{\ell} U'(\ell)
      & \text{if $m \ge 1$}\\
     \tfrac{1}{\ell-1} \left[ U'(\ell) - 1\right] & \text{if $m = 0$}
    \end{cases}\right\} \cdot  
    \grZ\left( \uphi_{1, \infty} \otimes \ch(\eta_m K_{\Gl}(\ell^m, \ell^n))\right)
   \]
  \end{proposition}
  
  \begin{proof} Writing $K = K_{\Gl}(\ell^m, \ell^n)$ for brevity, we have
   \begin{align*}
    U'(\ell) \cdot \grZ\left( \uphi_{1, \infty} \otimes \ch(\eta_m K )\right) &= \sum_{u, v, w \in \ZZ/\ell} \grZ\left( \uphi_{1, \infty} \otimes \ch\left(\eta_m \begin{smatrix} \ell & & u & v \\ &\ell&w&u\\&&1\\&&&1\end{smatrix} K_{m, n}\right)\right)\\
    &= \frac{1}{\vol K_{\Hl, 1}(\ell^n)} \sum_{u, v, w} \grZ\left( \left( \stbt{\ell}{v}{}{1}, \stbt{\ell}{w}{}{1} \right)^{-1} \uphi_{1, n} \otimes \ch(\eta_{m+1}^{(1 + \ell^m u)} K)\right)\\
    &= \frac{\ell^2}{\vol K_{\Hl, 1}(\ell^n)} \sum_u \grZ\left(  \ch( \ell^{n+1} \Zl \times (1 + \ell^n \Zl))^{\otimes 2} \otimes \ch(\eta_{m+1}^{(1 + \ell^m u)} K)\right)\\
    &= \sum_u \grZ\left( \uphi_{1, \infty} \otimes \ch(\eta_{m+1}^{(1 + \ell^m u)} K\right).
   \end{align*}
   There are now two cases to consider. If $m \ge 1$ then all terms in this sum are actually equal, since the powers of $\eta_{m+1}$ are conjugate via elements of the form $\dfour{a}{a}{1}{1}$ (with $a \in 1 + \ell^m \Zl$), which are in $K$ and act trivially on the Schwartz function $\uphi_{1, n}$; so the sum is simply $\ell \grZ\left( \uphi_{1, \infty} \otimes \ch(\eta_m K)\right)$ as required. If $m = 0$, then $\ell-1$ of the terms are conjugate, but the term for $u = -1$ requires special consideration since $1 + \ell^m u = 0$; thus we obtain
   \[ \left[ U'(\ell) -1\right] \cdot \grZ\left( \uphi_{1, \infty} \otimes \ch(K )\right) = (\ell - 1) \grZ\left( \uphi_{1, \infty} \otimes \ch(\eta_1 K )\right),\]
   which proves the formula in this case also.
  \end{proof}
  
  We also have an analogous result for $n = 0$, under rather stricter hypotheses. We take for $W$ the smooth dual $\sigma^\vee$ of an essentially tempered, unramified principal series representation of $\Gl$. We shall suppose that $\grZ \in \mathfrak{X}(\sigma^\vee)$ factors through a certain induced representation of $H$: more precisely, we shall take pairs of characters $\uchi = (\chi_1, \chi_2)$ and $\upsi = (\psi_1, \psi_2)$ of $\Ql^\times$ with $\psi_1 = \psi_2 = |\cdot|^{-1/2}$ and $\chi_i = |\cdot|^{k_i + 1/2} \tau_i$ for finite-order characters $\tau_i$ and positive integers $k_i$, so our setup is similar to Theorem \ref{thm:tamenormrel} except that we do \emph{not} assume the $\tau_i$ to be unramified. We then have a natural map 
  \[ 
   \cS(\Ql^2, \CC)^{\otimes 2} \rTo_{\uphi \mapsto F_{\uphi}} I_H(\uchi, \upsi),
  \]
  and we suppose that $\grZ$ factors through this map.
  
  \begin{corollary}
   \label{cor:tamenormfinal}
   In this situation, we have
   \[ 
    \grZ\left(\uphi_{1, \infty} \otimes \left(\ch(K_{\Gl}) - \ch(\eta_1 K_{\Gl})\right)\right) = \tfrac{\ell}{\ell-1} L(\sigma, -\tfrac{1}{2})^{-1} \cdot \grZ\left(\uphi_0, \ch(K_{\Gl})\right).
   \]
   In particular, if the $\tau_i$ are \emph{not} both unramified, then this holds vacuously (both sides of the formula are zero).
  \end{corollary}
  
  \begin{proof}
   Since $K_{\Gl, 0}(\ell)$ fixes $\ch(K_{\Gl})$ on the left, we have
   \[ \grZ(\uphi_{1, \infty} \otimes \ch(K_{\Gl})) = \frac{1}{\vol K_{\Hl, 0}(\ell)} \grZ\left(\uphi_{1} \otimes \ch(K_{\Gl}) \right) = (\ell + 1)^2 \grZ\left(\uphi_1 \otimes \ch(K_{\Gl}) \right). \]
   where (as in \S\ref{sect:Hindrep}) $\uphi_1 = \ch(\ell \Zl \times \Zl^\times)^{\otimes 2}$. On the other hand, taking $m=0$ and $n=1$ in Proposition \ref{prop:wildnormrel}, we have
   \[ 
    \grZ\left(\uphi_{1, \infty} \otimes \ch(\eta_1 K_{\Gl, 1}(\ell))\right) = \frac{U'(\ell) - 1}{\ell-1} \grZ\left(\uphi_{1, \infty} \otimes \ch(K_{\Gl, 1}(\ell))\right).
   \]
   Since the action of the quotient $K_{\Gl, 0}(\ell) / K_{\Gl, 1}(\ell) \cong \GL_2(\ZZ / \ell)$ commutes with the Hecke operator $U'(\ell)$, we can sum over representatives for the quotient to deduce that
   \[ 
    \grZ\left(\uphi_{1, \infty} \otimes \ch(\eta_1 K_{\Gl})\right) = \frac{(\ell + 1)^2}{\ell-1} \sum_{\gamma \in K_{\Gl} / K_{\Gl,0}(\ell)} \gamma \cdot (U'(\ell) - 1) \grZ\left(\uphi_{1} \otimes \ch(K_{\Gl, 0}(\ell))\right).
   \]
   Combining these two formulae we have
   \begin{multline*} 
    \grZ\left(\uphi_{1, \infty} \otimes \left(\ch(K_{\Gl}) - \ch(\eta_1 K_{\Gl})\right)\right) = 
    (\ell + 1)^2 (1 + \tfrac{1}{\ell-1}) \grZ\left(\uphi_{1} \otimes \ch(K_{\Gl}) \right)
    \\ - \tfrac{(\ell+1)^2}{\ell-1} \sum_{\gamma \in K / K_{\Gl, 0}(\ell)} \gamma U'(\ell) \cdot \grZ\left(\uphi_{1} \otimes \ch(K_{\Gl, 0}(\ell))\right).
   \end{multline*}
   We can now quickly dispose of the ramified cases. The map $\uphi \mapsto F_{\uphi}$ is a morphism of $(\GL_2 \times \GL_2)$-representations (not only of $H$-representations). Moreover, the elements $\uphi_0$ and $\uphi_1$ are the characteristic functions of subsets of $\Ql^2 \times \Ql^2$ invariant under $\Zl^\times \times \Zl^\times$; hence their images in any representation of $\GL_2 \times \GL_2$ with ramified central character must be zero. Hence $F_{\uphi_0}$ and $F_{\uphi_{1}}$ are both zero, and the desired formula becomes $0 = 0$, if either of the characters $\tau_i$ is ramified.
   
   Let us now assume that the $\tau_i$ are unramified, which means we may apply Theorem \ref{thm:tamenormrel}. Translating to the homomorphism $\grZ$ from its corresponding bilinear form $\grz$, the first statement in the theorem (for $t = 1$) becomes
   \[ 
    (\ell + 1)^2 (1 + \tfrac{1}{\ell-1}) \grZ\left(\uphi_{1} \otimes \ch(K_{\Gl}) \right) = \tfrac{\ell}{\ell-1} \left( 1 - \tfrac{\ell^{k_1}}{\tau_1(\ell)}\right) \left( 1 - \tfrac{\ell^{k_2}}{\tau_2(\ell)}\right) \cdot \grZ\left(\uphi_0 \otimes \ch(K_{\Gl})\right).
   \]
   On the other hand, the second statement of Theorem \ref{thm:tamenormrel} gives us
   \begin{multline*} 
    \tfrac{(\ell+1)^2}{\ell-1} \sum_{\gamma \in K_{\Gl} / K_{\Gl, 0}(\ell)} \gamma U'(\ell) \cdot \grZ\left(\uphi_{1} \otimes \ch(K_{\Gl, 0}(\ell))\right) \\
    = \tfrac{\ell}{\ell-1} 
    \left[
     \left( 1 - \tfrac{\ell^{k_1}}{\tau_1(\ell)}\right) 
     \left( 1 - \tfrac{\ell^{k_2}}{\tau_2(\ell)}\right) 
     - L(\sigma, -\tfrac12)^{-1} 
    \right]
    \grZ\left(\uphi_0 \otimes \ch(K_{\Gl})\right).
   \end{multline*}
   Combining these two formulae, the ``extra'' Euler factors coming from the $\tau_i$ cancel out, and we are left with the desired formula.
  \end{proof}

\section{Preliminaries II: Algebraic representations and Lie theory}

 \subsection{Representations of $G$}

  We recall the parametrization of algebraic representations of the group $\GSp_4$.

  \begin{notation}
   We write $T$ for the diagonal torus of $G$ (as in \S \ref{sect:groups} above), and we write $\chi_1, \dots, \chi_4$ for characters of $T$ given by projection onto the four entries. Thus $\chi_1 + \chi_4 = \chi_2 + \chi_3$ is the restriction to $T$ of the symplectic multiplier $\mu$, and $\{\chi_1, \chi_2, \mu\}$ is a basis of the character group $X^\bullet(T)$.
  \end{notation}
   
  \begin{definition}
   Let $a \ge 0, b \ge 0$ be integers. We denote by $V^{a,b}$ the unique (up to isomorphism) irreducible algebraic representation of $G$ whose highest weight, with respect to $B$, is the character $(a+b) \chi_1 + a \chi_2$.
  \end{definition}

  This representation has dimension $\frac{1}{6}(a + 1)(b + 1)(a + b + 2)(2a + b + 3)$. Its central character is $x \mapsto x^{2a + b}$, and it satisfies
  \[ (V^{a,b})^*  \cong V^{a,b} \otimes \mu^{-(2a+b)}.\]
  Note that $V^{0, 1}$ is the four-dimensional defining representation of $\GSp_4$, and $V^{1, 0}$ is the 5-dimensional direct summand of $\bigwedge^2 V^{0, 1}$. The representation $V^{1, 0} \otimes \mu^{-1}$ has trivial central character, and is the defining representation of $G / Z_G \cong \SO_5$.

 \subsection{Integral models}

  Let $\lambda = (a+b)\chi_1 + a\chi_2 + c\mu$, with $a, b \ge 0$, be a dominant integral weight, $V_\lambda$ the corresponding representation, and $v_\lambda$ a highest weight vector in $V_\lambda$. The pair $(V_\lambda, v_\lambda)$ is then unique up to unique isomorphism.

  An \emph{admissible lattice} in $V_\lambda$ is a $\ZZ$-lattice $L$ with the following properties:
  \begin{itemize}
   \item the homomorphism $\GSp_4 \to \GL(V_\lambda)$ extends to a homomorphism $\GSp_4 \to \GL(L)$ of group schemes over $\ZZ$;
   \item the intersection of $L$ with the highest weight space of $V_\lambda$ is $\ZZ \cdot v_\lambda$.
  \end{itemize}

  It is known that there are finitely many such lattices, each of which is the direct sum of its intersections with the weight spaces; and we set $V_{\lambda, \ZZ}$ to be the \emph{maximal} such lattice.

  \begin{proposition}
   \label{prop:Vabmult}
   Let $\lambda, \lambda'$ be dominant integral weights. Then there is a unique $G$-equivariant homomorphism, the \emph{Cartan product},
   \[ V_{\lambda, \ZZ} \otimes V_{\lambda', \ZZ} \to V_{\lambda + \lambda', \ZZ},\quad v \otimes w \mapsto v \cdot w \]
   such that $v_\lambda \cdot v_{\lambda'} = v_{\lambda + \lambda'}$. Moreover, for any non-zero $v \in V_\lambda$ and $v' \in V_{\lambda'}$, we have $v \cdot v' \ne 0$.
  \end{proposition}

  \begin{proof}
   After tensoring with $\QQ$ the existence and uniqueness of this homomorphism is obvious from highest-weight theory. Hence the image of $V_{\lambda, \ZZ} \otimes V_{\lambda', \ZZ}$ is a $\ZZ$-lattice in $V_{\lambda + \lambda'}$, which is clearly admissible; so it must be contained in the maximal one, which is $V_{\lambda + \lambda', \ZZ}$.

   This product gives the ring $\bigoplus_{\lambda} V_{\lambda, \ZZ}$ the structure of a graded ring. The Borel--Weil theorem shows that this ring injects into $\cO(G)$, which is an integral domain; so the Cartan product of non-zero vectors is non-zero.
  \end{proof}

 \subsection{Branching laws}

  We are interested in the restriction of $V^{a,b}$ to $H$ via the embedding $\iota: H \into G$, which we shall denote by $\iota^* (V^{a,b})$. Computing the weights of these representations (and their multiplicities), one deduces the following branching law describing $\iota^*(V^{a,b})$:

  \begin{proposition}
   The restriction of $V^{a,b}$ to $H = \GL_2 \times_{\GL_1} \GL_2$ via $\iota$ is given by
   \[ \iota^*(V^{a,b}) = \bigoplus_{0 \le q \le a} \bigoplus_{0 \le r \le b}  W^{a+b-q-r, a-q+r} \otimes \sideset{}{^q}\det,\]
   where $W^{c,d}$ denotes the representation $\Sym^c \boxtimes \Sym^d$ of $H$.
  \end{proposition}

  \begin{remark}
   Compare \cite[\S 1]{lemma17} for an equivalent, although less explicit, statement. In Lemma's notations the highest weight of our representation $V^{a,b}$ is $\lambda(a + b, a, 2a+b)$.
  \end{remark}

  For the constructions below it will be useful to fix choices of highest-weight vectors in each of these subrepresentations. For $0 \le q \le a$ and $0 \le r \le b$ we define a vector $v^{a,b,q,r} \in V^{a,b}_{\ZZ}$ as follows:
  \[ v^{a, b,q,r} = w^{a - q} \cdot v^{b-r} \cdot (w')^q \cdot (v')^r\]
  where the product operation is the Cartan product, and:
  \begin{itemize}
   \item $v \in V^{0, 1}$ is the highest-weight vector;
   \item $v' = X_{21} \cdot v$ is a basis of the $\chi_2$ weight space, where $X_{21} = \left(\begin{smallmatrix} 0 & 0 & 0 & 0 \\ 1 & 0 & 0 & 0 \\0 & 0 & 0 & 0 \\ 0 & 0 & -1 & 0 \end{smallmatrix}\right) \in \operatorname{Lie} G$.
   \item $w$ is the highest-weight vector of $V^{1, 0}$.
   \item $w' = Z \cdot w$ is a basis of the $\mu$ weight space of $V^{1,0}$, where 
   \( Z = \left(
    \begin{smallmatrix} 
     0 & 0 & 0 & 0 \\ 
     0 & 0 & 0 & 0 \\
     1 & 0 & 0 & 0 \\ 
     0 & 1 & 0 & 0 
    \end{smallmatrix}\right) \in \operatorname{Lie} G\).
  \end{itemize}

  \begin{remark}
   We can identify $V^{0, 1}$ with the standard representation of $\GSp_4 \subseteq \GL_4$, with basis $(e_1, \dots, e_4)$, by choosing the highest-weight vector $v = e_1$; of course we then have $v'= e_2$. Moreover, we can identify $V^{1, 0}$ with a subspace of $\bigwedge^2 V^{0, 1}$, by choosing $e_1 \wedge e_2$ for the highest-weight vector $w$; and it follows that $w'$ is the vector $e_1 \wedge e_4 - e_2 \wedge e_3$.
  \end{remark}

  \begin{proposition}
   \label{prop:vabqr}
   For all integers $0 \le q \le a$ and $0 \le r \le b$, the vector $v^{a,b,q,r}$ thus defined is a non-zero highest-weight vector for the unique irreducible $H$-summand of $\iota^*(V^{a,b})$ isomorphic to $W^{a+b-q-r, a-q+r} \otimes \sideset{}{^q}\det$.
  \end{proposition}

  \begin{proof}
   Since $v^{a,b,q,r}$ is a Cartan product of non-zero $H$-highest-weight vectors (i.e.~vectors fixed by the action of the unipotent radical of the Borel of $H$), it is itself a non-zero $H$-highest-weight vector, and thus generates an irreducible $H$-subrepresentation of $V^{a,b}$. The result now follows by comparing weights.
  \end{proof}

  Since the representation $W^{c,d}$ of $H$ has a canonical highest-weight vector (namely $e_1^c \boxtimes f_1^d$, where $(e_1, e_2)$ and $(f_1, f_2)$ are bases of the standard representations of the two $\GL_2$ factors), we therefore have a canonical homomorphism of $H$-representations
  \begin{equation}
   \label{eq:brab}
   \br^{[a,b,q,r]}:W^{a+b-q-r, a-q+r} \otimes \sideset{}{^q}\det \into \iota^* \left(V^{a,b}\right)
  \end{equation}
  mapping the highest-weight vector to $v^{a,b,q,r}$. We refer to these homomorphisms as \emph{branching maps}.
  
  \begin{proposition}
   \label{prop:branchingint}
   The maps $\br^{[a,b,q,r]}$ restrict to maps
   \[ W^{a+b-q-r, a-q+r}_{\ZZ} \otimes \sideset{}{^q}\det \into V^{a,b}_\ZZ, \]
   where $W^{a+b-q-r, a-q+r}_{\ZZ}$ is the minimal\footnote{Note that the minimal admissible lattice in the representation $\Sym^k$ of $\GL_2$ is isomorphic to the module $\operatorname{TSym}^k \ZZ^2$ of symmetric tensors, while $\Sym^k \ZZ^2$ is the maximal lattice.} admissible lattice in $W^{a+b-q-r, a-q+r}$.
  \end{proposition}
  
  \begin{proof}
   It is clear that $(\br^{[a,b,q,r]})^{-1}\left(V^{a,b}_\ZZ\right)$ is a lattice in $W^{a+b-q-r, a-q+r}\otimes\det^q$ stable under the action of $H$, and since $v^{a,b,q,r} \in V^{a,b}_\ZZ$, the intersection of this lattice with the highest-weight subspace contains the highest-weight vector $e_1^{a+b-q-r} \boxtimes f_1^{a-q+r}$. Hence this lattice must contain the minimal admissible lattice in $W^{a+b-q-r, a-q+r}$.
  \end{proof}

 \subsection{A Lie-theoretic computation}
  
  As in \S\ref{sect:groups} above, let $T' \subset T$ be the rank-1 split torus $\dfour{x}{x}{1}{1}$; and let $u$ be the element $\left(\begin{smallmatrix} 1 & & 1 & \\ & 1 &  & 1 \\ &&1\\&&&1 \end{smallmatrix}\right)$ of $G(\ZZ)$.

  Since $T'$ is split, the representations $V^{a,b}$ are the direct sums of their weight spaces relative to $T'$, with weights between $0$ and $2a + b$; and the $T'$-weight of $v^{a,b,q,r}$ is $2a + b - q$. The purpose of this section is to prove the following result, which will be used in \S \ref{sect:twistcompat}:

  \begin{lemma}
   \label{lemma:lielemmaV}
   Let $v = v^{a,b,q,r} \in V^{a,b}$ be one of the above $H$-highest-weight vectors. Then for any non-zero integer $h$, the projection of $u^h(v)$ to the highest $T'$-weight space is given by $(2h)^q v^{a,b,0,r}$.
  \end{lemma}

  \begin{proof}
   Recall that $v^{a,b,q,r} = v^{b-r} \cdot (v')^r \cdot w^{a - q} \cdot( w')^q$. The vectors $v$, $v'$, and $w$ all lie in the highest $T'$-weight subspaces of their parent representations, so they are fixed by $u$. Hence it suffices to check that the projection of $u^h(w')$ to the highest $T'$-weight subspace of $V^{1,0}$ is non-trivial; and one computes easily that $u^h(w') = w' + 2h w$.
  \end{proof}
  
\section{Modular varieties}

 \subsection{Modular curves}

  We fix conventions for modular curves.

  \begin{definition} \
   \begin{enumerate}[(a)]
    \item For $N \ge 3$, we let $Y(N)$ be the $\QQ$-variety pararametrising triples $(E, e_1, e_2)$, where $E$ is an elliptic curve (over some $\QQ$-algebra $R$) and $e_1, e_2 \in E(R)[N]$ are a basis of the $N$-torsion of $E$.
    \item If $\Sigma$ is a finite set of primes containing all those dividing $N$, we write $Y(N)_{\Sigma}$ for the natural model of $Y(N)$ over $\ZZ[1/\Sigma]$ (representing the same functor on $\ZZ[1/\Sigma]$-algebras).
   \end{enumerate}
  \end{definition}
  
  We identify $Y(N)(\CC)$ with the double quotient
  \[ \GL_2^+(\QQ) \backslash \left( \GL_2(\AA_f) \times \cH \right) / U(N), \]
  where $U(N) \subset \GL_2(\hat\ZZ)$ is the principal congruence subgroup of level $N$ and $\cH$ is the upper half-plane, in such a way that:
  \begin{itemize}
   \item The double coset of $(1, \tau)$, for $\tau \in \cH$, corresponds to the triple
   \[ \left( \frac{\CC}{\ZZ \tau + \ZZ}, \frac{\tau}{N}, \frac{1}{N}\right).\]
   \item The right-translation action of $g \in \GL_2(\hat\ZZ)$ on the double quotient corresponds to the action on $Y(N)(\CC)$ given by
   \[
    (E, e_1, e_2) \cdot g =
    (E, e_1', e_2'), \quad
    \begin{pmatrix} e_1'\\e_2' \end{pmatrix} = g^{-1} \cdot \begin{pmatrix} e_1\\e_2 \end{pmatrix}.
   \]
  \end{itemize}
  If $g \in \SL_2(\ZZ / N\ZZ)$ then the above action of $g$ on $Y(N)(\CC)$ coincides with the action of $\gamma^{-1}$ on $\cH$, for any $\gamma \in \SL_2(\ZZ)$ congruent to $g$ modulo $N$. The components of $Y(N)(\CC)$ are indexed by the set $\mu_N^\circ$ of primitive $N$-th roots of unity, via the Weil pairing $(E, e_1, e_2) \mapsto \langle e_1, e_2 \rangle_N$; and the induced action of $g \in \GL_2(\ZZ / N\ZZ)$ on $\mu_N^\circ$ is given by $g \cdot \zeta = \zeta^{1/\det(g)}$. 
  
  \begin{remark}
   \label{rmk:shdatum}
   Note that our model is \emph{not} the Deligne--Shimura canonical model of the Shimura variety for $\GL_2$ with its standard Shimura datum \cite[Example 5.6]{milne-shimura}. Rather, it is the canonical model for the twisted Shimura datum defined by 
   \[ (a + ib) \in \operatorname{Res}_{\CC/\RR}(\mathbf{G}_m) \mapsto \tfrac{1}{a^2+b^2}\stbt{a}{b}{-b}{a}, \]
   which has the effect of flipping the sign of the Galois action on the connnected components.
  \end{remark}

  By passage to the quotient, we define similarly algebraic varieties $Y(U)$ over $\QQ$, for every open compact subgroup $U \subset \GL_2(\AA_f)$; if $U$ is unramified outside the finite set $\Sigma$, then $Y(U)$ has a model over $\ZZ[1/\Sigma]$ which we denote $Y(U)_\Sigma$.
  
  The right-translation action gives isomorphisms $\eta: Y(U) \to Y(\eta^{-1} U \eta)$ for every $\eta \in \GL_2(\AA_f)$, which are compatible with the action of $\eta^{-1}$ on $\cH$ if $\eta \in \GL_2^+(\QQ)$; in particular, scalar matrices $\stbt{A}{0}{0}{A}$ with $A \in \QQ^\times$ act trivially. This structure allows us to view the inverse limit $Y = \varprojlim_U Y(U)$ as a pro-variety over $\QQ$ with a right action of $\GL_2(\AA_f)$, whose $\CC$-points are $\GL_2^+(\QQ) \backslash \left(\GL_2(\AA_f) \times \cH\right)$. 
  
  \begin{definition}
   We say $U \subset \GL_2(\AA_f)$ is \emph{sufficiently small} if every non-identity element of $U$ acts without fixed points on the set $\GL_2^+(\QQ) \backslash (\GL_2(\AA_f) \times \cH)$.
  \end{definition}

  This condition is equivalent to requiring that for all $g \in \GL_2(\AA_f)$, every non-identity element of the discrete group $\Gamma = \GL_2^+(\QQ) \cap gUg^{-1}$ acts without fixed points on $\cH$. For instance, $U(N)$ is sufficiently small if $N \ge 3$. If $U$ is sufficiently small, then $Y(U)$ is the solution to a moduli problem (classifying elliptic curves with appropriate level structure), and therefore has an associated universal elliptic curve $\sE(U) \to Y(U)$.

  We define similarly algebraic surfaces $Y_H(U)$, where $U$ is an open compact subgroup of $H(\AA_f)$, and by passage to the limit a pro-variety  $Y_H = \varprojlim_UY_H(U)$ with a right action of $H(\AA_f)$. Of course, if $U$ is a fibre product $U_1 \times U_2$ of subgroups of $\GL_2(\AA_f)$ such that $\det(U_1) = \det(U_2)$, then $Y_H(U)$ is the fibre product of the modular curves $Y(U_1)$ and $Y(U_2)$ over their common component set $\Zhat^\times / \det(U_1) = \Zhat^\times / \det(U_2)$.
  
 \subsection{Siegel modular varieties}
  \label{sect:siegelMV}

  \begin{definition}
   Let $N \ge 3$. There exists a $\QQ$-variety $Y_G(N)$, smooth and quasiprojective of dimension 3, parametrising $6$-tuples $(A, \lambda, e_1, \dots, e_4)$ where
   \begin{itemize}
    \item $A$ is an abelian surface (over some $\QQ$-algebra $R$);
    \item $\lambda$ is a principal polarization $A \cong A^\vee$;
    \item $e_1, \dots, e_4$ are $N$-torsion sections of $A$ giving an isomorphism $A[N] \cong (\ZZ / N\ZZ)^4$;
    \item the matrix of the Weil pairing (induced by the polarization $\lambda$), with respect to the basis $e_1, \dots, e_4$, is $J \cdot \zeta$ for some $\zeta \in R^\times$.
   \end{itemize}
   Moreover, for any finite set of primes $\Sigma$ containing the primes dividing $N$, $Y_G(N)$ has a model $Y_G(N)_{\Sigma}$ over $\ZZ[1/\Sigma]$, representing the same functor for $\ZZ[1/N]$-algebras.
  \end{definition}
 
  For the existence of this scheme see e.g.~\cite[Corollary 3.3]{laumon05}. The complex manifold $Y_G(N)(\CC)$ can be identified with the double quotient
  \[  \GSp_4^+(\QQ) \backslash \left( \GSp_4(\AA_f) \times \cH_2\right) / U_G(N), \]
  where $\cH_2$ denotes the genus 2 Siegel space of symmetric complex $2 \times 2$ matrices with positive-definite imaginary part, and $U_G(N)$ is the principal congruence subgroup of level $N$ (the kernel of reduction $\GSp_4(\widehat\ZZ) \to \GSp_4(\ZZ / N\ZZ)$).

  The right-translation action of $\GSp_4(\ZZ / N\ZZ)$ on $Y_G(N)$ corresponds to the action on the moduli problem given by $g: (A, \lambda, e_1, \dots, e_4) \mapsto (A, \lambda, e_1', \dots, e_4')$, where
  \[
   \begin{pmatrix} e_1' \\ \vdots \\ e_4' \end{pmatrix} = g^{-1} \cdot \begin{pmatrix} e_1 \\ \vdots \\ e_4 \end{pmatrix}.
  \]

  More generally, if $U$ is any open compact subgroup of $G(\AA_f)$, we define a $\QQ$-model for $Y_G(U)$ by taking the quotient of $Y_G(N)$ by the action of $U / U_G(N)$, for any $N \ge 3$ such that $U_G(N) \subseteq U$. The same procedure gives $\ZZ[1/\Sigma]$-models $Y_G(U)_{\Sigma}$ if $U$ is unramified outside $\Sigma$. If $U$ is sufficiently small (in the same sense as for $\GL_2$), then $Y_G(U)$ is smooth, and can be interpreted as a moduli space for abelian surfaces with level $U$ structure. As before, the inverse limit $\varprojlim_U Y_G(U)$ acquires a right action of $G(\AA_f)$.
 
 \subsection{The embedding of $Y_H$ in $Y_G$}
  \label{sect:Hsmall}
  
  For any open compact subgroup $U$ of $G(\AA_f)$, we have a natural morphism of $\QQ$-varieties
  \[ \iota_U: Y_H(U \cap H) \rTo Y_G(U). \]
  The map $\iota_U$ is not always injective (even if $U$ is sufficiently small). However, we have the following criterion:
  
  \begin{proposition}
   Suppose there is an open compact subgroup $\tilde U$ containing $U$ and $w U w$, where $w = \dfour{-1}{1}{1}{-1}$, with $\tilde U$ sufficiently small. Then $\iota_U$ is injective.
  \end{proposition}
 
  \begin{proof}
   As above, we write $Y_G$ for the infinite-level Shimura variety $G(\QQ)_+ \backslash \left( G(\AA_f) \times \cH_2 \right)$, and similarly for $Y_H$. It is clear that $\iota$ gives an injection $Y_H \into Y_G$. If $Q, Q' \in Y_H$ have the same image in $Y_G(U)$, then $Q' = Qu$ for some $u \in U$; we want this to imply that $u$ lie in $U \cap H$. So it suffices to prove that, for any element of $U - (U \cap H)$, we have $Y_H u \cap Y_H = \varnothing$ as subsets of $Y_G$.
   
   Since $w$ is central in $H$, its action on $Y_G$ fixes $Y_H$ pointwise. Thus, if $Q \in Y_H$ and $Qu \in Y_H$, we have $Quw = Qu = Qwu$, so $\tilde u = u \cdot w u^{-1} w$ fixes $Q$. This element $\tilde u$ lies in $\tilde U$, by hypothesis, and since $\tilde U$ is sufficiently small, we conclude that $\tilde u = 1$. Thus $u$ lies in the centraliser of $w$ in $G(\AA_f)$, which is exactly $H(\AA_f)$.
  \end{proof}

  We shall say a subgroup $U$ is \textbf{$H$-small} if it satisfies the hypotheses of the above proposition. For instance, if $U$ is contained in the principal congruence subgroup $U_G(N)$ for some $N \ge 3$, then $U$ is $H$-small (since $U_G(N)$ is normal in $G(\Zhat)$, and sufficiently small by \cite[\S 0.6]{pink90}).
   
 \subsection{Component groups and base extension}
  \label{sect:components}
  
  Via strong approximation for $\Sp_4$, we have an isomorphism of component sets
  \[ \pi_0(Y_G(\CC)) = \QQ^\times_+ \backslash \AA_f^\times \cong \Zhat^\times. \]
  Our moduli-space description of $Y_G$ determines a Galois action on these components as follows.
  
  \begin{definition}
   We write 
   \[ \operatorname{Art}: \QQ^\times_+ \backslash \AA_f^\times \to \Gal(\QQbar / \QQ)^{\mathrm{ab}} \]
   for the Artin reciprocity map of class field theory, normalised such that for $x \in \Zhat^\times \subset \AA_f^\times$, $\operatorname{Art}(x)$ acts on roots of unity as $\zeta \mapsto \zeta^x$ (and hence uniformizers map to geometric Frobenius elements). 
  \end{definition}
  
  \begin{proposition}
   All components of $Y_G(\CC)$ are defined over the cyclotomic extension $\QQ^{\mathrm{ab}} = \QQ(\zeta_n : n \ge 1)$, and the right-translation action of $u \in G(\AA_f)$ on $\pi_0(Y_G(\CC))$ coincides with the action of the Galois automorphism $\operatorname{Art}(\mu(u)^{-1})$.
  \end{proposition}
  \begin{proof}
   This is an instance of Deligne's reciprocity law for the action of Galois on the connected components of any Shimura variety; see e.g.~\cite[\S 13]{milne-shimura}.
  \end{proof}
  We will be particularly interested in the following special case. If $U \subset G(\AA_f)$ is an open compact subgroup, and $V_N$ is the subgroup of $\Zhat^\times$ defined by $\{ x: x = 1 \bmod N\}$ for some integer $N$, then there is an embedding of $\QQ$-varieties
  \[ Y_G\left(U \cap \mu^{-1}(V_N)\right) \into Y_G(U) \mathop{\times}_{\Spec \QQ} \Spec \QQ(\zeta_N), \]
  which is an isomorphism if $\mu(U)$ surjects onto $(\ZZ / N\ZZ)^\times$. This map intertwines the action of $g \in G(\AA_f)$ on the left-hand side with that of $(g, \sigma)$ on the right-hand side, where $\sigma$ is the image of $\operatorname{Art}(\mu(g)^{-1})$ in $\Gal(\QQ(\zeta_N) / \QQ)$.

\section{Coefficient sheaves on modular varieties}

 \subsection{\'Etale coefficient sheaves}

  Let $U \subset \GL_2(\hat\ZZ)$ be a sufficiently small open compact subgroup, and $S$ a finite set with a continuous left action of $U$. Then we may construct a finite \'etale covering of $Y(U)$ as follows: we take any open normal subgroup $V \trianglelefteqslant U$ acting trivially on $S$, and we let $\mathscr{S}$ be the quotient of $Y(V) \times S$ by the left action of $U / V$ given by
  \[ h \cdot (y, s) = (yh^{-1}, hs). \]
  If $S$ is a $\ZZ[U]$-module, then $\mathscr{S}$ can be considered as a locally constant \'etale sheaf of abelian groups over $Y(U)$. Note that the sections of $\mathscr{S}$ over $Y(V)$, for any $V  \trianglelefteqslant U$ open, are canonically isomorphic to $S^V$, and the pullback action of $u \in U/V$ on $H^0(Y(V), \mathscr{S})$ is identified with the native left action of $U/V$ on $S^V$. This construction extends in the obvious fashion to profinite modules $S$, and in particular to continuous representations of $U$ on finite-rank $\Zp$-modules; via passage to the isogeny category we may also allow $S$ to be a $\Qp$-vector space.

  If the action of $U$ on $S$ extends to some larger monoid $\mathcal{M} \subseteq \GL_2(\AA_f)$ containing $U$, then the sheaf $\mathscr{S}$ naturally becomes a $\mathcal{M}$-equivariant sheaf. That is, for every $\sigma \in \mathcal{M}$, giving a morphism of varieties $Y(U) \rTo^{\sigma} Y(\sigma^{-1} U \sigma)$, we have morphisms $\sigma^*(\mathscr{S}') \to \mathscr{S}$, where $\mathscr{S}$ and $\mathscr{S}'$ are the sheaves on $Y(U)$ and $Y(\sigma^{-1} U \sigma)$, respectively, corresponding to $S$; and these morphisms satisfy an appropriate cocycle condition. This construction equips the cohomology groups $H^*(Y(U), \mathscr{S})$ with an action of the Hecke algebra $\cH(U \backslash \mathcal{M} / U)$.
  
  \begin{remark}
   Compare \cite[Proposition 4.4.3]{loefflerzerbes16}; our conventions here are a little different as we are considering right, rather than left, actions on our Shimura varieties.
  \end{remark}
  
  Exactly the same theory applies, of course, to the modular varieties $Y_G(U)$ and $Y_H(U)$, and these constructions are compatible via $\iota$: the pullback functor $\iota^*$ on \'etale sheaves corresponds to restriction of representations from $G$ to $H$.

 \subsection{Sheaves corresponding to algebraic representations}
  \label{sect:ancona}

  As we have seen above, the modular curves $Y(U)$, for $U$ sufficiently small, are moduli spaces: $Y(U)$ parametrises elliptic curves $E$ equipped with a $U$-orbit of isomorphisms $E_{\mathrm{tors}} \cong (\QQ / \ZZ)^2$. Thus $Y(U)$ comes equipped with a universal elliptic curve $\sE$. From the description of the action of $\GL_2(\ZZ / N \ZZ)$ on the moduli problem, one deduces the following compatibility:

  \begin{lemma}
   Suppose $U \subseteq \GL_2(\hat \ZZ)$. For $N \ge 1$, the sheaf $\sE[N]$ of $N$-torsion points of $\sE$ is canonically isomorphic to the sheaf associated to the \textbf{dual} of the standard representation of $\GL_2(\ZZ / N\ZZ)$.
  \end{lemma}

  Let $p$ be prime. Taking $N = p^r$ and passing to the limit over $r$ shows that the relative Tate module $T_p\sE$ corresponds to the dual of the standard representation of $\GL_2(\Zp)$. On the other hand, $T_p \sE$ is a lattice in the $p$-adic \'etale realisation of a ``motivic sheaf'' -- a relative Chow motive -- over $Y(U)$, namely $h^1(\sE)^\vee$. 
  
  This is the first instance of a more general phenomenon. Let $\sG$ temporarily denote any of  the three groups $\left\{ \GL_2, \GL_2 \times_{\GL_1} \GL_2, \GSp_4\right\}$; and let $U$ be a sufficiently small open compact in $\sG(\AA_f)$, so we have an associated Shimura variety $Y_{\sG}(U)$.

  \begin{lemma}[{\cite[Theorem 8.6]{ancona15}}]
   There is an additive functor 
   \[ \Anc_{\sG, U}: \operatorname{Rep}(\sG) \to \operatorname{CHM}(Y_{\sG}(U))\]
   from the category of representations of $\sG$ over $\QQ$ to the category of relative Chow motives over $Y_{\sG}(U)$ with the following properties:
   \begin{itemize}
    \item $\Anc_{\sG, U}$ preserves tensor products and duals;
    \item if $\mu$ denotes the multiplier map $\sG \to \mathbf{G}_m$, then $\Anc_{\sG, U}(\mu)$ is the Lefschetz motive $\QQ(-1)$;
    \item if $V$ denotes the defining representation of $\sG$, then $\Anc_{\sG, U}(V) = h^1(\sA_U)$, where $\sA_U$ is the universal PEL abelian variety over $Y_{\sG}(U)$;
    \item for any prime $p$ and $\sG$-representation $V$, the $p$-adic realisation of $\Anc_{\sG, U}(V)$ is the \'etale sheaf associated to $V \otimes \Qp$, regarded as a left $U$-representation via $U \into \sG(\AA_f) \onto \sG(\Qp)$.
   \end{itemize}
  \end{lemma}

  \begin{remark}
   In fact Ancona's construction is much more general, applying to arbitrary PEL Shimura varieties, but we shall only need the above three groups here. The theorem stated in \emph{op.cit.} is slightly different from ours, since he normalises his functor to send the multiplier representation to $\QQ(1)$, and the defining representation to $h^1(\sA)^\vee$; our functor is obtained from his by composing with the automorphism of $\operatorname{Rep}(\sG)$ induced by the map $g \mapsto \mu(g)^{-1} g$ on $\sG$.
  \end{remark}

  We shall need some ``naturality'' properties of Ancona's construction, which we now recall.
  
  \begin{proposition}
   \label{prop:ancfunctorial}
   Suppose $U, U'$ are open subgroups of $\sG(\AA_f)$ and $\sigma \in \sG(\AA_f)$ is such that $\sigma^{-1} U \sigma \subseteq U'$, so that right-translation defines a map $\sigma: Y_{\sG}(U) \to Y_{\sG}(U')$. Then we have isomorphisms of functors
   \[ \sigma^* \circ \Anc_{\sG, U'} \cong \Anc_{\sG, U}, \]
   satisfying a suitable compatibility condition under composition.
  \end{proposition}
  
  \begin{proof}
   Since $\Anc_{\sG, U}(V)$ for a general $V$ is defined as a direct summand of a tensor power of $h^1(\sA_U)$, one reduces easily to checking this functoriality property for the specific relative motives $h^1(\sA_U)$.
   
   By a standard argument (see e.g.~\cite[Prop 3.3]{deligne69}) one can interpret $Y_{\sG}(U)$ as a moduli space for abelian varieties \emph{up to isogeny}, from which we can deduce that there is a canonical isomorphism
   \[ \lambda_\sigma: \sA_U \otimes \QQ \rTo^\cong \sigma^*\left(\sA_{U'}\right)\otimes \QQ \]
   in the isogeny category of abelian varieties over $Y_{\sG}(U)$. The functor $h^1(-)$ extends to the isogeny category of abelian varieties, so $\lambda_{\sigma}$ induces an isomorphism of relative Chow motives. Moreover, the $\lambda_{\sigma}$ satisfy a cocycle condition for varying $\sigma$ (which we leave it to the reader to formulate explicitly).
  \end{proof}
   
  \begin{remark}
   Note that $\lambda_{\sigma}$ comes from a ``genuine'' isogeny if and only if the matrix of $\sigma$ (in the defining matrix representation of $\sG$) has entries in $\hat{\ZZ}$. To fix conventions, note that if $U = U'$ and $\sigma = \operatorname{diag}(x, \dots, x)$ for some $x \in \QQ^{\times +}$, then the map $\lambda_{\sigma}$ is given by multiplication by $x$.
  \end{remark}
 
  An equivalent way of stating Proposition \ref{prop:ancfunctorial} is as follows. We define a \emph{$\sG(\AA_f)$-equivariant relative Chow motive} over $Y_\sG$ to be the data of a relative Chow motive $\sV_U$ over $Y_\sG(U)$ for each sufficiently small open $U \subset G(\AA_f)$, together with a collection of isomorphisms $\sigma^*(\sV_{U'}) \cong \sV_U$ for each inclusion $\sigma^{-1} U \sigma \subseteq U'$, compatible with composition. These objects form a category $\operatorname{CHM}(Y_{\sG})^{\sG(\AA_f)}$, and the proposition states that the functors $\Anc_{\sG, U}$ for varying $U$ assemble into a functor 
  \[ \Anc_{\sG}: \operatorname{Rep}(\sG) \rTo \operatorname{CHM}(Y_{\sG})^{\sG(\AA_f)}. \]
  
  (Note that the isomorphisms $\Anc_{\sG, U}(\mu) \cong \QQ(-1)$ are not compatible with the equivariant structure, since the isogenies $\lambda_{\sigma}$ only preserve the polarisation up to a scalar; as equivariant motives we have $\Anc_{\sG}(\mu) = \QQ(-1)[-1]$, where the notation $[-1]$ denotes that the equivariant structure is twisted by the character $\|\mu\|^{-1}$ of $G(\AA_f)$.)
  
  If $\sV$ is a $\sG(\AA_f)$-equivariant relative Chow motive over $Y_\sG$, we can define its motivic cohomology by
  \begin{equation}
   \label{eq:directlimit}
    H^*_{\mot}(Y_\sG, \sV) \coloneqq \varinjlim_UH^*_\mot(Y_\sG(U), \sV_U),
  \end{equation}
  and this is naturally a smooth representation of $\sG(\AA_f)$. As motivic cohomology with rational coefficients satisfies Galois descent (see e.g.~\cite[\S 1.3]{deningerscholl91}), for each sufficiently small $U$ we can recover $H^*_\mot\left(Y_\sG(U),\sV_U\right)$ as the $U$-invariants of the direct limit \eqref{eq:directlimit}.

  Finally, we shall need to show a compatibility with respect to changing $\sG$:
  
  \begin{proposition}[``Branching'' for motivic sheaves]
   \label{prop:branching}
   Let $G = \GSp_4$ and $H = \GL_2 \times_{\GL_1} \GL_2$, as in \S \ref{sect:groups} above. Then there is a commutative diagram of functors
   \[
    \begin{diagram}
     \operatorname{Rep}(G) & \rTo^{\Anc_G} & \operatorname{CHM}(Y_G)^{G(\AA_f)} \\
     \dTo<{\iota^*} & & \dTo>{\iota^*} \\
     \operatorname{Rep}(H) & \rTo^{\Anc_H} & \operatorname{CHM}(Y_H)^{H(\AA_f)}
    \end{diagram}
   \]
   where the left-hand $\iota^*$ denotes restriction of representations, and the right-hand $\iota^*$ denotes pullback of relative motives.
  \end{proposition}

  \begin{proof}
   This is an instance of a general theorem due to Torzewski \cite[Corollary 9.8]{torzewski18}, which verifies the above naturality property for a wide class of homomorphisms of PEL-type Shimura data.
  \end{proof}

\section{Eisenstein classes for \texorpdfstring{$\GL_2$}{GL(2)}}

 \subsection{Modular units}
  \label{sect:modunits}
  
  Let $\cS_0(\AA_f^2, \QQ) \subseteq \cS(\AA_f^2, \QQ)$ denote the subspace of functions satisfying $\phi(0,0) = 0$. Recall that $Y$ denotes the infinite-level modular curve, so that $\cO(Y) = \varinjlim_U \cO(Y(U))$.

  \begin{proposition}
   There is a canonical, $\GL_2(\AA_f)$-equivariant map $\cS_0(\AA_f^2, \QQ) \to \cO(Y)^\times \otimes \QQ$, $\phi \mapsto g_\phi$, with the following characterising property: if $\phi$ is the characteristic function of $(a, b) + N\hat\ZZ^2$, for some $N \ge 1$ and $(a, b) \in \QQ^2 - N\ZZ^2$, then $g_\phi$ is the Siegel unit $g_{a/N, b/N}$ in the notation of \cite[\S 1.4]{kato04}.
  \end{proposition}
  
  \begin{proof} See e.g.~\cite[Th\'eor\`eme 1.8]{colmezBSD}.
  \end{proof}
 
  In order to work integrally, we need to modify the construction somewhat. Let $c > 1$ be an integer. We let ${}_c \cS_0(\AA_f^2, \ZZ)$ denote the subgroup of $\cS_0(\AA_f^2, \QQ)$ consisting of functions of the form $\phi = \phi^{(c)}\cdot \ch(\ZZ_c^2)$, where $\phi^{(c)}$ is a $\ZZ$-valued Schwartz function on $(\AA_f^{(c)})^2$, and $\ZZ_c = \prod_{\ell \mid c} \Zl$. Then we have the following refinement:

  \begin{proposition}
   \label{prop:integral-siegel}
   If $c$ is coprime to 6, there is a map ${}_c \cS_0(\AA_f^2, \ZZ) \to \cO(Y)^\times$, $\phi \mapsto {}_c g_\phi$, which is equivariant for the action of $\GL_2\left(\AA_f^{(c)}\right)$ and satisfies
   \[ {}_c g_\phi \otimes 1 = \left( c^2 - {\stbt{c}{0}{0}{c}}^{-1}\right) g_\phi \quad\text{as elements of}\quad \cO(Y)^\times \otimes \QQ, \]
   where $\stbt{c}{0}{0}{c}$ is understood as an element of $\GL_2(\AA_f^{(c)})$.
  \end{proposition}
 
  \begin{proof}
   See \cite[\S 1.4]{kato04}.
  \end{proof}

 \subsection{Higher Eisenstein classes}
  \label{sect:higherEis}
  
  \begin{definition}
   For $k \ge 0$, let $\sH^k_{\QQ}$ denote the $\GL_2(\AA_f)$-equivariant relative Chow motive over $Y$ associated to the representation $\Sym^k(\mathrm{std}) \otimes \det^{-k}$ of $\GL_2 / \QQ$.
  \end{definition}

  \begin{theorem}[Beilinson]
   Let $k \ge 1$. There is a $\GL_2(\AA_f)$-equivariant map $\cS(\AA_f^2, \QQ) \to H^1_{\mot}\left(Y, \sH^k_{\QQ}(1)\right)$, $\phi \mapsto \Eis^k_{\mot, \phi}$, the \emph{motivic Eisenstein symbol}, with the following property: the pullback of the de Rham realization $r_{\dR}\left( \Eis^k_{\mot, \phi} \right)$ to the upper half-plane is the $\sH^k$-valued differential 1-form
   \[ -F^{(k+2)}_{\phi}(\tau) (2\pi i \dd  z)^k (2\pi i\dd \tau), \]
   where $F^{(k+2)}_{\phi}$ is the Eisenstein series defined by 
   \[ 
    F^{(k+2)}_\phi(\tau) = \frac{(k+1)!}{(-2\pi i)^{k+2}}  \sum_{\substack{x, y \in \QQ \\ (x, y) \ne (0,0)}} \frac{\hat\phi(x, y)}{(x\tau + y)^{k+2}}.
   \]
  \end{theorem}
 
  \begin{proof}
   See \cite{beilinson86}.
  \end{proof}

  \begin{remark}
   Note that if $\phi$ is the characteristic function of $(0, b) + N\hat\ZZ^2$, then $\Eis^k_{\mot, \phi}$ is the class $\Eis^k_{\mot, b, N}$ defined in \cite[Theorem 4.1.1]{KLZ1a}. If $k = 0$, then we need to assume $\phi \in \cS_0(\AA_f^2, \QQ)$ in order for the series defining $F^{(2)}_{\phi}$ to be absolutely convergent. With this assumption, we may define $\Eis^0_{\mot, \phi}$ to be the unit $g_{\phi}$, since $H^1_{\mot}(Y, \QQ(1)) = \cO^\times(Y) \otimes \QQ$; the de Rham realisation of this class is then $\operatorname{dlog} g_{\phi} = - F^{(2)}_{\phi} \cdot (2\pi i \dd\tau)$, so our statements are consistent.
  \end{remark}
  
  Since we lack a good theory of relative Chow motives with coefficients in $\ZZ$, we do not have an integral version of the motivic Eisenstein classes for $k > 0$. However, we can obtain a $\Zp$-structure (for a fixed $p$) using \'etale cohomology instead. For each $U$ we have an \emph{\'etale realisation} map
  \[ r_{\et}: H^1_{\mot}\left(Y(U), \sH^k_{\QQ}(1)\right) \to H^1_{\et}\left(Y(U), \sH^k_{\Qp}(1)\right) \]
  for any prime $p$, where $H^*_\et$ denotes continuous \'etale cohomology in the sense of \cite{jannsen88}, and $\sH^k_{\Qp}$ is the lisse \'etale $\Qp$-sheaf which is the $p$-adic realisation of $\sH_{\QQ}$. This is naturally the base-extension to $\Qp$ of the \'etale $\Zp$-sheaf $\sH_{\Zp}^k$ associated to the \emph{minimal} admissible lattice in the $\GL_2$-representation $\Sym^k(\operatorname{std}) \otimes \det^{-k}$.
  
  \begin{proposition}
   \label{prop:integral-Eis}
   Let $k \ge 0$. If $c$ is coprime to $6p$, then for each sufficiently small open compact $U \subset \GL_2\left(\AA_f^{(pc)} \times \ZZ_{pc}\right)$ there is a map 
   \[ 
    {}_c \cS\left( (\AA_f^{(p)} \times \Zp)^2, \Zp\right)^U \to H^1_{\et}\left(Y(U), \sH^k_{\Zp}(1)\right),\quad \phi \mapsto \cEis^k_{\et, \phi},
   \]
   which is equivariant for the action of $\GL_2\left(\AA_f^{(pc)} \times \ZZ_{pc}\right)$, and satisfies
   \[  \cEis^k_{\et, \phi} \otimes 1 = \left( c^2 - c^{-k} {\stbt{c}{0}{0}{c}}^{-1}\right) r_{\et}\left(\Eis^k_{\mot, \phi}\right) \quad\text{as elements of}\quad H^1_{\et}\left(Y(U), \sH^k_{\Qp}(1)\right). \]
  \end{proposition}
  
  Note that \'etale cohomology with $\Zp$ coefficients does not satisfy Galois descent, so it is important to formulate Proposition \ref{prop:integral-Eis} for each individual level, rather than simply passing to the direct limit.
  
  \begin{proof} For levels of the form $U(N)$ this is explained in \cite{kings15}, and the arguments apply without change to a general $U$. \end{proof}
  
  \begin{remark}
   For $k = 0$, $\sH^k_{\Zp}$ is the constant sheaf $\Zp$, and of course $\cEis^0_{\et, \phi}$ is the image of ${}_cg_{\phi}$ under the Kummer map, so the $k = 0$ case of Proposition \ref{prop:integral-Eis} is consistent with Proposition \ref{prop:integral-siegel}.
  \end{remark}
  
 \subsection{The modular unit representation}

  We will need a description of the modular units $\cO^\times(Y) \otimes \CC$ as a $\GL_2(\AA_f)$-representation. 
  
  \begin{definition}
   For $k \ge 0$, and $\eta$ a finite-order character of $\AA_f^\times / \QQ^{\times +}$ satisfying $\eta(-1) = (-1)^k$, let $I_k(\eta)$ denote the space of functions $f: \GL_2(\AA_f) \to \CC$ satisfying 
   \[ f\left( \stbt{a}{b}{}{d} g\right) = \|a\|^{k+1} \|d\|^{-1} \eta(a) f(g)\qquad \text{for all $a, b, d \in \GL_2(\AA_f)$},\]
   regarded as a representation of $\GL_2(\AA_f)$ by right translation. For $k = 0$ and $\eta = 1$, let $I_0^0(1)$ denote the subrepresentation which is the kernel of the natural map $I_0(1) \to \CC$ given by integration over $\GL_2(\AA_f) / B(\AA_f)$.
  \end{definition}
  
  \begin{theorem}
   There is a $\GL_2(\AA_f)$-equivariant isomorphism
   \[ 
    \partial_0: \frac{\cO^\times(Y)}{(\QQ^{\mathrm{ab}})^\times} \otimes \CC
    \rTo^\cong
    I_0^0\left(1\right)
    \oplus \bigoplus_{\eta \ne 1} I_0(\eta),
   \]
   characterised by the statement that if $g \in \cO^\times(Y)$, then $\partial_0(g)(1)$ is the order of vanishing of $g$ at the cusp $\infty$.
  \end{theorem}
 
  \begin{proof} See \cite[Theorem 3]{scholl89}. (Scholl's normalisations are slightly different from ours, as he uses the canonical model of $Y$ for a different choice of Shimura datum; see Remark \ref{rmk:shdatum}. The above formulation is correct for our choice of model.)
  \end{proof}
 
  For $k \ge 1$ we have an analogous statement for the image of the Eisenstein symbol, although we do not know if this image is the whole of the motivic cohomology:
  
  \begin{theorem}
   For $k \ge 1$, there is a surjective $\GL_2(\AA_f)$-equivariant map
   \[ \partial_k: H^1_{\mot}\left(Y, \sH^k_\QQ(1)\right) \otimes \CC \onto \bigoplus_{\eta} I_k(\eta), \]
   such that $\partial_k(x)(1)$ is the residue at $\infty$ of the 1-form $r_{\dR}(x)$. This map is an isomorphism on the image of the Eisenstein symbol $\phi \mapsto \Eis_{\mot, \phi}^k$.
  \end{theorem}
  
  \begin{proof}
   It is shown in \cite[Theorem 7.4]{schappacherscholl91} that the residue map $\partial_k$ gives an isomorphism between the image of the Eisenstein symbol and a certain vector space denoted $\mathcal{B}_k$. For the description of this $\mathcal{B}_k$ as a sum of induced representations, see \cite[Lemma 4.3]{lemma17}.
  \end{proof}
    
  For $\eta$ a character of $\AA_f^\times / \QQ^{\times +}$ as above, let us write $\cS(\AA_f^2, \CC)^\eta$ for the subspace of $\cS(\AA_f^2, \CC)$ on which $\Zhat^\times$ acts via the character $\eta$.
  
  \begin{proposition}
   Let $\phi \in \cS(\AA_f^2, \CC)^\eta$ be of the form $\prod_{\text{$\ell$ prime}} \phi_\ell$. If $k = 0$ and $\eta = 1$ then assume we have $\phi(0, 0) = 0$. Then we have
   \[ 
    \partial_k\left( \Eis^k_{\mot, \phi}\right) = \frac{2(k+1)! L(k+2, \eta)}{(-2\pi i)^{k+2}} \prod_{\ell} f_{\hat\phi_\ell, \eta_\ell |\cdot|^{k+1/2}, |\cdot|^{-1/2}}.
   \]
  \end{proposition}
  
  \begin{proof}
   This follows from a computation of the constant term of the Eisenstein series $F^{(k+2)}_{\phi}$ at the cusp $\infty$, using the formulae in \cite[Proposition 3.10]{kato04}.
  \end{proof}

\section{Construction of Lemma--Eisenstein classes}
 
 \subsection{Coefficients}
  
  Let $(a, b) \ge 0$ be a pair of non-negative integers, defining an algebraic representation $V^{a,b}$ of $G$. We write $D^{a,b}$ for the twist $V^{a,b} \otimes \mu^{-(2a+b)}$. We then have an equivariant relative Chow motive $\sD_{\QQ}^{a,b} = \Anc_G(D^{a,b})$ over the Shimura variety $Y_G$.
  
  \begin{notation}
   Let us choose integers $(q, r)$ with $0 \le q \le a$ and $0 \le r \le b$, and set $c = (a-q) + (b-r)$, $d = (a-q)+r$ (so $c$, $d \ge 0$).
  \end{notation}
  
  Then there is a branching map
  \[ 
   \br^{[a,b,q,r]}: (\Sym^c \boxtimes \Sym^d) \otimes \det{}^{-(c+d)} \to D^{a,b} \otimes \mu^q
  \]
  as in \eqref{eq:brab}. Via the commutative diagram of functors in Proposition \ref{prop:branching}, we have a homomorphism of equivariant Chow motives over $Y_H$,
  \[ \br^{[a,b,q,r]}: \sH^{c, d}_{\QQ} \to \iota^*\left( \sD^{a,b}_{\QQ}(-q)[-q] \right), \]
  where $\sH^{c, d}_{\QQ}$ is the relative motive associated to the $H$-representation $(\Sym^c \boxtimes \Sym^d) \otimes \det{}^{-(c+d)}$. (Recall that the notation $[m]$ denotes twisting by the character $\|\mu(-)\|^m$ of $G(\AA_f)$.)

 \subsection{Pushforwards in motivic cohomology}
 
  Let $(a,b,q,r)$ and $(c, d)$ be as in the previous section. We shall define in this section a homomorphism of left $H(\AA_f)\times G(\AA_f)$-representations
  \begin{equation}
  \label{eq:iota}
  \iota^{[a,b,q,r]}_*:
  H^2_{\mot}\left(Y_H, \sH^{c,d}_{\QQ}(2)\right) \otimes \cH(G(\AA_f);\QQ) \rightarrow
  H^4_\mot\left(Y_G, \sD^{a,b}_\QQ(3-q)\right)[-q].
  \end{equation}
  The actions of $H(\AA_f)\times G(\AA_f)$ for which this map is equivariant are given as follows:
  \begin{itemize}
   \item The $H(\AA_f)$ factor acts trivially on the right-hand side of \eqref{zetamap}, and on the left-hand side it acts via the formula
   \[ h\cdot(x\otimes\xi) = (h\cdot x) \otimes \xi(h^{-1}(-)).\]
   \item On the left-hand side, the $G(\AA_f)$ factor acts trivially on $H^2_{\mot}\left(Y_H, \sH^{c,d}_{\QQ}(2)\right)$, and on $\cH(G(\AA_f); \ZZ)$ it acts via $g \cdot \xi = \xi( (-) g)$.
   \item On the right-hand side, $G(\AA_f)$ via its natural action on $H^4_\mot\left(Y_G(U),\sD^{a,b}_\QQ(3-q)\right)$ (deduced from Proposition \ref{prop:ancfunctorial}) twisted by the character $\|\mu(-)\|^{-q}$.
  \end{itemize}
  
  For ease of reading we shall drop the superscripts $[a,b,q,r]$ for the rest of this section.
  
  \begin{lemma}
   \label{lem:austrianspan}
   The space $H^2_{\mot}\left(Y_H, \sH^{c,d}_{\QQ}(2)\right) \otimes \cH(G(\AA_f);\QQ)$ is spanned by vectors of the form $x \otimes \ch(g U)$, where $x \in H^2_{\mot}\left(Y_H, \sH^{c,d}_{\QQ}(2)\right)$, $g \in G(\AA_f)$, and $U$ is an open compact subgroup of $G(\AA_f)$ such that $H(\AA_f) \cap g Ug^{-1}$ fixes $x$ and $g Ug^{-1}$ is $H$-small in the sense of \S \ref{sect:Hsmall}.
  \end{lemma}
  
  \begin{proof}
   This is immediate from the fact that the principal congruence subgroup $U_G(N)$, for any $N \ge 3$, is $H$-small, and these are cofinal among open compact subgroups of $G(\AA_f)$.
  \end{proof}
  
  For an element $x \otimes \ch(Ug)$ as in Lemma \ref{lem:austrianspan}, we have a closed immersion $\iota_{gU}: Y_H(V) \into Y_G(U)$, where $V = U \cap H(\AA_f)$, given by the composite
  \[ Y_H(V) \rTo^{\iota_{g U g^{-1}}} Y_G(gUg^{-1}) \rTo^{g^{-1}} Y_G(U).\]
  Combining this with the morphism of sheaves $\br^{[a,b,q,r]}$ gives a map
  \[ \iota_{gU,*}: H^2_\mot\left(Y_H(V),\sH^{c,d}_\QQ(2)\right)
  \to H^4_{\mot}\left(Y_G(U), \sD^{a,b}_{\QQ}(3-q)\right)[-q]. \]
  We define $\iota_*(x \otimes \ch(gU))$ as the image of the element
  \[ \vol(V) \cdot \iota_{gU,*}\left(x\right) \in H^4_\mot\left(Y_G(U),\sD^{a,b}_{\QQ}(3-q)\right)[-q] \]
  in the direct limit \eqref{eq:directlimit}. (Here $\vol(V)$ is volume with respect to Haar measure, normalised such that $\vol H(\hat{\ZZ}) = 1$.)
  
  Now, suppose $U' \subset U$ is another $H$-small open compact subgroup, so that $\ch(g U) = \sum_{\gamma \in U / U'} \ch(g \gamma U')$. We want to show that $\iota_*(x \otimes \ch(g U)) = \sum_{\gamma \in U / U'} \iota_*(x \otimes \ch(g \gamma U'))$ for any $x$ invariant under $V$. It suffices to prove this when $U' \trianglelefteqslant U$ (since otherwise we may compare both $U$ and $U'$ with a third open compact $U''$ normal in both $U$ and $U'$); we may clearly also assume $g = 1$. 
  
  Let $V' = U' \cap H(\AA_f)$; we then have degeneracy maps $\pr_U^{U'}: Y_G(U') \to Y_G(U)$ and $\pr_V^{V'}: Y_H(V') \to Y_H(V)$, fitting into a commutative diagram
  \begin{diagram}[small]
   Y_H(V') & \rTo^{\iota_{U'}} & Y_G(U') \\
   \dTo<{\pr_V^{V'}} & & \dTo>{\pr_U^{U'}} \\
   Y_H(V) & \rTo_{\iota_U} & Y_G(U).
  \end{diagram} 
  By the functoriality of the pushforward maps, we have
  \[ (\pr_U^{U'})^* \circ (\pr_U^{U'})_* \circ \iota_{U',*} = (\pr_U^{U'})^* \circ \iota_{U, *} \circ (\pr_V^{V'})_*. \]
  The composite $(\pr_U^{U'})^*(\pr_U^{U'})_*$ is given by $\sum_{\gamma \in U/U'} \gamma^*$; and as $x$ is invariant under $V$ (not only $V'$), then we have $(\pr_V^{V'})_*(x) = [V : V'] x$. So we can write this as
  \[ \vol(V') \cdot \sum_{\gamma \in U/U'} \gamma^* \iota_{U',*}(x) = \vol(V) \cdot (\pr_U^{U'})^*\iota_{U, *}(x). \]
  Pulling back from level $U'$ to the direct limit over all levels, we can write this as
  \[ \iota_*(x \otimes \ch(U)) = \sum_{\gamma \in U / U'} \iota_*(x \otimes \ch(\gamma U')). \]
  
  It follows that $\iota_*$ is well-defined on $H^2_{\mot}\left(Y_H, \sH^{c,d}_{\QQ}(2)\right) \otimes \cH(G(\AA_f);\QQ)$. It is obvious that this map is $G(\AA_f)$-equivariant; and the $H(\AA_f)$-equivariance follows from the obvious compatibility
  \[ \iota_{U, *}(h \cdot x) = h \cdot \iota_{h U h^{-1}, *}(x) \]
  of pushforward maps at finite level.

 \subsection{The Lemma--Eisenstein map}
  \label{sect:LEconstruct}
  
  The cup-product of the Eisenstein symbols for the two factors of $H$ defines an $H(\AA_f)$-equivariant map
  \[ \cS_0(\AA_f^2;\QQ)^{\otimes 2} \to H^2_{\mot}\left(Y_H, \sH^{c,d}_{\QQ}(2)\right),\quad \uphi \mapsto \Eis^{c,d}_{\mot, \uphi}.\]
  
  \begin{definition}
   We define the \emph{Lemma--Eisenstein map} 
   \begin{equation}
    \label{zetamap}
    \LE^{[a,b,q,r]}:
    \cS_0(\AA_f^2;\QQ)^{\otimes 2}\otimes\cH(G(\AA_f);\QQ) \rightarrow
    H^4_\mot\left(Y_G, \sD^{a,b}_\QQ(3-q)\right)[-q]
   \end{equation}
   by $\LE^{[a,b,q,r]}(\uphi \otimes \xi) = \iota_*^{[a,b,q,r]}\left( \Eis^{c,d}_{\mot, \uphi} \otimes \xi\right)$, where $\iota_*^{[a,b,q,r]}$ is as in \eqref{eq:iota}.
  \end{definition}

  \begin{remark}
   When $\xi$ is the characteristic function of an open compact subgroup $U \subseteq G(\AA_f)$, our class $\LE^{[a,b,q,r]}_{U}(\uphi \otimes \xi)$ coincides with the motivic cohomology class $\Eis^{m, n, W}_{\mathcal{M}}(h)$ considered in \cite{lemma15}, for $S = Y_G(U)$, $(m, n) = (a-q+b-r, a-q+r)$, $W$ the representation $V^{a,b}$, and $h$ an appropriate element of Lemma's space $\mathcal{B}_m \otimes \mathcal{B}_n$ depending on $\uphi$. In particular, it follows from the regulator computations of \cite[\S 7]{lemma17} that the Lemma--Eisenstein map is non-zero under fairly mild hypotheses on $a,b,q,r$.
  \end{remark}
  
 \subsection{Choices of the local data}
  \label{sect:localdata}
  
  We shall now fix choices of the input data to the above map $\LE^{[a,b,q,r]}$, in order to define a collection of motivic cohomology classes satisfying appropriate norm relations (a ``motivic Euler system''). We shall work with arbitrary (but fixed) choices of local data at the bad primes; it is the local data at good primes which we shall vary, according to the values of three parameters $M, m, n$.

  \subsubsection{Subgroups of tame level 1}

   We fix a prime $p$, a finite set of primes $S$ not containing $p$, and an arbitrary open compact subgroup $K_S \subset G(\QQ_S) = \prod_{\ell \in S} G(\Ql)$. By enlarging $S$ and shrinking $K_S$ if necessary, we may assume that the open compact subgroup
   \[ K_G = K_S \times \prod_{\ell \notin S} G(\Zl) \subseteq G(\AA_f)\]
   is \emph{sufficiently small} in the sense of \S\ref{sect:siegelMV}. For each $n \ge 0$, we define an compact subgroup of $G(\AA_f)$ by
   \[ K_{G, 0}(p^n) \coloneqq K_S \times K_{G_p, 0}(p^n) \times \prod_{\ell \notin S \cup \{p\}} G(\Zl), \]
   where the subgroup $K_{G_p, 0}(p^n)$ of $G_p = G(\Qp)$ is as defined in \S\ref{sect:compacta}. We define similarly subgroups $K_{G, 1}(p^n)$ for $n \ge 0$, and $K_G(p^m, p^n)$ for $m, n \ge 0$, using the other local subgroups at $p$ defined in \S\ref{sect:compacta}. All of these groups are contained in $K_G$, and hence are sufficiently small.
   
   \begin{notation}
    We adopt the notational convention that if $K_{G, \star}(\square)$ denotes some open subgroup of $G(\AA_f)$, then $Y_{G,\star}(\square)$ denotes the corresponding Shimura variety, so e.g.~$Y_{G, 1}(p^n)$ is an abbreviation for $Y_G(K_{G, 1}(p^n))$.
   \end{notation}
   
  \subsubsection{Local data at the bad primes}
  
   We choose the following ``test data'' at $S$:
   \begin{itemize}
    \item A vector $\uphi_S \in \cS(\QQ_S^2, \ZZ)^{\otimes 2}$.
    \item An open compact subgroup $W_S \subset H(\QQ_S)$ such that $W_S \subseteq H(\QQ_S) \cap K_S$, and $W_S$ acts trivially on $\uphi_S$.
   \end{itemize}
   
   Whenever we deal with norm-compatibility relations we shall assume that the local data $K_S$, $W_S$, $\uphi_S$ remains fixed (i.e.~we shall not attempt to formulate any non-trivial norm-compatibilities at the bad primes). Regarding the choice of $\uphi_S$, see Remark \ref{remark:localdata} below.
      
  \subsubsection{Subgroups of higher tame level}
  
   Now let us choose a square-free integer $M \ge 1$ coprime to $S \cup \{p\}$ (which we shall refer to as a ``tame level''). For $m \ge 0$ and $n \ge 1$, we define a subgroup $K_G(M, p^m, p^n) \subseteq K_G(p^m, p^n)$ by
   \[
    K_G(M, p^m, p^n) = \{ k \in K_G(p^m, p^n) : \mu(k) = 1 \bmod M \}.
   \]
   As explained in \S\ref{sect:components}, we have isomorphisms
   \begin{equation}
    \label{eq:components}
    \jmath_{Mp^m}: Y_G(M, p^m, p^n) \rTo^\cong Y_{G, 1}(p^n) \mathop{\times}_{\Spec \QQ} \Spec \QQ(\zeta_{Mp^m}).
   \end{equation}
   Assuming $n \ge m$, we also define $K_G'(M, p^m, p^n) = \{ k \in K_G'(p^m, p^n) : \mu(k) = 1 \bmod M \}$; note that the difference between this and $K_G(M, p^m, p^n)$ is only at $p$ -- we do not impose stronger congruences at $M$.
   
  \subsubsection{Test data of higher level}

   Let $(M, m, n)$ be integers as above. For each such triple, we shall define the following data:
   \begin{itemize}
    \item an element $\xi_{M, m, n} \in \cH(G(\AA_f), \ZZ)$, fixed by the right-translation action of $K_G(M, p^m, p^n)$;
    \item a subgroup $W$ of $H(\AA_f)$, such that for all $x$ in the support of $\xi_{M, m, n}$, we have $W \subseteq H(\AA_f) \cap x K_G(M, p^m, p^n) x^{-1}$;
    \item an element $\uphi_{M, m, n} \in \cS_0(\AA_f^2, \ZZ)^{\otimes 2}$ stable under $W$.
   \end{itemize}
   
   We shall define these as products
   \begin{align*}
    \xi_{M, m, n} &= \ch(K_S) \otimes \bigotimes_{\ell \notin S} \xi_\ell, &
    W &= W_S \times \prod_{\ell \notin S} W_\ell,&
     \uphi_{M, m, n} &= \uphi_S \otimes \bigotimes_{\ell \notin S} \uphi_\ell,
   \end{align*}
   where the local data $K_S, W_S, \uphi_S$ at the bad places are the ones chosen above (independently of $M, m, n$), and the local data at primes $\ell \notin S$ are as follows. As in \S\ref{sect:repth}, we let $\eta_{\ell, r} \in G(\Ql)$ denote the element
   \[
    \begin{smatrix} 1 &  & {\ell^{-r}} &  \\  & 1 &  & \ell^{-r}
    \\  &  & 1 &  \\  &  &  & 1\end{smatrix}.
   \]
   \begin{itemize}
    
    \item If $\ell \nmid Mp$, we set $\xi_\ell = \ch\left(G(\Zl)\right)$, $W_\ell = H(\Zl)$, and $\uphi_\ell = \ch(\Zl^2)^{\otimes 2}$.
     
    \item If $\ell \mid M$, we set $\xi_\ell = \ch(K_{\Gl}(\ell, 1)) - \ch\left(\eta_{\ell, 1} \cdot K_{\Gl}(\ell, 1)\right)$, and $W_\ell = K_{\Hl}(\ell, \ell^2)$. We take $\uphi_\ell = \ch\left( \ell^2 \Zl \times (1 + \ell^2 \Zl)\right)^{\otimes 2}$. 
    
    \item For $\ell = p$, we set $\xi_p = \ch\left(\eta_{p, m} \cdot K_{G_p}(p^m, p^n)\right)$. We choose an integer $t \ge 1$ sufficiently large\footnote{One can check that $t = n + 2m$ suffices.} that $K_{H_p}(p^m, p^t)$ is contained in $\eta_{p, m} \cdot K_{G_p}(p^m, p^n) \cdot \eta_{p, m}^{-1}$; we let $W_p$ be this subgroup, and we define 
    \[ \uphi_p = \ch\left( p^t \Zp \times (1 + p^t \Zp)\right)^{\otimes 2}.\]
   \end{itemize}
   
   Note that $\uphi_{M, m, n} \in \cS_0(\AA_f^2, \ZZ)^{\otimes 2} \subset \cS(\AA_f^2, \ZZ)^{\otimes 2}$, since our local Schwartz functions at $p$ vanish at $(0, 0)$. Both the element $\uphi_{M, m, n}$, and the group $W$, depend on the auxilliary choice of $t$; but if we let $t^\circ > t$ be another choice, and $\uphi_{M, m, n}^\circ$, $W^\circ$ the objects defined using $t^\circ$ in place of $t$, then we have
   \begin{equation} 
    \label{change-t-eq}
    \uphi_{M, m, n} = \sum_{w \in W / W^\circ} w \cdot \uphi_{M,m, n}^\circ.
   \end{equation}
   
   We also define a version mildly modified at $p$, assuming that $n \ge 1$ and $m \le n$. Recall the subgroup $K_{G_p}'(p^m, p^n)$ defined in \S \ref{sect:compacta}. We define $\xi'_p = \ch(\eta_{p, 0} K_{G_p}'(p^m, p^n)) = \ch(K_{G_p}'(p^m, p^n) \eta_{p, 0})$. Thus $\xi_p'$ is preserved under left-translation by $W_p' =K_{H_p}'(p^m, p^n)$; and we choose $\uphi_p' = \ch\left( p^n \Zp \times (1 + p^n \Zp)\right)^{\otimes 2}$.  We let $K'_G(M, p^m, p^n)$, $\xi'_{M, m, n}$, $\uphi_{M, m, n}'$ and $W'$ be the ad\`elic objects defined using these modified choices at $p$, and the same choices as before at all other primes. 
  
  \begin{remark}
   These alternative local choices will give elements related to the ``non-dashed'' versions in the same way as the elements $\cZ_{\dots}$ relate to the elements $\Xi_{\dots}$ in \cite{LLZ}. As in \emph{op.cit.}, it is the non-dashed versions which are of interest for applications, but the dashed versions are convenient for certain calculations, in particular for studying $p$-adic integrality and interpolation properties.
  \end{remark}

  \subsubsection{The Lemma--Eisenstein classes and their norm relations at $p$} 
  
   With the above notations and choices, let us define
   \[
    z^{[a,b,q,r]}_{M, m, n} = \frac{1}{\vol(W)} \LE^{[a,b,q,r]}_{K_G(M, p^m, p^n)}\left(\uphi_{M, m, n} \otimes \xi_{M, m, n}\right) \in H^4_\mot\left(Y_G(M, p^m, p^n), \sD^{a,b}_\QQ(3-q)\right).
   \]
   We refer to these elements as \emph{Lemma--Eisenstein classes}. {\em A priori} this element depends on the auxilliary integer $t$, but it follows readily from \eqref{change-t-eq} that it is in fact independent of this choice (this is essentially the same computation as Lemma \ref{lem:indept}). It can be written concretely as follows: letting $U$ be the subgroup $K_G(M, p^m, p^n)$, we can write our Hecke-algebra element $\xi_{M, m, n}$ as a finite $\ZZ$-linear combination of characteristic functions $\ch(x_i U)$. For each of these terms, if we set $U_i = x_i U x_i^{-1}$, then by hypothesis we have $W \subseteq V_i \coloneqq H \cap U_i$, and we can consider the composition of maps
   \begin{multline}
    \label{eq:explicitLE}
    H^2_{\mot}(Y_H(W), \sH^{c,d}_{\QQ}(2)) \rTo^{(\pr^W_{V_i})_*} H^2_{\mot}(Y_H(V_i), \sH^{c,d}_{\QQ}(2)) \rTo^{\iota^{[a,b,q,r]}_{U_i,*}}  H^4_{\mot}(Y_G(U_i), \sD^{a,b}_{\QQ}(3-q)) \\
    \rTo_{\iota_U}^{x_i}_\cong H^4_{\mot}(Y_G(U), \sD^{a,b}_{\QQ}(3-q)).
   \end{multline}
   (Note that $U_i$ may not be $H$-small, so $\iota_{U_i}: Y_H(V_i) \to Y_G(U_i)$ may not be a closed immersion, but it is still a finite morphism of smooth varieties and this suffices to define the pushforward map $\iota^{[a,b,q,r]}_{U_i,*}$.)
   
   \begin{theorem}
    \label{thm:pnorm}
    The Lemma--Eisenstein classes satisfy the following norm-compatibility relations, as $m$ and $n$ vary:
    \begin{enumerate}[(i)]
     \item 
     For $n \ge 1$, we have
     \[ \left(\pr^{K_G(M, p^m, p^{n+1})}_{K_G(M, p^m, p^n)}\right)_*\left(z^{[a,b,q,r]}_{M, m, n+1}\right) =  z^{[a,b,q,r]}_{M, m, n}.\]
     \item For $m \ge 0$, we have
     \[ \left(\pr^{K_G(M, p^{m+1}, p^{n})}_{K_G(M, p^m, p^n)}\right)_* \left(z^{[a,b,q,r]}_{M, m+1, n}\right) =  
     \left.\begin{cases}
      \frac{U'(p)}{p^q} & \text{if $m \ge 1$} \\
      \left(\frac{U'(p)}{p^q} - 1\right) & \text{if $m=0$}
      \end{cases}\right\} \cdot z^{[a,b,q,r]}_{M, m, n}. 
     \]
    \end{enumerate}
    Here $U'(p) \in \cH(G_p)$ is given by the $K_{G_p}(p^m, p^n)$-double coset of $\dfour{p^{-1}}{p^{-1}}{1}{1}$.
   \end{theorem}
   
   \begin{proof}
    Part (i) of the theorem is immediate from the definition of the classes, since the sum of the translates of $\xi_{M, m, n+1}$ over $K_G(M, p^m, p^n) / K_G(M, p^m, p^{n+1})$ is $\xi_{M, m, n}$.
    
    For part (ii), we note that the Hecke-algebra elements $\xi_{M, m, n}$ and $\xi_{M, m+1, n}$ are identical outside $p$, as are the Schwartz functions $\uphi_{M, m, n}$ and $\uphi_{M, m+1, n}$. So we need to compare two values of a map on $\cS(\Qp^2, \QQ)^{\otimes 2} \otimes \cH(G(\Qp))$ (given by tensoring with the common away-from-$p$ parts and applying $\LE$). It clearly suffices to check the equality after tensoring with $\CC$, which puts us in a position where we may apply Proposition \ref{prop:wildnormrel} (for $\ell = p$). If we assume that the parameters $t$ are chosen identically for the two elements, then the proposition shows that we have
    \[ 
     \left(\pr^{K_G(Mp^{m+1}, p^n)}_{K_G(M, p^m, p^n)}\right)_* \LE \left(\uphi_{M, m+1, n} \otimes \xi_{M, m+1, n} \right) = \left.\begin{cases}\tfrac{1}{p} U'(p) \\ \tfrac{1}{p-1}(U'(p) - 1) \end{cases}\right\} \cdot \LE \left(\uphi_{M, m, n} \otimes \xi_{M, m, n} \right) 
    \]
    as elements of $H^4_\mot\left(Y_G(M, p^m, p^n),\sD^{a,b}_\QQ(3-q)\right)[-q]$. The factor of $\tfrac{1}{p}$ (resp.~$\tfrac{1}{p-1}$) is cancelled out by the factors $\tfrac{1}{\vol(W)}$, since the subgroups $W$ corresponding to the classes at level $Mp^{m+1}$ and $Mp^m$ differ in volume by exactly this factor. Finally, the twist $[-q]$ gives a factor of $p^q$.
   \end{proof}

  \subsubsection{Integral $p$-adic \'etale classes} 
   \label{sect:integralEis}
   
   We now treat questions of integrality. We choose integers $c_1, c_2 > 1$ satisfying the following list of conditions:
   \begin{itemize}
    \item The $c_i$ are coprime to $6p\prod_{\ell \in S} \ell$.
    \item Our chosen vector $\uphi_S \in \cS(\QQ_S^2, \ZZ)^{\otimes 2}$ is preserved by the action of the elements $\left(\stbt{c_1}{}{}{1}, \stbt{c_1}{}{}{1}^{-1}\right)$ and $\left(\stbt{c_2}{}{}{1}, \stbt{c_2}{}{}{1}^{-1}\right)$ of $(\GL_2 \times \GL_2)(\QQ_S)$. (Note that these elements are not in $H$.)
    \item For each $\ell \in S$, the subgroup $K_\ell$ is normalised by the elements $\dfour{1}{c_1}{1}{c_1}$ and $\dfour{c_2}{1}{c_2}{1}$ of $G(\Ql)$.
   \end{itemize}
   (The last two conditions can, of course, always be achieved by taking $c_1$ and $c_2$ to be sufficiently close $\ell$-adically to 1, for all $\ell \in S$.)
   
   Recall the alternative local data $\xi'_{M, m, n}$, $\uphi'_{M, m, n}$, $W'$ introduced at the end of  \S\ref{sect:localdata}.
   
   \begin{definition}
    \label{def:primedclass}
    For $n \ge \max(m, 1)$, let
    \[ {}_{c_1, c_2} \cZ^{[a,b,q,r]}_{\et, M, m, n} \in H^4_{\et}\left(Y_G'(M, p^m, p^n), \sD^{a,b}_{\Zp}(3-q)\right) \]
    be the class defined using the alternative local data $\xi'_{M, m, n}$, $\uphi'_{M, m, n}$, $W'$ in place of their non-dashed versions, and substituting for $\Eis_{\mot, \uphi}^{c,d}$ the integral \'etale Eisenstein classes ${}_{c_1, c_2} \Eis_{\et, \uphi}^{c, d}$.
   \end{definition}
   
   To see that this is well-defined, we use the explicit description of the Lemma--Eisenstein map as a sum of pushforward maps as in \eqref{eq:explicitLE}. Since $\xi'_{M, m, n}$ is a $\ZZ$-linear combination of cosets $x_i U$ (where $U = K'_G(M, p^m, p^n)$) with $x_i \in G(\AA_f^{(p)} \times \Zp)$, the maps $\iota_{x_i U x_i^{-1}, *}^{[a,b,q,r]}$ are well-defined on \'etale cohomology with coefficients in the integral sheaf $\sD_{\Zp}^{a,b}(3-q)$. (Note that $\xi_{M, m, n}$ is not supported in $G(\AA_f^{(p)} \times \Zp)$, which is why we need to introduce the alternative data.)
   
   \begin{definition}
    For $n \ge \max(m, 1)$, let
    \( s_m: Y_G'(M, p^m, p^n) \to Y_G(M, p^m, p^n) \)
    be the map given by the action of $\dfour{p^m}{p^m}{1}{1} \in G(\Qp)$; and let 
    \( s_{m, \sharp}: \sD^{a,b}_{\Zp} \to s_m^* (\sD^{a,b}_{\Zp}) \) be the morphism of sheaves given by the action of $\dfour{p}{p}{1}{1}^{-m}$ on the representation $D^{a,b}_{\Zp}$.
   \end{definition}
   
   Cf.~\cite[\S 6.1]{KLZ1b}. The morphism $s_m$ is well-defined on the Shimura varieties, since we have 
   \[ \dfour {p^m}{p^m}11^{-1}  K'_{G_p}(p^m, p^n) \dfour {p^m}{p^m}11 \subseteq K_{G_p}(p^m, p^n). \]
   To construct the morphism $s_{m, \sharp}$ of integral coefficient sheaves, we note that the representation $D^{a,b}$ of $G$ has all weights $\le 0$ for the torus $T'$ of \S\ref{sect:groups}, so the action of $G(\Zp)$ on $D^{a,b}_{\Zp}$ extends to an action of the monoid generated by $G(\Zp)$ and $\dfour{p}{p}{1}{1}^{-1}$.
   
   \begin{proposition}
    \label{prop:integralclass}
    If $(c_1, c_2)$ satisfy the above conditions, then for any $m,n,M$ as in the previous section with $M$ chosen coprime to $c_1$ and $c_2$, there is a class
    \[ {}_{c_1,c_2} z^{[a,b,q,r]}_{\et, M, m, n} \in H^4_{\et}\Big(Y_G(M, p^m, p^n), \sD_{\Zp}^{a,b}(3-q)\Big) \]
    whose image in the cohomology of $\sD_{\Qp}^{a,b}(3-q)$ is
    \[ p^{mq} \left( c_1^2 - c_1^{-(a-q+b-r)} \dfour{1}{c_1}{1}{c_1}^{-1} \right)\left(c_2^2 - c_2^{-(a-q+r)}\dfour{c_2}{1}{c_2}{1}^{-1}\right) r_{\et}\left(z^{[a,b,q,r]}_{M, m, n}\right), \]
    where the matrices on the left-hand side are understood as elements of $\prod_{\ell \mid Mp S}G(\Zhat)$ acting on $Y_G(M, p^m, p^n)$ by right-translation.
   \end{proposition}
   
   \begin{remark}
    If $a = b = 0$, then $ {}_{c_1,c_2} z^{[0,0,0,0]}_{\et, M, m, n}$ is the image of a class ${}_{c_1, c_2} z^{[0,0,0,0]}_{M, m, n}$ in the motivic cohomology with $\ZZ$ coefficients. However, we do not know how to define this group for $a, b > 0$.
   \end{remark}
   
   \begin{proof}
    Assume for the moment that $m \le n$, and define 
    \[ {}_{c_1,c_2} z^{[a,b,q,r]}_{\et, M, m, n} \coloneqq s_{m, *} \left({}_{c_1, c_2} \cZ^{[a,b,q,r]}_{\et, M, m, n}\right), 
    \]
    where the action of $s_m$ on the coefficients is given by $s_{m, \sharp}$. By definition, this has coefficients in $\sD_{\Zp}^{a,b}(3-q)$, so it remains to verify that is related to $z^{[a,b,q,r]}_{M, m, n}$ in the stated manner. 
    
    A simple check shows that we have
    \[  
     {}_{c_1,c_2} z^{[a,b,q,r]}_{\et, M, m, n} = p^{mq} r_{\et}\left(\tilde{z}^{[a,b,q,r]}_{M, m, n}\right)
    \]
    where $\tilde{z}$ is the class obtained in the same way as $z$, with the Schwartz function $\uphi$ replaced by
    \[ \left(c_1^2 - c_1^{-(a-q+b-r)}(\stbt{c_1}{0}{0}{c_1}, \mathrm{id})^{-1}\middle) \middle(c_2^2 - c_2^{-(a-q+r)} (\mathrm{id}, \stbt{c_2}{0}{0}{c_2})^{-1}\right)\uphi.\] 
    (The factor $p^{mq}$ appears because we ignored the twist $[-q]$ in the definition of $s_{m, \sharp}$.) However, our assumptions on the $c_i$ imply that we have
    \begin{multline*}
    \left(c_1^2 - c_1^{-(a-q+b-r)}(\stbt{c_1}{0}{0}{c_1}, \mathrm{id})^{-1}\middle) \middle(c_2^2 - c_2^{-(a-q+r)} (\mathrm{id}, \stbt{c_2}{0}{0}{c_2})^{-1}\right)\uphi
    \\= 
    \left(c_1^2 - c_1^{-(a-q+b-r)}(\stbt{1}{0}{0}{c_1}, \stbt{c_1}{0}{0}{1})^{-1}\middle) \middle(c_2^2 - c_2^{-(a-q+r)} (\stbt{c_2}{0}{0}{1}, \stbt{1}{0}{0}{c_2})^{-1}\right)\uphi
    \end{multline*} 
    and in this second formula, all the elements acting are in $H$, normalise our level groups, and commute with $\eta$, so we can pull them through the equivariance properties of the Lemma--Eisenstein map to obtain the result.
    
    Finally, we remove the restriction $m \le n$: if $n < m$, then we simply define ${}_{c_1,c_2} z^{[a,b,q,r]}_{\et, M, m, n}$ to be the pushforward of ${}_{c_1,c_2} z^{[a,b,q,r]}_{\et, M, m, n'}$ for any integer $n' \ge m$. This is independent of the choice of $n'$, as is easily seen, and using Theorem \ref{thm:pnorm}(i) and the preceding argument with $n'$ in place of $n$, it has the required properties.
   \end{proof}
   
   \begin{remark}
    \label{rmk:integralnorm}
    It follows from Theorem \ref{thm:pnorm} that the Lemma--Eisenstein classes ${}_{c_1, c_2} z^{[a,b,q,r]}_{\et, M, m, n}$ and their variants ${}_{c_1, c_2} \cZ^{[a,b,q,r]}_{\et, M, m, n}$ satisfy norm-compatibility relations in both $m$ and $n$ after tensoring with $\Qp$. However, one can check that these norm relations actually hold \emph{integrally}, without needing to quotient out by the torsion subgroup of the \'etale cohomology group. This is not obvious from the proofs we have given, but can easily be verified after carefully unwinding the normalisation factors.
   \end{remark}
   
\section{Moment maps and p-adic interpolation}
 
 We now study the interpolation of the \'etale Euler system classes, for varying values of the parameters $(a,b,q,r)$. Our goal is Theorem \ref{thm:hard-interpolation-II}, which shows that these classes can all be obtained as specialisations of a single class ``at infinite level''. 
 
 \subsection{Interpolation of the $\GL_2$ Eisenstein classes} 
  
  We begin by recalling a theorem of Kings \cite{kings15}, which will be the fundamental input for our $p$-adic interpolation results. In this section, let us fix an arbitrary open compact subgroup $K^{(p)} \subset \GL_2(\AA_f^{(p)})$, and for $n \ge 1$, write $K_n = K^{(p)} \times \{ g \in \GL_2(\Zp): g \cong \stbt{*}{*}{0}{1} \bmod p^n\}$. Let us assume, by shrinking $K^{(p)}$ if necessary, that $K_1$ is sufficiently small (and hence so is $K_n$ for all $n \ge 1$). 
  
  We also choose a finite set of primes $\Sigma$ containing $p$ and all primes where $K^{(p)}$ is ramified, so the modular curves have models $Y(K_n)_{\Sigma}$ over $\ZZ[1/\Sigma]$ for all $n$. 
  
  \begin{remark}
   Working with integral models is necessary here, because continuous \'etale cohomology for $\QQ$-varieties does not necessarily commute with inverse limits, but this problem does not arise for finite-type $\ZZ$-schemes such as the $Y(K_n)_{\Sigma}$.
  \end{remark}
  
  \begin{definition}
   We define
   \[ H^1_{\Iw}\Big(Y(K_\infty)_{\Sigma}, \Zp(1)\Big) \coloneqq \varprojlim_{s \ge 1} H^1_{\et}\Big(Y(K_s)_{\Sigma}, \Zp(1)\Big)\]
   where the inverse limit is with respect to the pushforward maps.
  \end{definition}
  
  If $\sH^k_n$ denotes the mod $p^n$ reduction of the sheaf $\sH^k_{\Zp}$ on $Y(K_n)$ (cf.~\S \ref{sect:higherEis}), then we have a canonical section
  \[ e_{k, n} = (e_1)^k \in H^0_{\et}\left(Y(K_n)_\Sigma, \sH^k_n\right). \]
  
  Hence, for any $n \ge 1$, we have a map
  \[ 
   \mom^k_n:
   H^1_{\Iw}\Big(Y(K_\infty)_\Sigma, \Zp(1)\Big) \to 
   H^1_{\et}\Big(Y(K_n), \sH^k_{\Zp}(1)\Big),
  \]
  mapping $(g_s)_{s \ge 1}$ to the restriction to the generic fibre $Y(K_n) \subset Y(K_n)_{\Sigma}$ of the element 
  \[ 
   \left(\pr^{K_s}_{K_n}\right)_* \left( g_s \cup e_{k, s} \right)_{s \ge n} 
   \in \varprojlim_{s \ge n} H^1_{\et}\Big(Y(K_n)_\Sigma, \sH^k_s(1)\Big) = H^1_{\et}\Big(Y(K_n)_{\Sigma}, \sH^k_{\Zp}(1)\Big).
  \]
  \begin{definition}
   Let $\phi$ be a $\Zp$-valued Schwartz function on $(\AA_f^{(p)})^2$, stable under $K^{(p)}$; and let $\phi_s = \phi \otimes \ch(p^s \Zp \times (1 + p^s \Zp))$. For $n \ge 1$, and $c > 1$ coprime to $6p$ and to all primes where $\phi$ is ramified, we define
   \[ 
    \cEI_{\phi} = \left( {}_c g_{\phi_s} \right)_{s \ge 1} \in H^1_{\Iw}\Big(Y(K_\infty)_{\Sigma}, \Zp(1)\Big).
   \]
  \end{definition}

  The following theorem, which will be fundamental for our $p$-adic interpolation results later in this paper, shows that the Siegel units interpolate Eisenstein classes of all weights via these moment maps:

  \begin{theorem}[Kings]
   \label{thm:kings}
   For all integers $k \ge 0$ and $n \ge 1$, we have
   \[ \mom^k_n\left( \cEI_{\phi, n} \right) =\cEis^k_{\et, \phi_n} \]
   as elements of $H^1_{\et}\left(Y(K_n), \sH^k_{\Zp}(1)\right).$
  \end{theorem}
  
  \begin{proof}
   This is a generalisation of \cite[Theorem 4.4.4 \& Theorem 4.5.1]{KLZ1b}, which is the case where $\phi$ is the characteristic function of $(0, 1) + N (\Zhat^{(p)})^2$ for some integer $N$. The general case can be recovered from this using the action of the group $J = \prod_{\ell \in \Sigma - \{p\}} \GL_2(\Ql)$, since both the Siegel units and Eisenstein classes depend $J$-equivariantly on $\phi$, and the moment map clearly commutes with the action of $J$ (as it acts trivially on $e_{k, s}$).
  \end{proof}
  
  Of course, this argument carries over readily to the modular varieties for $H$: if we fix a small enough prime-to-p level $K_H^{(p)}$ and let $K_{H, n} = K_H^{(p)} \times K_{H_p, 1}(p^n)$, then we obtain moment maps
  \[ 
    \mom^{c,d}_{H, n}:
    H^2_{\Iw}\Big(Y_H(K_{H,\infty})_{\Sigma}, \Zp(2)\Big) \to 
    H^2_{\et}\Big(Y_H(K_{H,n}), \sH^{c,d}_{\Zp}(2)\Big),
  \]
  for any $n \ge 1$ and integers $c, d \ge 0$; and there is a class ${}_{c_1, c_2} \EI_{\uphi}$ in $H^2_{\Iw}$, for any $\uphi \in \cS(\AA_f^{(p) 2}, \Zp)^{\otimes 2}$ stable under $K_H^{(p)}$ and unramified at the primes dividing $c_1c_2$, whose images under $\mom^{c,d}_n$ are the Eisenstein classes ${}_{c_1, c_2} \Eis^{c,d}_{\et, \uphi_n}$.
  
 \subsection{Moment maps for $G$}

  For the group $G = \GSp_4$ we have analogous moment maps, as we shall now explain. As in the $\GL_2$ case, we fix an arbitrary subgroup $K^{(p)}_G \subset G(\AA_f^{(p)})$ unramified outside $\Sigma$, and write $K_{G, n} = K_G^{(p)}\times K_{G_p, 1}(p^n)$. We assume that $K_{G, n}$ is sufficiently small for all $n \ge 1$.
  
  \begin{proposition}
   Let $d^{[a,b,q,r]}$ be the image of $v^{[a,b,q,r]} \in V^{a,b}_{\ZZ}$ in $D^{a,b}_{\ZZ} = V^{a,b}_{\ZZ} \otimes \mu^{-(2a+b)}$; and let $d_n^{[a,b,q,r]}$ be its reduction modulo $p^n$. Then the vectors $d_n^{[a,b,0,r]}$, for $0 \le r \le b$, are stable under $K_{G_p, 1}(p^n)$.
  \end{proposition}
  
  \begin{proof}
   We recall that the vectors $v^{[a,b,0,r]}$ lie in the highest $T'$-weight subspace of $V^{a,b}$, where $T'$ is the torus $\dfour{x}{x}{1}{1}$. Hence they are fixed by the unipotent radical $N_\mathrm{S}$ of the Siegel parabolic, and $T'$ acts on them as $x \mapsto x^{2a+b}$. Thus the twists $d^{[a,b,0,r]}$ are fixed by $T'$ and by $N_{\mathrm{S}}$, and the same holds for their reductions modulo $p$.
  \end{proof}
  
  \begin{definition}
   For $n \ge 1$, and any integers $0 \le q \le a$ and $0 \le r \le b$, we define the moment map $\mom^{[a,b,q,r]}_{G, n}$ as the following composition of maps:
   \begin{align*}
    H^*_{\Iw}\left(Y_G(K_{G, \infty})_{\Sigma}, \sD^{q,0}_{\Zp}(3)\right) 
    & \rTo^\cong \varprojlim_s H^*_{\et}\left(Y_G(K_{G,s})_{\Sigma}, \sD^{q,0}_s(3)\right)\\
    &\rTo\varprojlim_s H^*_{\et}\left(Y_G(K_{G,s})_{\Sigma}, (\sD^{q,0}_s \otimes \sD^{a-q,b}_s)(3)\right)\\
     &\rTo \varprojlim_s H^*_{\et}\left(Y_G(K_{G,s})_{\Sigma}, 
               \sD^{a,b}_s(3)\right)\\
    &\rTo \varprojlim_s H^*_{\et}\left(Y_G(K_{G,n})_{\Sigma}, \sD^{a,b}_s(3)\right)\\
    &\rTo^\cong H^*_{\et}\left(Y_G(K_{G,n})_\Sigma, \sD^{a,b}_{\Zp}(3)\right) \\
    &\rTo H^*_{\et}\left(Y_G(K_{G,n}), \sD^{a,b}_{\Zp}(3)\right).
   \end{align*}
   Here the second arrow is given by cup-product with the class $d_s^{[a-q,b,0,r]} \in H^0_\et(Y_G(K_{G,s})_{\Sigma}, \sD_s^{a-q,0})$; the third arrow is given by the Cartan product; the fourth by projection to level $n$; and the final one by restriction to the generic fibre.
  \end{definition}
   
  \begin{remark}
   This construction also has an interpretation in terms of sheaves of measures as in \cite{kings15}. Suppose $q = 0$ for simplicity. One finds that $H^*_\Iw\left(Y_G(K_{G,\infty})_{\Sigma}, \Zp(3)\right)= H^*_\et\left(Y_G(K_{G,n})_{\Sigma}, \Lambda(3)\right)$, where $\Lambda$ is the sheaf on $Y_G(K_{G,n})_{\Sigma}$ given by $\varprojlim_s (\pr^{K_{G,s}}_{K_{G,n}})_*\left(\Zp\right)$. We can interpret $\Lambda$ as the sheaf corresponding to the profinite $\Zp[K_{G,n}]$-module of $\Zp$-valued measures on the quotient $X_n = K_{G,n} / K_{G,\infty}$; in this optic, the moment map is given by the morphism of sheaves
   \[ 
    \Lambda(X_n) \to D^{a,b}_{\Zp}, \qquad \mu \mapsto \int_{X_n} g \cdot d^{[a,b,0,r]}\dd \mu(g).\qedhere
   \]
  \end{remark}
  
  \begin{proposition}
   The Hecke operator $U'(p)$ is well-defined as an endomorphism of the inverse limit $H^*_\Iw\left(Y_G(K_{G,\infty})_\Sigma, \sD^{q,0}_{\Zp}(3)\right)$, and the moment map $\mom^{[a,b,q,r]}_{G, n}$ is compatible with the actions of $U'(p)$ and of $\prod_{\ell \in \Sigma - \{p\}}G(\Ql)$ on both sides.
  \end{proposition}
  
  \begin{proof} Easy check, compare e.g.~\cite[Remark 4.5.3]{KLZ1b}.\end{proof}

 \subsection{Compatibility with the moment maps for $H$}

  We shall now consider compatibility of these constructions between $G$ and $H$. We take $K_H^{(p)} = K_G^{(p)} \cap H(\AA_f^{(p)})$, so we have maps $\iota_n: Y_H(K_{H,n}) \to Y_G(K_{G,n})$ for all $n \ge 1$, and write $\iota_\infty$ for the collection $(\iota_s)_{s \ge 1}$. As before, let $(c, d) = (a+b-q-r, a-q+r)$.
  
  \begin{proposition}
   \label{prop:easy-moment-compat}
   There is a commutative diagram
   \begin{diagram}
    H^2_\Iw\left(Y_H(K_{H,\infty})_\Sigma, \Zp(2)\right)
    & \rTo^{\iota_{\infty,*} \circ \br^{[q,0,q,0]}} & 
    H^4_\Iw\left(Y_G(K_{G,\infty})_\Sigma, \sD^{q,0}_{\Zp}(3-q)\right)\\
    \dTo<{\mom^{c,d}_{H, n}} & & \dTo>{\mom^{[a,b,q,r]}_{G, n}}\\
    H^2_\et\left(Y_H(K_{H,n}), \sH^{c,d}_{\Zp}(2)\right)
    & \rTo^{\iota_{n,*} \circ \br^{[a,b,q,r]}} & 
    H^4_\et\left(Y_G(K_{G,n}), \sD^{a,b}_{\Zp}(3-q)\right)
   \end{diagram}
  \end{proposition}
  
  \begin{proof}
   After much unwinding, this reduces to the assertion the modulo $p^s$ reduction of $\br^{[a,b,q,r]} : \sH^{c,d}_{\Zp} \to \iota^* \sD^{a,b}_{\Zp}$ maps the section $(e_{1,s})^c \boxtimes (f_{1,s})^d$ over $K_{H, s}$ to the pullback of $d_s^{[q,0,q,0]} \cdot d_s^{[a-q,b,0,r]} = d_s^{[a,b,q,r]}$, which is true by the construction of the branching maps.
  \end{proof}
  
  \begin{remark}
   In \cite{KLZ1b}, the analogous statment for the $\GL_2\times\GL_2$-moment maps (Lemma 6.3.1) gives rise to a binomial factor; so using the Cartan product simplifies matters considerably.
  \end{remark}
   
  \subsection{Application to Lemma--Eisenstein classes}
  
  We now return to the situation considered in \S \ref{sect:localdata}. We can apply the machinery of the previous section with $K_G^{(p)}$ taken to be the product $K_S \times \prod_{\ell \notin S\cup \{p\}} G(\Zl)$, so the $K_{G, n}$ of the previous section is $K_{G, 1}(p^n)$, and we obtain moment maps
  \[ \mom_{G,n}^{[a,b,q,r]}: H^4_\Iw\left(Y_{G, 1}(p^\infty)_\Sigma, \sD^{q,0}_{\Zp}(3-q)\right) \to H^4_\et\left(Y_{G, 1}(p^n), \sD^{a,b}_{\Zp}(3-q)\right), \]
  for any $\Sigma \supseteq S \cup \{p\}$. More generally, if we take any $m \ge 0$ and any squarefree $M \ge 1$ whose prime divisors lie in $\Sigma - (S \cup \{p\})$, the same construction also gives maps
  \begin{subequations}
   \begin{equation} 
    \label{eq:mom-a}
    H^4_\Iw\left(Y_G(M, p^m, p^\infty)_\Sigma, \sD^{q,0}_{\Zp}(3-q)\right) 
    \to H^4_\et\left(Y_G(M, p^m, p^n), \sD^{a,b}_{\Zp}(3-q)\right), 
   \end{equation}
   for any $n \ge 1$, and
   \begin{equation} 
    \label{eq:mom-b}
    H^4_\Iw\left(Y_G'(M, p^m, p^\infty)_\Sigma, \sD^{q,0}_{\Zp}(3-q)\right) 
    \to H^4_\et\left(Y_G'(M, p^m, p^n), \sD^{a,b}_{\Zp}(3-q)\right), 
   \end{equation}
   for any $n \ge \max(m, 1)$, which we also denote by $\mom_{G,n}^{[a,b,q,r]}$.
  \end{subequations}
  
  \begin{proposition}
   \label{prop:easy-interpolation-1}
   For each $q \ge 0$, there exists a class
   \[ {}_{c_1, c_2} \cZ^q_{\Iw, M, m} \in 
   H^4_\Iw\left( Y'_G(M, p^m, p^\infty)_\Sigma, \sD^{q,0}_{\Zp}(3-q)\right), \]
   such that
   \[ \mom_{G, n}^{[a,b,q,r]}\left( {}_{c_1, c_2} \cZ^q_{\Iw, M, m} \right) 
   = {}_{c_1, c_2} \cZ^{[a,b,q,r]}_{\et, M, m, n}\]
   for all integers $a, b, r, n$ with $a \ge q$, $0 \le r \le b$ and $n \ge \max(m, 1)$.
  \end{proposition}

  \begin{remark}
   Strictly speaking the elements ${}_{c_1, c_2} \cZ^q_{\Iw, M, m}$ depend also on $\Sigma$, but they are easily seen to be compatible with the natural maps given by enlarging $\Sigma$, so shall suppress this from the notation.
  \end{remark}
 
  \begin{proof}
   We define ${}_{c_1, c_2} \cZ^q_{\Iw, M, m}$ to be the sequence $\left({}_{c_1, c_2} \cZ^{[q,0,q,0]}_{\et, M, m, n}\right)_{n \ge \max(1, m)}$, which is norm-compatible by the same argument as in Theorem \ref{thm:pnorm}(i) (and Remark \ref{rmk:integralnorm}).
      
   By construction, there is a finite set of integers $a_i$ and $x_i \in G(\AA_f)$, independent of $m$ and $n$, such that we have
   \[ \xi'_{M, m, n} = \sum_i a_i \ch\left( x_i K_G'(M, p^m, p^n) \right).\]
   We can therefore write
   \begin{align*}
    {}_{c_1, c_2} \cZ^q_{\Iw, M, m} &= \sum_i a_i \left( \iota_{x_i K_G'(M, p^m, p^\infty),*} \circ \br^{[q,0,q,0]}\right) {}_{c_1, c_2}\EI_{\uphi},\\
    {}_{c_1, c_2} \cZ^{[a,b,q,r]}_{\et, M, m, n} &= \sum_i a_i\left( \iota_{x_i K_G'(M, p^m, p^n),*} \circ \br^{[a,b,q,r]}\right) {}_{c_1, c_2}\Eis^{[c,d]}_{\uphi_n}.
   \end{align*}
    All the $x_i$ have the same $p$-component, namely $\eta_{p, 0}$; this acts trivially on the vector $d^{[a-q,b,0,r]}$, and hence commutes with the moment map. We know that $\Eis^{[c,d]}_{\uphi_n}$ is the image of $\EI_{\uphi}$ under $\mom^{c,d}_{H, n}$ by Theorem \ref{thm:kings}, and the commutative diagram of Proposition \ref{prop:easy-moment-compat} (taking for $K^{(p)}$ the prime-to-$p$ part of $x_i^{-1} K_G'(M, p^m, p^n) x_i$, for each $i$) shows that the image of each summand at level $\infty$ under $\mom^{[a,b,q,r]}_{G, n}$ coincides with the corresponding summand at level $n$. 
  \end{proof}

  \begin{corollary}
   \label{prop:easy-interpolation-2}
   For each $q \ge 0$, there exists a class
   \[ {}_{c_1, c_2} z^q_{\Iw, M, m} \in 
   H^4_\Iw\left( Y_G(M, p^m, p^\infty)_\Sigma, \sD^{q,0}_{\Zp}(3-q)\right), \]
   such that
   \[ \mom_{G, n}^{[a,b,q,r]}\left( {}_{c_1, c_2} z^q_{\Iw, M, m} \right) 
   = {}_{c_1, c_2} z^{[a,b,q,r]}_{\et, M, m, n}\]
   for all integers $a, b, r, n$ with $a \ge q$, $0 \le r \le b$ and $n \ge 1$.
  \end{corollary}
 
  \begin{proof}
   The moment maps of \eqref{eq:mom-a} and \eqref{eq:mom-b} are compatible with respect to the pushforwards $(s_m)_*$, because the action of $\operatorname{diag}(p^{-1}, p^{-1}, 1, 1)$ on $D_{\Zp}^{a-q, b}$ fixes the vector $d^{[a-q,b,0,r]}$. So for $n \ge m$ the corollary follows from Proposition \ref{prop:easy-interpolation-1} by applying $(s_m)_*$ to both sides; and since both sides of the desired formula are norm-compatible in $n$, the result follows for $n < m$ also.
  \end{proof}
  
 \subsection{Cyclotomic twists}
  \label{sect:twistcompat}

  We now consider the more difficult problem of interpolating the classes of the previous sections as the parameter $q$ varies, analogously to \cite[Theorem 6.2.4]{KLZ1b} in the Rankin--Selberg setting. Our main technical result will be the following:
  
  \begin{proposition}
   \label{prop:hard-interpolation-I}
   For each $m \ge 1$ and $q \ge 0$, we have
   \[ 
    {}_{c_1, c_2} z^q_{\Iw, M, m} = (-2)^q \cdot{}_{c_1, c_2} z^0_{\Iw, M, m} \cup (d_m^{[q,0,0,0]} \otimes \zeta_{p^m}^{-q}) \bmod p^m.
   \]
  \end{proposition}
  
  Here $\zeta_{p^m} \in H^0_\et\left(Y_G(1, p^m, p^m), \ZZ/p^m(1)\right)$ is the canonical $p^m$-th root of unity given by the isomorphism \eqref{eq:components}. In an attempt to restrain the excessive proliferation of indices in our notations, we shall give the arguments assuming $M = 1$, and drop $M$ from the subscripts throughout the remainder of the section; the case of general $M$ can be handled similarly (using the decomposition of $\xi'_{M, m, n}$ as a finite sum of characteristic functions of cosets, as in the proof of Proposition \ref{prop:easy-interpolation-1}). In this case, we have
  \[ 
   {}_{c_1, c_2} \cZ^q_{\Iw, m} \coloneqq \left(\eta_* \circ \iota_{\infty, *} \circ \br^{[q,0,q,0]}\right)\left( {}_{c_1, c_2}\EI_{\uphi'} \right) \in H^4_{\Iw}\left(Y_G'(p^m, p^\infty), \sD_{\Zp}^{q, 0}(3-q)\right).
  \]
  
  The branching map $\br^{[q,0,q,0]}$ appearing in the above constructions is given by mapping $1 \in \Zp$ to the $H(\Zp)$-invariant element $d^{[q,0,q,0]} \otimes \zeta^{-q} \in D^{q,0}_{\Zp}(-q)$, where $\zeta$ denotes a basis of the multiplier representation $\mu$ of $G$. After reducing modulo $p^m$, this element is invariant under a larger group:
  
  \begin{proposition}
   The modulo $p^m$ reduction $d_m^{[q,0,q,0]}$ is stable under $K'_{G_p}(p^m, p^\infty) \subset G(\Zp)$.
  \end{proposition}
  
  \begin{proof} 
   Since $K'_{G_p}(p^m, p^\infty)$ is contained in the principal congruence subgroup modulo $p^m$, it acts trivially on  $\sD_m^{a,b}$ for any $a,b$.
  \end{proof}
  
  It follows that we may write
  \begin{equation}
   \label{eq:cyclotwist1}
   {}_{c_1, c_2} \cZ^q_{\Iw, m} = \left[ (\eta_* \circ \iota_{\infty, *}) \left({}_{c_1, c_2}\EI_{\uphi'}\right)\right] \cup \left( \eta_* d_m^{[q,0,q,0]} \otimes \zeta^{-q}\right), 
  \end{equation}
  where $\eta_* d_m^{[q,0,q,0]} \in H^0(Y'_G(p^m, p^\infty), \sD^{q,0}_m)$.
  
  \begin{proposition}
   We have
   \[ s_{m,\sharp}\left( \eta_* d_m^{[q,0,q,0]}\right) = (-2)^q s_m^*\left(d_m^{[q,0,0,0]}\right)\]
   as sections of $s_m^*(\sD^{q,0}_m)$.
  \end{proposition}
  
  \begin{proof}
   We may decompose $D_{\Zp}^{q,0}$ as a direct sum of its eigenspaces for the action of the torus $T'$, which all have weights $\le 0$. On all eigenspaces other than the weight 0 eigenspace, the map $s_{m,\sharp}$ is zero, since $\operatorname{diag}(p, p, 1, 1)^{-m}$ acts as a positive power of $p^m$, which annihilates the module  $D_m^{q,0}$. Hence $s_{m,\sharp}$ factors through projection to the highest weight space relative to $T'$. So we need to compute the projection of $\eta_* (d_m^{[q,0,q,0]}) = (\eta^{-1})^* (d_m^{[q,0,q,0]})$ to this weight space. This is precisely the situation of Lemma \ref{lemma:lielemmaV} (with $h = -1$ in the notation of the lemma), which gives the result above.
  \end{proof}
  
  Proposition \ref{prop:hard-interpolation-I} follows immediately from this, by applying $s_{m, *}$ to both sides of \eqref{eq:cyclotwist1}.

 \subsection{Projection to the ordinary part}
 
  We now define a limiting element in which $m$ (as well as $n$) goes to $\infty$. We set
  \[ 
   H^4_{\Iw}\Big(Y_G(Mp^\infty, p^\infty)_{\Sigma}, \Zp(3)\Big) = \varprojlim_m H^4_{\Iw}\Big(Y_G(M, p^m, p^\infty)_{\Sigma}, \Zp(3)\Big).
  \]
  On this module, there is an action of the ordinary idempotent $e'_{\ord} = \lim_{k \to \infty} U'(p)^{k!}$.
  
  \begin{remark}
   The fact that this limit exists, and is an idempotent, follows from the corresponding statements for \'etale cohomology at finite levels, for which see \cite{tilouineurban99}. Note that Tilouine and Urban define multiple ordinary idempotents, one for each standard parabolic subgroup; ours is the one associated to the Siegel parabolic $P_\mathrm{S}$.
  \end{remark}
  
  \begin{definition}
   We set
   \[ {}_{c_1, c_2} z_{\Iw, M} = \left(U'(p)^{-t}  e'_{\ord} \cdot {}_{c_1, c_2} z^0_{\Iw, M, m}\right)_{m \ge 1}. \]
  \end{definition}
  
  This is a well-defined element of $H^4_{\Iw}\Big(Y_G(Mp^\infty, p^\infty)_{\Sigma}, \Zp(3)\Big)$, since $U'(p)$ is invertible on the image of $e'_{\ord}$, and the terms in the limit are norm-compatible by Theorem \ref{thm:pnorm}(ii).
  
  \begin{definition}
   For integers $m \ge 0, n \ge 1$, $0 \le r \le b$, $a \ge 0$, and $q \in \ZZ$, we define
   \[
    \mom^{[a,b,q,r]}_{G, m, n}: H^4_{\Iw}\Big(Y_G(Mp^\infty, p^\infty)_{\Sigma}, \Zp(3)\Big) \to H^4_{\et}\Big(Y_G(M, p^m, p^n), \sD^{a,b}_{\Zp}(3-q)\Big)
   \]
   by cup-product with $d^{[a,b,0,r]} \otimes \zeta^{-q} \in H^0_{\et}\left(Y_G(Mp^\infty, p^\infty)_{\Sigma}, \sD^{a,b}_{\Zp}(-q)\right)$, where $\zeta$ is the canononical basis of $\Zp(1)$ over $\Spec \ZZ[\Sigma^{-1}, \zeta_{Mp^\infty}]$.
  \end{definition}
  
  Note that we do \emph{not} need to assume that $q$ lies in the interval $\{ 0,\dots, a\}$ in order to define this moment map. However, when we do impose this additional assumption, we obtain compatibility with the preceding constructions: 
  
  \begin{theorem}
   \label{thm:hard-interpolation-II}
   For integers $m \ge 0$, $n \ge 1$, $0 \le q \le a$ and $0 \le r \le b$, we have
   \[ 
    \mom^{[a,b,q,r]}_{G, m, n}\left( {}_{c_1, c_2} z_{\Iw, M} \right)
    = 
    \frac{1}{(-2)^q} \left.\begin{cases}
      U'(p)^{-m} & \text{if $m \ge 1$}\\
    \left(1 - \tfrac{p^q}{U'(p)}\right) & \text{if $m = 0$}
    \end{cases}\right\} \cdot e'_{\ord} \left({}_{c_1, c_2} z^{[a,b,q,r]}_{\et, M, m, n}\right).
   \]
  \end{theorem}
  
  \begin{proof}
   We factor $d^{[a,b,0,r]}$ as the Cartan product of $d^{[a-q,b,0,r]}$ and $d^{[q,0,0,0]}$. Proposition \ref{prop:hard-interpolation-I} shows that cup-product with $d^{[q,0,0,0]} \otimes \zeta^{-q}$ sends ${}_{c_1, c_2} z_{\Iw, M}$ to the inverse system
   \[ 
    (-2)^{-q} \left(U'(p)^{-t} e'_{\ord} \cdot {}_{c_1, c_2} z^q_{\Iw, M, t}\right)_{t \ge 1}. 
   \]
   Projecting this to level $m$ gives the element 
   \[ 
    \frac{1}{(-2)^q}\left.\begin{cases}
    U'(p)^{-m}
      & \text{if $m \ge 1$}\\
     \left(1 - \tfrac{p^q}{U'(p)}\right)  & \text{if $m = 0$}
     \end{cases}\right\} \cdot e'_{\ord} \left({}_{c_1, c_2} z^{q}_{\Iw, M, m}\right).
   \]
   (This is true by definition for $m \ge 1$, and the case $m = 0$ follows by computing the norm of the $m = 1$ element using the appropriate case of Theorem \ref{thm:pnorm}(ii)). Computing the image of this element under $\mom^{[a,b,q,r]}_{G, n}$ using Proposition \ref{prop:easy-interpolation-2} gives the result.
  \end{proof}
 
  \begin{remark}
   The module $e'_{\ord} H^4_{\Iw}\Big(Y_G(Mp^\infty, p^\infty)_{\Sigma}, \Zp(3)\Big)$ can be regarded as an interpolation of the Iwasawa cohomology of the Galois representations attached to $p$-ordinary Siegel modular forms with weights varying in a Hida family. Thus Theorem \ref{thm:hard-interpolation-II} can be interpreted as stating that our Euler system classes interpolate in Hida families. We have not pursued this viewpoint in the present paper for reasons of space, but we intend to revisit the topic of Hida-family variation in a future project.
  \end{remark}

  \begin{corollary}[Cohomological triviality]
   \label{cor:cohtrivial}
   If $m \ge 1$ or $q \ge 1$, then $e'_{\ord} \left({}_{c_1, c_2} z^{[a,b,q,r]}_{\et, M, m, n}\right)$ is in the kernel of the base-extension map
   \[ H^4_{\et}(Y_G(M,p^m, p^n), \sD^{a,b}_{\Zp}(3-q)) \to H^4_{\et}\big(Y_G(M,p^m, p^n)_{\QQbar}, \sD^{a,b}_{\Zp}(3-q)\big)^{\Gal(\QQbar/\QQ)}. \]
  \end{corollary}
 
  \begin{proof}
   Using \eqref{eq:components}, for each $M$ and each $n \ge 1$ we have 
   \[ \varprojlim_m H^0\left(\QQ, H^4_{\et}(Y_G(M,p^m, p^n)_{\QQbar}, \sD^{a,b}_{\Zp}(3-q))\right) = \varprojlim_m H^0\left(\QQ(\zeta_{Mp^m}), H^4_{\et}(Y_{G, 1}(p^n)_{\QQbar}, \sD^{a,b}_{\Zp}(3-q))\right).\]
   This inverse limit is zero, by standard properties of Iwasawa cohomology (see e.g~\cite[Proposition 8.3.5]{nekovar06}).
   
   For $m \ge 1$, it is immediate from Theorem \ref{thm:hard-interpolation-II} that the image of $e'_{\ord}\cdot {}_{c_1, c_2} z^{[a,b,q,r]}_{\et, M, m, n}$ under the edge map lies in the image of this inverse limit, and hence is also zero. For $m = 0$, this argument shows that the class becomes cohomologically trivial after applying $\left(1 - \tfrac{p^q}{U'(p)}\right)$, and for $q \ge 1$ this operator acts invertibly on the image of $e'_{\ord}$.
  \end{proof}
  
\section{Mapping to Galois cohomology}
 
 In this section, we will use the motivic and \'etale classes we have constructed above in order to define Galois cohomology classes in automorphic Galois representations. We begin by recalling a number of results (due to various authors) on Galois representations appearing in the cohomology of the Siegel varieties $Y_G$.
 
 \subsection{Automorphic representations of $\GSp_4$}
 
  Let $(k_1, k_2)$ be integers with $k_1 \ge k_2 \ge 3$, and let $(a, b) = (k_2 - 3, k_1 - k_2)$. There are exactly two unitary discrete-series representations of $G(\RR)$ which are cohomological with coefficients in the algebraic representation $V^{a,b}$: the holomorphic discrete series $\Pi_{k_1, k_2}^H$, and a non-holomorphic generic discrete series $\Pi_{k_1, k_2}^W$. We refer to these as the \emph{discrete series representations of weight $(k_1, k_2)$}. The cuspidal automorphic representations with infinite component $\Pi_{k_1, k_2}^H$ are precisely those generated by classical holomorphic Siegel modular forms of weight $(k_1, k_2)$.

  \begin{remark}
   The pair $\{\Pi_{k_1, k_2}^H, \Pi_{k_1, k_2}^W\}$ is an example of a local $L$-packet.
  \end{remark}
  
  \begin{definition} 
   Let $\Pi = \Pi_f \otimes \Pi_\infty$ be a cuspidal automorphic representation of $G(\AA_f)$ with $\Pi_\infty$ discrete series of weight $(k_1, k_2)$.
   \begin{itemize}
    \item We say $\Pi$ is of \emph{Saito--Kurokawa type} if $k_1 = k_2$ and there exists a Dirichlet character $\chi$ such that $\chi^2 = \omega_\Pi$, and a cuspidal automorphic representation of $\GL_2(\AA)$ of central character $\omega_\Pi$ attached to some holomorphic newform of weight $2k_1 - 2$, such that for all but finitely many places $v$ we have 
    \[ L(\Pi_v, s) =  L(\pi_v, s) L(\chi_v, s - \tfrac12) L(\chi_v, s+\tfrac12). \]

    \item We say $\Pi$ is of \emph{Yoshida type} if there is a pair $(\pi_1, \pi_2)$ of cuspidal automorphic representations of $\GL_2(\AA)$, both with central character $\omega_\Pi$, corresponding to two elliptic modular newforms of weights $r_1 = k_1+k_2-2$ and $r_2 = k_1-k_2 + 2$, such that for all but finitely many places $v$ we have
    \[ L(\Pi_v, s) = L(\pi_{1, v}, s) L(\pi_{2, v}, s). \]
    
    \item Otherwise, we say $\Pi$ is \emph{non-endoscopic}.
   \end{itemize}
  \end{definition}

  \begin{theorem}[Taylor, Weissauer, Urban, Xu]
   \label{thm:taylorweissauer}
   Let $\Pi$ be as in the previous definition, and suppose $\Pi$ is non-endoscopic. Let $S$ be the set of primes at which $\Pi$ ramifies, and let $w = k_1 + k_2 - 3$.
   \begin{enumerate}
    
    \item The representation $\Pi$ is one of a pair $\{ \Pi^H, \Pi^W \} = \{ \Pi_f \otimes \Pi_{k_1,k_2}^H, \Pi_f \otimes \Pi_{k_1,k_2}^W\}$ of cuspidal automorphic representations having the same finite part, both of which have multiplicity one in $L^2_{\mathrm{cusp}}\left(G(\QQ) \backslash G(\AA), \omega_\Pi\right)$.
    
    \item For any prime $\ell \notin S$, the local representation $\Pi_\ell$ is an unramified principal series representation.
    
    \item For $\ell \notin S$, let $P(X) \in \CC[X]$ denote the quartic polynomial such that
    \[ 
     L(\Pi_\ell, s - \tfrac{w}{2}) = P_\ell(\ell^{-s})^{-1}.
    \]
    Then the subfield $E \subset \CC$ generated by the coefficients of the $P_\ell(X)$, for all $\ell \notin S$, is a finite extension of $\QQ$.
    
    \item For any prime $p$ and choice of embedding $E \into \Qpbar$, there is a semi-simple Galois representation
    \[ \rho_{\Pi, p}: \Gal(\QQbar / \QQ) \to \GL_4(\Qpbar) \]
    characterised (up to isomorphism) by the property that, for all primes $\ell \notin S \cup \{p\}$, we have 
    \[ \det\left(1 - X \rho_{\Pi, p}(\Frob_\ell^{-1})\right) = P_\ell(X), \]
    where $\Frob_\ell$ is the arithmetic Frobenius. 
    
    \item The representation $\rho_{\Pi, p}$ is either irreducible, or is the direct sum of two distinct irreducible two-dimensional representations. In particular, we have
    \[ H^0\left(\QQ^{\mathrm{ab}}, \rho_{\Pi, p} \right) = 0.\]
    
    \item The restriction of $\rho_{\Pi, p}$ to a decomposition group at $p$ is de Rham, and its Hodge numbers\footnote{Here ``Hodge numbers'' are the jumps in the Hodge filtration of $\DD_{\dR}$, which are the negatives of Hodge--Tate weights, so the cyclotomic character has Hodge number $-1$.} are $\{ 0, k_2 - 2, k_1 - 1, k_1 + k_2 - 3 \}$. If $\Pi_p$ is unramified, then $\rho_{\Pi, p}$ is crystalline, and we have $\det\big(1 - X \varphi : \DD_{\mathrm{cris}}(\rho_{\Pi, p}) \big) = P_p(X)$.
   \end{enumerate}
  \end{theorem}
 
   \begin{remark}
    Note that $P_\ell(X)$ in (3) is given by the action of the Hecke-operator-valued polynomial $\mathcal{P}_\ell(X)$ of \S \ref{sect:heckeops} on the spherical vector of $\Pi_\ell \otimes |\cdot|^{(3-w)/2}$.
   \end{remark}
    
  \begin{proof}
   Parts (3) and (4) are \cite[Theorem I]{weissauer05}. Part (2) is also implicit in this theorem, since the ``purity'' statement on $\rho_{\Pi, p}$ implies that the local $L$-factor $L(\Pi_\ell, s)$ has the form $\prod_{i=1}^4(1 -\alpha_i \ell^{-s})^{-1}$ with $|\alpha_i| = 1$, which rules out all of the other classes of unramified representations of $G(\Ql)$ (all of which have $\alpha_i / \alpha_j = \ell$ for some $i, j$). Part (5) is contained in Theorem II of \emph{op.cit.}.
   
   For parts (1) and (6), the fact that $\Pi^H$ and $\Pi^W$ have the same multiplicity $m(\Pi^H) = m(\Pi^W)$, and the characterisation of the Hodge numbers of $\rho_{\Pi, p}$, was proved in \cite[Proposition 1.5 \& Theorem III]{weissauer05} under the assumption that $\Pi$ is weakly equivalent to a globally generic representation; and in fact all such $\Pi$ have this property by the main theorem of \cite{weissauer08}. The assertion regarding $\DD_{\mathrm{cris}}$ is \cite[Theorem 1]{urban05}. Finally, Xu has shown in \cite[\S 3.5]{xu18} that the common multiplicity of $\Pi^H$ and $\Pi^W$ is equal to 1.
  \end{proof}

  It is expected that $\rho_{\Pi, p}$ is always irreducible, but this is only known for large $p$:
  
  \begin{theorem}[Ramakrishnan]
   If $\Pi$ is unramified at $p$ and $p > 2w + 1$, then the representation $\rho_{\Pi, p}$ is irreducible.
  \end{theorem}
  
  \begin{proof}
   By \cite{gantakeda11}, the automorphic representation $\Pi$ lifts to a cuspidal automorphic representation of $\GL_4$. This lifted representation is regular at $\infty$ (since its local parameter at $\infty$ is determined by that of $\Pi_\infty$, via the compatibility of the local and global liftings). If $\Pi_p$ is unramified, then the corresponding Galois representation is crystalline at $p$, and hence Theorem B of \cite{ramakrishnan13} shows that it is irreducible as long as $p - 1$ is greater than the largest difference between the Hodge--Tate weights of $\rho_{\Pi, p}$, which translates into the condition on $p$ stated above.
  \end{proof}
    
 \subsection{Automorphic cohomology}
 
  As before, given integers $k_1 \ge k_2 \ge 3$ as above, we let $(a, b) = (k_2-3, k_1 - k_2)$, so that $w = 2a + b + 3$. Choosing a (sufficiently small) level group $K$, we thus have a Shimura variety $Y_G(K)$, and a relative Chow motive $\sD^{a,b}_{\QQ} = \Anc_G(D^{a,b})$ over this variety. We are interested in the parabolic cohomology of the $p$-adic \'etale realisation
  \[  
   H^3_{\et, !}\left( Y_G(K)_{\QQbar}, \sD^{a,b}_{\Qp}\right) = \operatorname{image}\left(H^3_{\et, c} \to H^3_{\et}\right). 
  \]
  Since $\sD^{a,b}_{\Qp}$ has an equivariant structure, the direct limit $H^3_{\et, !}\left( Y_{G, \QQbar}, \sD^{a,b}_{\Qp}\right)$ is an admissible smooth representation of $G(\AA_f)$, with an action of $\Gal(\QQbar /\QQ)$ commuting with the $G(\AA_f)$-action.

  \begin{notation}
   Let $\Sigma(k_1, k_2)$ denote the set of isomorphism classes of representations $\Pi_f$ of $G(\AA_f)$ which are the finite part of a cuspidal automorphic representation $\Pi = \Pi_f \otimes \Pi_\infty$ in which $\Pi_\infty$ is one of the two discrete series representations of weight $(k_1, k_2)$.
  \end{notation}
  
  \begin{theorem}[Taylor, Weissauer]
   There is a $G(\AA_f) \times \Gal(\QQbar/\QQ)$-equivariant direct sum decomposition
   \[
    H^3_{\et, !}\left( Y_{G, \QQbar}, \sD^{a,b}_{\Qp}(3)\right) \otimes \Qpbar
    \cong
    \bigoplus_{\Pi_f \in \Sigma(k_1, k_2)}\left( \Pi_f^*[\tfrac{w-3}{2}] \otimes W_{\Pi_f}^* \right)
   \]
   where $W_{\Pi_f}$ is a finite-dimensional $p$-adic representation of $\Gal(\QQbar/\QQ)$. If $\Pi$ is non-endoscopic, then the term corresponding to $\Pi$ is a direct summand of the full cohomology $H^3_{\et}\left(Y_{G,\QQbar}, \sD^{a,b}_{\Qp}\right)$, and the semisimplification of $W_{\Pi_f}$ is isomorphic to $\rho_{\Pi, p}$.
  \end{theorem}
  
  Here $\Pi_f[r] = \Pi_f \otimes \|\mu(-)\|^r$, as above.
  
  \begin{proof} 
   See \cite[\S 1]{weissauer05}.
  \end{proof}

 \subsection{Arithmetic \'etale cohomology}
 
  We now consider the \'etale cohomology of $Y_G(K)$ as a $\QQ$-variety (not as a $\QQbar$-variety); more precisely, we work with continuous \'etale cohomology in the sense of Jannsen \cite{jannsen88}. Note that this space is not finite-dimensional in general. Nonetheless, there is a Hochschild--Serre spectral sequence associated to the structure map $Y_G(K) \to \Spec \QQ$, for any lisse \'etale $\Zp$-sheaf or $\Qp$-sheaf $\sF$,
  \[ 
   E_2^{rs} = H^r\!\left(\QQ, H^s_\et(Y_G(K)_{\QQbar}, \sF)\right) \Rightarrow H^{r+s}_\et(Y_G(K), \sF),
  \]
  for any integer $n$. Consequently, if $H^4_\et(Y_G(K), \sF)_0 = \ker\left(H^4_\et(Y_G(K), \sF) \to H^4_\et(Y_G(K)_{\QQbar}, \sF)\right)$ (the cohomologically trivial classes), then we have a natural ``Abel--Jacobi'' map
  \[ H^4_\et(Y_G(K), \sF)_0 \to H^1\!\left(\QQ, H^3_\et(Y_G(K)_{\QQbar}, \sF)\right).\]

  These maps are compatible with respect to change of $K$, and therefore assemble into a map of $G(\AA_f)$-representations. More generally, this also applies with $Y_G$ replaced by its base-extension $Y_G \times \Spec \QQ(\zeta_N)$, for any integer $N$.
  
  \begin{definition}
   Suppose $\Pi$ is non-endoscopic. For any $N \ge 1$ and $q \in \ZZ$, we write
   \[ 
    \pr_{\Pi^*}: H^4_{\et}\big(Y_{G, \QQ(\zeta_N)}, \sD^{a,b}_{\Qp}(3-q)\big)_0 \rTo \Pi_f^*[\tfrac{w-3}{2}] \otimes H^1(\QQ(\zeta_N), W_{\Pi_f}^*(-q))
   \]
   for the composite of the Abel--Jacobi map and projection onto the $\Pi_f^*$-isotypical component.
  \end{definition}
 
 \subsection{Ordinarity}
  
  Let us now choose a non-endoscopic $\Pi$ with $\Pi_\infty$ discrete series of weight $(k_1, k_2)$, as above, and let $E$ be a number field as in Theorem \ref{thm:taylorweissauer}(3), so we have polynomials $P_\ell(X) \in E[X]$ for all unramified primes $\ell$.
  
  \begin{proposition}
   Let $p$ be a prime such that $\Pi_p$ is unramified, and write $P_p(X) = 1 + a_1 X + \dots + a_4 X^4$. Then, for any embedding $E \into \Qpbar$, we have the relations
   \begin{align*}
    v_p(a_1) &\ge 0,             & v_p(a_2) &\ge k_2 - 2, \\
    v_p(a_3) &\ge k_1 + k_2 - 3, & v_p(a_4) &= 2k_1 + 2k_2 - 6,
   \end{align*}
   where the valuation is normalised such that $v_p(p) = 1$.
  \end{proposition}
  
  \begin{proof}
   By Theorem \ref{thm:taylorweissauer}(6), we know that the crystalline $G_{\Qp}$-representation $\rho_{\Pi, p} |_{G_{\Qp}}$ has Hodge numbers $\{0, k_2 - 2, k_1 - 1, k_1 + k_2 - 3\}$ and its Frobenius has characteristic polynomial $P_p(X)$. The above relations are now precisely the assertion that the Newton polygon associated to this representation lies on or above its Hodge polygon (with the same endpoints), which is a general property of crystalline representations: see e.g.~\cite[Proposition 5.4.2]{fontaine94b}.
  \end{proof}
 
  \begin{remark}
   These inequalities can also be proved directly (without using $p$-adic Hodge theory), by expressing the coefficients of $P_p(X)$ as Hecke eigenvalues using the formula of Lemma \ref{lemma:Tleigenvalues}, and showing that suitable scalar multiples of the Hecke operators preserve an integral lattice in Betti cohomology.
  \end{remark}
  
  \begin{definition}
   \label{def:goodordinary}
   We say that $\Pi$ is \emph{good ordinary} at $p$, with respect to some choice of embedding $E \into \Qpbar$, if $\Pi$ is unramified at $p$ and the eigenvalue of $T(p)$ acting on $\Pi_p[\tfrac{3-w}{2}]^{G(\Zp)}$ is a $p$-adic unit.
  \end{definition}
  
  Since this $T(p)$-eigenvalue is $a_1$ in the notation of the previous proposition, it follows easily from the theory of Newton polygons that $\Pi$ is good ordinary at $p$ if and only if it is unramified at $p$ and $P_p(X)$ has a factor in $\Qpbar[X]$ of the form $1 - \alpha X$ with $\alpha$ a $p$-adic unit. Moreover, this $\alpha$ is unique if it exists. By Proposition \ref{prop:Uleigenvalues}, $\alpha$ is also an eigenvalue of $U(p)$ acting on the four-dimensional space $\Pi_p[\tfrac{3-w}{2}]^{K_{G_p, 0}(p)}$, or dually of $U'(p)$ acting on $\Pi_p^*[\tfrac{w-3}{2}]^{K_{G_p, 0}(p)}$.
  
  We assume henceforth that $\Pi$ is good ordinary at $p$; and we choose one final piece of data needed to define our Euler system. Having fixed a set $S$ as in \S\ref{sect:localdata} above, we obtain a level group 
  \[ K_{G, 0}(p) = K_S \times \prod_{\ell \nmid pS} G(\Zl) \times K_{G_p, 0}(p).\]
  Enlarging $S$ and shrinking $K_S$ if necessary, we suppose that $\Pi_f$ has non-zero invariants under $K_{G, 0}(p)$.
  
  \begin{definition}
   We choose a vector $v_\alpha \in \Pi_f$ invariant under the subgroup $K_{G, 0}(p)$ and lying in the $U(p) = \alpha$ eigenspace. 
  \end{definition}
 
  This choice gives a linear functional $\Pi_f^* \to \Qpbar$, and hence a homomorphism of Galois representations
  \[ H^3_{\et}\left(Y_{G, \QQbar}, \sD^{a,b}_{\Qp}(3)\right) \to W^*_{\Pi_f},\]
  which we also denote by $v_\alpha$. It seems reasonable to interpret this as a ``modular parametrisation'' of the Galois representation $W^*_{\Pi_f}$. Composing with the map $\pr_{\Pi^*}$ of the previous section, for each $N \ge 1$ and $0 \le q \le a$ we have a map
  \[ \pr_{(\Pi^*, v_\alpha)}: H^4_{\et}\left(Y_{G, \QQ(\zeta_N)}, \sD^{a,b}_{\Qp}(3-q)\right) \to H^1\left(\QQ(\zeta_N), W^*_{\Pi_f}(-q)\right).\]
  Our local hypothesis at $p$ implies that this homomorphism factors through projection to the $U'(p) = \alpha$ eigenspace at level $K_{G, 0}(p)$, and in particular through the ordinary idempotent $e'_{\ord}$.
  
  \begin{remark}
   It will come as no great surprise to the well-informed reader to learn that the ordinarity condition can be relaxed somewhat, to allow sufficiently small positive values of the ``slope'' $v_p(t(p))$, where $t(p)$ is the eigenvalue of $T(p)$ on $\Pi_f[\tfrac{3-w}{2}]$. (Slope $< 1$ is easy; slope $< 1 + a$ may be accessible, using the methods of \cite{loefflerzerbes16}.) However, we stick with $v_p(t(p)) = 0$ in the present paper for simplicity.
  \end{remark}
  
 \subsection{Lemma--Eisenstein classes in Galois cohomology}
 
  We apply the maps of the previous section to the classes $z_{M, m, 1}^{[a,b,q,r]}$, or more precisely to their images in $H^4_{\mot}(Y_{G, 1}(p)_{\QQ(\zeta_{Mp^m})}, \sD_{\QQ}^{a,b}(3-q))$ via the map $\jmath_{Mp^m}$ of \eqref{eq:components}.
    
  \begin{definition}
   For $m \ge 0$, we define a class
   \[ z_{M, m}^{[\Pi, q, r]} \in H^1\left(\QQ(\zeta_{Mp^m}), W_{\Pi_f}^*(-q)\right) \]
   by
   \[ 
    \frac{1}{M} \cdot (\pr_{\Pi^*, v_\alpha} \circ \jmath_{Mp^m}) \left.\begin{cases} 
    \left(\tfrac{p^q}{U'(p)}\right)^m e'_{\ord} & \text{if $m \ge 1$}\\
    \left( 1 - \tfrac{p^q}{U'(p)}\right) e'_{\ord} & \text{if $m = 0$}
   \end{cases}\right\} \cdot
      r_{\et} \left( z_{M, m, 1}^{[a,b,q,r]}\right).
   \]
  \end{definition}
    
  Note that this is well-defined, since the classes concerned are cohomologically trivial by Corollary \ref{cor:cohtrivial}. It follows from Theorem \ref{thm:pnorm}(ii) that the classes $z_{M, m}^{[\Pi, q, r]}$ are compatible under the Galois corestriction maps as $m$ varies. More importantly, they also satisfy a compatibility with respect to $M$:
  
  \begin{proposition}
   \label{prop:ESnorm-nonint}
   Let $\ell \nmid Mp$ be a prime, with $\ell \notin S$. Then we have
   \begin{equation}
    \label{eq:ESnorm} 
    \operatorname{cores}_{\QQ(\zeta_{Mp^m})}^{\QQ(\zeta_{\ell M p^m})}\left(z_{\ell M, m}^{[\Pi, q, r]}\right) = P_\ell(\ell^{-1-q} \sigma_\ell^{-1}) z_{M, m}^{[\Pi, q, r]}
   \end{equation}
   where $\sigma_\ell \in \Gal(\QQ(\zeta_{Mp^m}) / \QQ)$ is the arithmetic Frobenius at $\ell$.
  \end{proposition}
  
  \begin{proof}
   This will (eventually) turn out to be a mildly disguised form of Corollary \ref{cor:tamenormfinal}. As in the proofs of the ``$p$-direction'' norm relations, we are comparing the images of two elements of $\cS(\AA_f^2)^{\otimes 2} \otimes \cH(G(\AA_f))$ which are pure tensors having the same local components at all primes away from $\ell$, and at $\ell$ are
   \[ 
    \tfrac{\ell-1}{\ell}\cdot \uphi_{1, 1} \otimes \ch(\eta_{\ell, 1} G(\Zl))\quad\text{and}\quad \uphi_0 \otimes \ch(G(\Zl)).
   \]
   (The factor $(\ell - 1)$ arises by comparing the volumes of the subgroups $W$, and the $\ell$ from the $1/M$ in the definition of $z_{M, m}^{[\Pi, q, r]}$.)
   
   Moreover, the map which we are applying to these elements factors through the Eisenstein symbol map $\Eis^{(c, d)}$, where $(c, d) = (a+q +b-r, a-q+r)$. We first give the proof assuming $c, d \ge 1$. In this case the divisor map $\partial$ is injective on the image of the Eisenstein symbol, so the local version of our map factors through the map 
   \[ \cS(\Ql^2)^{\otimes 2} \rTo_{\uphi \mapsto F_{\uphi}} \bigoplus I(\underline{\eta}), \]
   where the sum is over some set of pairs $\eta = (\eta_1, \eta_2)$ of finite-order characters of $\Ql^\times$, and $I(\underline{\eta})$ is an irreducible principal series representation of $H(\Ql)$. So it suffices to check that for any $(G\times H)(\Ql)$-equivariant homomorphism $I(\eta) \otimes \cH(G(\Ql)) \to W$, where $W$ is an irreducible principal series representation of $G(\Ql)$, the images of $F_{\uphi_{1, 1}} \otimes \ch(\eta_{\ell, 1} G(\Zl))$ and $F_{\uphi_0} \otimes \ch(G(\Zl))$ in $W^{G(\Zl)}$ are related via an Euler factor. Corollary \ref{cor:tamenormfinal} gives a relation of exactly this form, with $\tfrac{\ell}{\ell-1} L(W^\vee, -\tfrac12)^{-1}$ as the correction factor. If $M = 1$ and $m = 0$, applying this with $W = \Pi_\ell^*[\tfrac{w-3}{2}-q]$ gives the result, noting that $L(W^\vee, -\tfrac12)^{-1} = L(\Pi_\ell, 1 + q - \tfrac{w}{2})^{-1} = P_\ell(\ell^{-1-q})$. 
   
   To obtain the result in general, we apply this to each of the twists of $W$ by Dirichlet characters modulo $Mp^m$; the twist of course modifies the Euler factor $P_\ell$, which corresponds to the appearance of $\sigma_\ell$ in the statement of the theorem.
   
   If either $c$ or $d$ (or both) is zero, then the divisor map has a kernel (consisting of modular units which are constant along one of the factors of $Y_H$). However, the kernel of this map is a sum of non-generic representations of $H(\Ql)$, and Lemma \ref{lem:properquot} shows that if $\tau_\ell$ is any of these representations, then every $G \times H$-equivariant map $\tau_\ell \otimes \cH(G(\Ql)) \to W$ is zero. 
  \end{proof}
  
  This Euler system norm relation is not terribly useful on its own -- we need to combine it with some uniform control over the denominators of these classes. Let $T_{\Pi_f}^*$ denote the lattice in $W_{\Pi_f}^*$ generated by the image of $H^3_{\et}(Y_{G, 1}(p)_{\QQbar}, \sD^{a,b}_{\Zp})$ under $v_\alpha$. We choose a pair $(c_1, c_2)$ of integers $> 1$ as before, and we impose the additional assumption that the elements $\operatorname{diag}(1,c_1,1,c_1)$ and $\operatorname{diag}(c_2, 1, c_2, 1) \in G(\AA_f)$ act trivially on the vector $v_\alpha$ (which can be arranged, as usual, by assuming that the $c_i$ are sufficiently close to 1 modulo the primes in $S$).

  \begin{definition}
   For $M$ coprime to $c_1 c_2$, and $m \ge 0$, let us define
    \[ {}_{c_1, c_2} z_{M, m}^{[\Pi, q, r]} \in H^1\left(\QQ(\zeta_{Mp^m}), T_{\Pi_f}^*(-q)\right) \]
    by
    \[ 
    \frac{1}{M} \cdot (\pr_{\Pi^*, v_\alpha} \circ \jmath_{Mp^m}) \left.\begin{cases} 
    U'(p)^{-m} & \text{if $m \ge 1$}\\
    \left( 1 - \tfrac{p^q}{U'(p)}\right) & \text{if $m = 0$}
    \end{cases}\right\} \cdot
    e'_{\ord} \left({}_{c_1, c_2} z_{M, m, 1}^{[a,b,q,r]}\right).
    \]
   \end{definition}
   
   From Proposition \ref{prop:integralclass} (and the interaction between $\jmath_{Mp^m}$ and the $G(\AA_f)$ action described in \S\ref{sect:components}), we see that the image of ${}_{c_1, c_2} z_{M, m}^{[\Pi, q, r]}$ after inverting $p$ is
   \[  (c_1^2 - c_1^{-(a-q+b-r)}\sigma_{c_1}) (c_2^2 - c_2^{-(a-q+r)} \sigma_{c_2}) \cdot z_{M, m}^{[\Pi, q, r]}, \]
   where $\sigma_{c_i} \in \Gal(\QQ(\zeta_{Mp^m}) / \QQ)$ maps to $c_i \bmod Mp^m$ under the mod $Mp^m$ cyclotomic character.
  
  \begin{theorem}
   \label{thm:Iwnormrel}
   Suppose $M$ is coprime to $c_1 c_2$. Then:
   \begin{enumerate}
    \item[(a)] For each integer $r$ with $0 \le r \le b$, there is a class
    \[  {}_{c_1, c_2} z_{\Iw, M}^{[\Pi, r]} \in H^1_{\Iw}\left(\QQ(\zeta_{Mp^\infty}), T_{\Pi_f}^*\right) \]
    uniquely determined by the following property: for each $m \ge 0$ and $q \in \{0, \dots, a\}$, the image of $(-2)^q {}_{c_1, c_2} z_{\Iw, M}^{[\Pi, r]}$ in $H^1\left(\QQ(\zeta_{Mp^m}), T_{\Pi_f}^*(-q)\right)$ is ${}_{c_1, c_2} z_{M, m}^{[\Pi, q, r]}$.
    
    \item[(b)] If $\ell \nmid Mpc_1 c_2$ and $\ell \notin S$, then we have the norm relation
    \[ 
     \operatorname{cores}^{\QQ(\zeta_{\ell M p^\infty})}_{\QQ(\zeta_{M p^{\infty}})} \left({}_{c_1, c_2} z_{\Iw, \ell M}^{[\Pi, r]}\right) = P_\ell(\ell^{-1} \sigma_\ell^{-1}) \cdot {}_{c_1, c_2} z_{\Iw, M}^{[\Pi, r]}.
    \]
   \end{enumerate}
  \end{theorem}
  
  \begin{proof}
   The image of the class ${}_{c_1, c_2} z_{\Iw, M}$ under the maps $\mom^{[a,b,0,r]}_{G, m, 1}$, for $m \ge 0$, define a class in $H^4_{\Iw}\left(Y_G(M, p^\infty, p)_{\Sigma}, \sD^{a,b}_{\Zp}(3)\right)$ where $\Sigma$ is the set of primes dividing $MpS$. This class is cohomologically trivial by Corollary \ref{cor:cohtrivial}. By Theorem \ref{thm:hard-interpolation-II}, the image of this class under tensor product with $\zeta^{-q}$ and projection to level $m$ agrees, up to an appropriate correction factor, with the \'etale class ${}_{c_1, c_2} z^{[a,b,q,r]}_{\et, M, m, 1}$. Applying the maps $\pr_{(\Pi^*, v_\alpha)} \circ \jmath_{Mp^m}$ and dividing by $M$ gives an Iwasawa cohomology class with the required properties.
   
   Let us now prove (b). We claim that, for any given $q$, the map
   \[ H^1_{\Iw}(\QQ(\zeta_{Mp^m}), T_{\Pi_f}^*) \to \varprojlim_m H^1(\QQ(\zeta_{Mp^m}), W_{\Pi_f}^*(-q))\]
   is injective (i.e.~there is no non-zero Iwasawa cohomology class whose image is at every finite level is $p$-torsion). This is the map denoted by $\lambda_q$ in \cite[Appendix A]{LLZ}. By Proposition A.2.6 of \emph{op.cit.}, its kernel is contained in the $\Lambda(\Gamma)$-torsion submodule of the Iwasawa cohomology; and this torsion submodule is in fact zero, by Theorem \ref{thm:taylorweissauer}(5). With this injectivity in hand, the norm relation (b) follows from the norm relation of Proposition \ref{prop:ESnorm-nonint} for the non-integral classes at finite level.
  \end{proof}
  
  We conclude this section with a local property of these Galois cohomology classes.
  
  \begin{proposition}
   \label{prop:eltisSelmer}
   For each prime $\lambda \nmid p$ of $\QQ(\zeta_{Mp^m})$, the image of ${}_{c_1, c_2} z_{M, m}^{[\Pi, q, r]}$ in $H^1\!\left(I_\lambda, T_{\Pi_f}^*(-q)\right)$ is zero, where $I_\lambda \subset \Gal(\QQbar / \QQ(\zeta_{Mp^m}))$ is an inertia group at $\lambda$. For the primes above $p$, the localisation lies in the Bloch--Kato crystalline subspace $H^1_\f\!\left(\QQ(\zeta_{Mp^m})_{\lambda},  T_{\Pi_f}^*(-q)\right)$. 
  \end{proposition}
  
  \begin{proof}
   It is a standard result that Galois cohomology classes which are universal norms in the $p$-cyclotomic extension are always unramified outside the primes above $p$ (see e.g.~\cite[Corollary B.3.5]{rubin00}). The fact that our classes satisfy the Bloch--Kato condition at $p$ is deeper. The comparison between \'etale cohomology and the syntomic cohomology of Nekov\'a\v{r}--Nizio\l{} \cite{nekovarniziol} shows that the localisations lie in the possibly larger Bloch--Kato space $H^1_{\mathrm{g}} \supseteq H^1_\f$. It suffices to check that $p^{-1}$ is not an eigenvalue of crystalline Frobenius on $\DD_{\mathrm{cris}}(W_{\Pi_f}^*(-q))$, since this implies that the $H^1_\f$ and $H^1_{\mathrm{g}}$ spaces coincide. However, the eigenvalues of Frobenius on this space are exactly the quantities $p^q \alpha_i^{-1}$, where $L(\Pi_p, s-\tfrac{w}{2}) = \prod_{i=1}^4 (1 - \alpha_i p^{-s})^{-1}$. Since the $\alpha_i$ are Weil numbers of weight $w$, the $p^q \alpha_i^{-1}$ have weight $2q-w \le -3$; so none of these quantities may be equal to $p^{-1}$, and $H^1_\f$ and $H^1_{\mathrm{g}}$ coincide for this representation.    
  \end{proof}
  
  Equivalently, we have shown that 
  \[ {}_{c_1, c_2} z_{M, m}^{[\Pi, q, r]} \in H^1_\f\!\left( \QQ(\zeta_{Mp^m}), T_{\Pi_f}^*(-q)\right), \]
  where the right-hand side is the global Bloch--Kato Selmer group.
 
 \subsection{The Euler-system map}
 
  In this short section we give a slicker reinterpretation of the above results. Let $L$ be a finite extension of $\Qp$ with ring of integers $\cO$; and $\Sigma$ a finite set of primes including $p$. We let $\mathcal{R}$ denote the set of square-free products of primes not in $\Sigma$. If $T$ is a finite free $\cO$-module with a continuous action of $\Gal(\QQ^\Sigma / \QQ)$, and $\ell \notin \Sigma$, we let $P_\ell(T; X) = \det(1 - X \Frob_\ell^{-1} : T)$
  
  \begin{definition}
   For $(T, \Sigma)$ as above, we define $\operatorname{ES}(T, \Sigma)$ to be the set of families of cohomology classes $\left( c_M\right)_{M \in \mathcal{R}}$, with $c_M \in H^1_{\Iw}(\QQ(\zeta_{Mp^\infty}), T)$, satisfying
   \[ \operatorname{cores}^{\QQ(\zeta_{\ell M p^\infty})}_{\QQ(\zeta_{Mp^\infty})} \left(c_{M\ell}\right) = P_\ell(T^*(1); \sigma_\ell^{-1}) c_M\]
   for $\ell$ prime with $\ell \nmid M$, $\ell \notin \Sigma$. We refer to such families as \emph{Euler systems for $(T, \Sigma)$}.
  \end{definition}
  
  (In the notation of \cite[Definition 2.1.1]{rubin00}, these are Euler systems for $(T, \mathcal{K}, \mathcal{N})$ where $\mathcal{K}$ is the compositum of the $\QQ( \zeta_{Mp^\infty})$ for $M \in \mathcal{R}$, and $\mathcal{N} = \prod_{\ell \in \Sigma} \ell$.) 
  
  Theorem \ref{thm:Iwnormrel} shows that the classes $\left({}_{c_1, c_2} z_{\Iw, M}^{[\Pi, r]}\right)_{M \in \mathcal{R}}$ are an Euler system for $T = T_{\Pi_f}^*$, with $\Sigma$ the set of primes not dividing $p c_1 c_2 S$. This Euler system, of course, depends on choices of local data at the primes in $S$, namely the Schwartz function $\uphi_S$ and the group $W_S = \prod_{\ell \in S} W_\ell \subset H(\QQ_S)$. The goal of this section is to make this dependence precise. 
  
  Let $K^{S,p} = \prod_{\ell \nmid pS} G(\Zl) \times K_{G_p,0}(p)$. Then our modular parametrisation $v_\alpha$ is an element of the space
  \[ \sigma_S \coloneqq \Pi_f[\tfrac{3-w}{2}]^{(K^{S,p}, U(p) = \alpha)},\]
  which is an irreducible representation of $G(\QQ_S)$, isomorphic to $\Pi_S[\tfrac{3-w}{2}]$. The choice of subgroups $K_S$ and $W_S$ at the bad primes only affects the construction of the Euler system through the volume factor $\vol(W)$, so we have in fact defined a canonical bilinear map
  \[ 
   \LE^{[\Pi, r]}_{\mathrm{ES}, S}: \cS(\QQ_S^2, L)^{\otimes 2} \otimes \sigma_S \rTo \operatorname{ES}\left(T_{\Pi_f}^*,\Sigma \right) \otimes_{\cO} L,
  \]
  mapping $\uphi_S \otimes v_{\alpha}$ to $\vol(W_S) \cdot \left({}_{c_1, c_2} z_{\Iw, M}^{[\Pi, r]}\right)_{M \in \mathcal{R}}$. 
  
  \begin{remark}
   The map $\LE^{[\Pi, r]}_{\mathrm{ES},S}$ does still depend on $(c_1, c_2)$. In fact, we assumed above that the $c_i$ were close to 1 locally above $S$ -- where the meaning of ``close'' depended on the local data chosen -- but this is not needed in order to define the classes ${}_{c_1, c_2} z_{\Iw, M}^{[\Pi, r]}$ or to prove their norm relations, only to state simply their relation to the non-integral classes.
  \end{remark}
    
  \begin{proposition}
   This map satisfies the equivariance property
   \[ 
    \LE^{[\Pi, r]}_{\mathrm{ES},S}\Big(h\uphi, hv\Big) = \operatorname{Art}(\det h)^{-1} \cdot \LE^{[\Pi, r]}_{\mathrm{ES},S}(\uphi, v),
   \]
   for all $h \in H(\QQ_S)$, where we let $\Gal(\QQ^{\mathrm{ab}} / \QQ)$ act on $\operatorname{ES}(T_{\Pi_f}^*, \Sigma)$ via its natural map to $\Gal(\QQ(\zeta_{Mp^\infty}) / \QQ)$ for all $M \in \mathcal{R}$.
  \end{proposition}
 
  \begin{proof} This follows from the $H(\AA_f) \times G(\AA_f)$-equivariance of the Lemma--Eisenstein map. \end{proof}
  
  \begin{remark}
   \label{remark:localdata}
   Similarly, mapping $\uphi_S \otimes v_\alpha$ to $ \vol(W_S) \cdot z_{1, 0}^{[\Pi, q, r]}$ defines an $H(\QQ_S)$-invariant bilinear form on $\cS(\QQ_S^2, L)^{\otimes 2} \otimes \sigma_S[q]$. If we fix characters $\nu_1, \nu_2$ of $\ZZ_S^\times$ and restrict to Schwartz functions on which the centre of $\GL_2(\ZZ_S) \times \GL_2(\ZZ_S)$ acts via $\nu_1 \times \nu_2$, then this bilinear form is forced to factor through $\tau_S \otimes \sigma_S[q]$ for some irreducible principal series representation $\tau_S$ of $H(\QQ_S)$. Of course, the restriction is zero unless $(\nu_1 \nu_2)^{-1}$ coincides with the restriction to $\ZZ_S^\times$ of the central character of $\sigma_S$.
   
   We expect that $\dim \Hom_{H(\QQ_S)}(\tau_S \otimes \sigma_S[q], L)$ should have $L$-dimension $\le 1$. This follows from the Gan--Gross--Prasad conjecture for $\SO_4 \times \SO_5$ if $\nu_1 \nu_2$ is a square in the group of characters of $\ZZ_S^\times$, and should probably be true more generally, but we do not know a reference; let us assume this for the duration of this remark.
   
   If this dimension is 0, then the cohomology class $z_{1, 0}^{[\Pi, q, r]}$ is zero for every choice of the local data in the $(\nu_1, \nu_2)$ eigenspace. If this dimension is $1$ -- which is the case if $\Pi_\ell$ is generic for all $\ell \in S$ -- then the choice of local data only affects this class up to a scaling factor, which is essentially the local zeta integral of Piatetski-Shapiro appearing in \cite[\S 5.2]{lemma17}. We expect, but cannot prove, that if there exists a prime $\ell$ such that $\Pi_\ell$ is not generic, then the classes $z_{M, m}^{[\Pi, q, r]}$ are zero for all choices of the local data and the parameters $(q,r,M,m)$.
  \end{remark}
  
\section{Selmer groups and p-adic L-functions}
\label{sect:selmer}
 
 We conclude by showing that, if the above Euler system is non-zero, it gives bounds on Selmer groups.
 
 \subsection{Assumptions on $\Pi$}
 
  In this section, $\Pi$ will denote a non-endoscopic automorphic representation of $\GSp_4$, discrete-series at $\infty$ of some weight $(k_1, k_2)$, as before. We suppose that $p \ne 2$, and fix an embedding $E \into \Qpbar$, where $E$ is a number field as in Theorem \ref{thm:taylorweissauer}(3). We let $L$ be a sufficiently large finite extension of $\Qp$, with ring of integers $\cO_L$ and residue field $k_L$, such that $L$ contains the image of $E$ and $W_{\Pi_f}^*$ is definable over $L$ as a quotient of the cohomology of $Y_{G, \QQbar}$. 
  
  We also impose the following extra hypotheses:
  
  \begin{assumption}[``no exceptional zero'']
   None of the roots of the polynomial $P_p(X)$ are of the form $p^n \zeta$, with $n \in \ZZ$ and $\zeta$ a root of unity.
  \end{assumption}
  
  \begin{assumption}[``big image''] \
   \begin{enumerate}
    \item[(i)] The representation $T_{\Pi_f}^* \otimes k_L$ is irreducible as a $k_L[\Gal(\QQbar / \QQ(\zeta_{p^\infty})]$-module.
    \item[(ii)] There exists $\tau \in \Gal(\QQbar / \QQ(\zeta_{p^\infty}))$ such that $T_{\Pi_f}^* / (\tau - 1) T_{\Pi_f}^*$ is free of rank 1 over $\cO$.
   \end{enumerate}
   (This is precisely $\operatorname{Hyp}(K_\infty, T)$ of \cite{rubin00}.)
  \end{assumption}

  \begin{assumption}[``Siegel ordinarity'']
   $\Pi$ is good ordinary at $p$ in the sense of Definition \ref{def:goodordinary}.
  \end{assumption}
  \begin{remark}
   Note that the ``big image'' assumption is clearly satisfied if the image of $\Gal(\QQbar / \QQ)$ in $\operatorname{Aut} W_{\Pi_f}^*$ contains a conjugate of $\Sp_4(\Zp)$. This is expected to hold for all but finitely many $p$ if $\Pi$ is ``sufficiently general'' (i.e.~not a functorial lift from a proper subgroup of $\GSp_4$). However, the big image assumption is also satisfied in certain other cases, such as twisted Yoshida lifts of suitable Hilbert modular forms.
  \end{remark}
  
 \subsection{Ordinary submodules at $p$}  Recall that the ordinarity property implies that there is a unique reciprocal root $\alpha$ of $P_p(X)$ which is a $p$-adic unit.

  \begin{proposition}[{Urban; \cite[Corollaire 1(i)]{urban05}}]
   The representation $W_{\Pi_f}|_{G_{\Qp}}$ has a one-dimensional unramified subrepresentation on which geometric Frobenius acts as $\alpha$.
  \end{proposition}

  Equivalently, $W_{\Pi_f}^*$ has a decreasing filtration by $G_{\Qp}$-stable subspaces,
  \[ W_{\Pi_f}^* \supsetneq \sF^1 W_{\Pi_f}^*\supsetneq \sF^3 W_{\Pi_f}^* \supsetneq  \{ 0\},  \]
  with $\sF^i$ having codimension $i$, and the quotient $W_{\Pi_f}^* / \sF^1$ is unramified. We let $\sF^i T_{\Pi_f}^*$ be the intersection of $T_{\Pi_f}^*$ with $\sF^i W_{\Pi_f}^*$.

  \begin{proposition}
   For any $0 \le r \le b$, and any $M \in \mathcal{R}$, the image of the element ${}_{c_1, c_2} z_{\Iw, M}^{[\Pi, r]}$ in $H^1_{\Iw}(\Qp(\zeta_{Mp^\infty}), T_{\Pi_f}^* / \sF^1 )$ is zero.
  \end{proposition}

  \begin{proof}
   From the ordinarity of $\Pi_p$ it follows that $H^0\left(\Qp(\zeta_{Mp^\infty}), T_{\Pi_f}^* / \sF^1\right) = 0$, and hence that the natural map
   \[ 
    H^1_{\Iw}\left(\Qp(\zeta_{Mp^\infty}), T_{\Pi_f}^* / \sF^1\right) \to \varprojlim_m H^1_{\Iw}\left(\Qp(\zeta_{Mp^m}), W_{\Pi_f}^*(-q) / \sF^1\right) 
   \]
   is injective for any integer $q$. So it suffices to show that the image of ${}_{c_1, c_2} z_{M, m}^{[\Pi, q, r]}$ in $W_{\Pi_f}^*(-q) / \sF^1$ is zero for all $m$ (for any choice of $q$). However, if $q$ is chosen such that $0 \le q \le a$, then this element lands in the Bloch--Kato subspace $H^1_\f(\Qp(\zeta_{Mp^m}), W_{\Pi_f}^*(-q) / \sF^1)$, by Proposition \ref{prop:eltisSelmer}; and this subspace is zero, since $W_{\Pi_f}^*(-q) / \sF^1$ has all Hodge--Tate weights $\le 0$ (and is not the trivial representation).
  \end{proof}
  
  Let us write $\QQ_\infty = \QQ(\zeta_{p^\infty})$. We can use the submodule $\sF^1 T^*_{\Pi_f}$ to define Selmer groups $\widetilde H^i_{\Iw}(\QQ_\infty, T_{\Pi_f}^*)$, via Nekov\'a\v{r}'s formalism of Selmer complexes \cite{nekovar06}; cf.~\cite[\S 11.2]{KLZ1b}. These are finitely-generated $\Lambda$-modules, where $\Lambda$ denotes the Iwasawa algebra $\cO[[\Gamma]]$ of $\Gamma = \Gal(\QQ_\infty / \QQ) \cong \Zp^\times$. They admit the following somewhat concrete descriptions:
  \begin{itemize}
   \item $\widetilde H^i_{\Iw}(\QQ_\infty, T_{\Pi_f}^*) = 0$ unless $i = 1$ or $i = 2$.
   \item $\widetilde H^1_{\Iw}(\QQ_\infty, T_{\Pi_f}^*)$ is the kernel of the map
   \[ 
    H^1_{\Iw}(\QQ_\infty, T_{\Pi_f}^*) \rTo H^1_{\Iw}(\QQ_{p,\infty}, T_{\Pi_f}^* / \sF^1).
   \]
   \item Let $S_\infty$ denote the set of primes of $\QQ_\infty$ above $S$, and $A = T_{\Pi_f}(1) \otimes \Qp/\Zp$. If we define $\cS(\QQ_\infty, A)$ to be the $p$-torsion group 
   \[ \ker \left( H^1(\ZZ[\zeta_{p^\infty}, 1/pS], A)
  \rTo \bigoplus_{v \in S_\infty} H^1(\QQ_{\infty, v}, A) \oplus H^1(\QQ_{p, \infty}, A / \sF^3 A) \right), \] 
   then there is an exact sequence
   \[ 0 \to \cS(\QQ_\infty, A)^\vee \to \widetilde H^2_{\Iw}(\QQ_\infty, T_{\Pi_f}^*) \to H^2_{\Iw}(\QQ_{p,\infty}, \sF^1 T_{\Pi_f}^*)\]
   where the last module is a finite group (cf.~\cite[Proposition 11.2.8]{KLZ1b}).
  \end{itemize}
  
  Moreover, an Euler characteristic computation (using the fact that complex conjugation acts on $W_{\Pi_f}$ with two $+1$ eigenvalues and two $-1$ eigenvalues) shows that $\operatorname{rank}_{\Lambda} \widetilde H^1 - \operatorname{rank}_{\Lambda} \widetilde H^2 = 1$.
  
  \begin{theorem}
   Suppose that there exists an $r$ such that ${}_{c_1, c_2} z_{\Iw, 1}^{[\Pi, r]}$ is non-torsion. Then $\widetilde H^2_{\Iw}(\QQ_\infty, T_{\Pi_f}^*)$ is a torsion $\Lambda$-module, and we have the divisibility of characteristic ideals
   \[ \operatorname{char}_{\Lambda}\left( \widetilde H^2_{\Iw} \right) \mid \operatorname{char}_{\Lambda} \left( \frac{ \widetilde H^1_{\Iw} } { \Lambda \cdot {}_{c_1, c_2} z_{\Iw, 1}^{[\Pi, r]} } \right).\]
  \end{theorem}
   
  \begin{proof}
   This follows by applying the ``Euler system machine'' to the Euler system ${}_{c_1, c_2} z_{\Iw, M}^{[\Pi, r]}$. Compare \cite[Theorem 11.4.3]{KLZ1b}.
  \end{proof}
    
  \begin{corollary}
   \label{cor:BKSel}
   Suppose $q$ is an integer $\ge 0$, and assume that there is an $r \in \{0, \dots, b\}$ such that the cohomology class $z_{1, 0}^{[\Pi, q, r]} \in H^1(\QQ, W_{\Pi_f}^*(-q))$ is non-zero.
   \begin{itemize}
    \item If $\operatorname{loc}_p\left(z_{1, 0}^{[\Pi, q, r]}\right)$ lies in the subspace $H^1_\f(\Qp, W_{\Pi_f}^*(-q))$, then the Bloch--Kato Selmer group $H^1_\f\left( \QQ, W_{\Pi_f}^*(-q)\right)$ has dimension $1$ over $L$, and $ z_{1, 0}^{[\Pi, q, r]}$ is a basis of this space.
    \item If $\operatorname{loc}_p\left(z_{1, 0}^{[\Pi, q, r]}\right)$ does not lie in $H^1_\f(\Qp, W_{\Pi_f}^*(-q))$, then $H^1_\f\left( \QQ, W_{\Pi_f}^*(-q)\right)$ is zero.
   \end{itemize}
  \end{corollary}
  
  Note that the second case cannot occur if $0 \le q \le a$, by Proposition \ref{prop:eltisSelmer}. 
  
  \begin{proof}
   This follows from the previous theorem by descent; compare \cite[Proposition 11.5.1]{KLZ1b}.
  \end{proof}

 \subsection{The motivic p-adic L-function}

  We can also interpret the above results in terms of $p$-adic $L$-functions. For simplicity, we assume that $\Pi_p$ is ordinary for the Borel subgroup (not just for the Siegel parabolic) in the sense of \cite{urban05}. In this case, Corollary 1 of \emph{op.cit.} shows that there is a 2-dimensional $G_{\Qp}$-stable subspace $\sF^2 W_{\Pi_f}^*$, with $\sF^1 \supsetneq \sF^2 \supsetneq \sF^3$. Then the graded piece $\sF^1  W_{\Pi_f}^*/ \sF^2$ has Hodge--Tate weight $a + 1$, and Perrin-Riou's ``big logarithm'' map gives a canonical isomorphism
  \[ 
   \mathcal{L}: H^1_{\Iw}(\QQ_{p, \infty}, \sF^1  W_{\Pi_f}^*/ \sF^2) \rTo \Lambda_L(\Gamma) \otimes \DD_{\mathrm{cris}}(\sF^1  W_{\Pi_f}^*/ \sF^2).
   \]
   Composing $\mathcal{L}$ with evaluation at $\chi^q$, where $\chi$ is the cyclotomic character, interpolates the Bloch--Kato logarithm (for $q \le a$) or dual exponential (for $q \ge a+1$) of the image of $z$ in $H^1(\QQ_p, \sF^1  W_{\Pi_f}^*(-q)/ \sF^2)$.
   
   \begin{definition}
    We let ${}_{c_1, c_2} L_p^{\mot, r}(\Pi_f)$ be the image of ${}_{c_1, c_2} z_{\Iw, 1}^{[\Pi, r]}$ under the map $\mathcal{L}$.
   \end{definition}
   
   By choosing a basis of the 1-dimensional $L$-vector space $\DD_{\mathrm{cris}}(\sF^1  W_{\Pi_f}^*/ \sF^2)$, we may regard this as an element of $\Lambda_L(\Gamma)$, well-defined up to non-zero scalars. We call this measure the \textbf{motivic $p$-adic $L$-function}.
   
   \begin{theorem}
    Let $a + 1 \le q \le a + b + 1$, and suppose that ${}_{c_1, c_2} L_p^{\mot, r}(\Pi_f)$ is non-vanishing at $\chi^q$ for some $r$. Then $H^1_\f\left( \QQ, W_{\Pi_f}^*(-q)\right) = H^1_\f\left( \QQ, W_{\Pi_f}(1 + q)\right) = 0$.
   \end{theorem}
   
   \begin{proof}
    By construction, if the motivic $L$-function does not vanish, then the image of ${}_{c_1, c_2} z_{\Iw, 1}^{[\Pi, r]}$ in $H^1(\Qp, \sF^1  W_{\Pi_f}^*(-q)/ \sF^2)$ is non-zero. The hypotheses on $q$ imply that this subquotient has vanishing $H^1_\f$, so we conclude that ${}_{c_1, c_2} z_{\Iw, 1}^{[\Pi, r]}$ cannot lie in $H^1_\f$ locally at $p$. By Corollary \ref{cor:BKSel} it follows that the global $H^1_\f$ is zero.
   \end{proof}
   
   \begin{remark}
    Note that $a + 1 \le q \le a + b + 1$ is precisely the range such that $L(\Pi, 1+q - \tfrac{w}{2})$ is a critical value of the spin $L$-function. We conjecture that, for an appropriate $r$ and suitably chosen test data $\uphi, v_\alpha$, the value at $\chi^q$ of ${}_{c_1, c_2} L_p^{\mot, r}(\Pi_f)$ should be non-zero if and only if the critical $L$-value does not vanish.
   \end{remark}

\providecommand{\bysame}{\leavevmode\hbox to3em{\hrulefill}\thinspace}
\providecommand{\MR}[1]{}
\renewcommand{\MR}[1]{%
 MR \href{http://www.ams.org/mathscinet-getitem?mr=#1}{#1}.
}
\newcommand{\articlehref}[2]{\href{#1}{#2}}

\end{document}